\documentclass[a4paper,11pt]{article}
\usepackage{amssymb}
\usepackage{amsmath}
\usepackage{amsfonts}
\usepackage{amscd}
\usepackage{verbatim}
\usepackage{graphicx}
\usepackage{geometry}
\usepackage{lipsum}

\newcommand*\Abstract[1]{   \begingroup\noindent\leftskip=.7cm
   \rightskip\leftskip
\small\textbf{Abstract}\quad #1  \par\vspace*{1mm}\endgroup}
\newcommand*\Keywords[1]{  \begingroup\noindent\leftskip=.7cm
  \rightskip\leftskip
  {\hangafter=1\hangindent=2.1cm\noindent\small\textbf{Key words}\quad #1  \par\vspace*{1mm}}\endgroup}
\newcommand*\MRSubClass[1]{  \begingroup\noindent\leftskip=.7cm
  \rightskip\leftskip
  \small\textbf{2010 MR Subject Classification}\quad #1  \par\vspace*{1mm}\endgroup}
\newtheorem{theorem}{Theorem}[section]
\newtheorem{corollary}{Corollary}[theorem]
\newtheorem{lemma}[theorem]{Lemma}
\input{tcilatex}
\numberwithin{equation}{section}

\input{tcilatex}

\makeatletter

\renewcommand{\maketitle}{\bgroup\setlength{\parindent}{0pt}
  \begin{center}
   \textbf{\@title}
  \@author
\end{center}
  \@thanks
\egroup
}
\makeatother

\begin{document}

\title{\fontsize{16}{6}\selectfont Unified Bessel, Modified Bessel,
Spherical Bessel and Bessel-Clifford Functions }
\author{\normalfont \\
%EndAName
\bigskip \fontsize{10}{6}\selectfont \textit{\ Banu Y\i lmaz YA\c{S}AR and Mehmet Ali \"{O}ZARSLAN\\
{\tiny Eastern Mediterranean University, Department of Mathematics
Gazimagusa, TRNC, Mersin 10, Turkey}} \\
\textit{{\tiny E-mail: banu.yilmaz@emu.edu.tr; \hspace{0.1mm}
mehmetali.ozarslan@emu.edu.tr}}  }
\maketitle

\Abstract{
In the present paper, unification of Bessel, modified Bessel, spherical
Bessel and Bessel-Clifford functions via the generalized Pochhammer symbol [
Srivastava HM, \c{C}etinkaya A, K{\i }ymaz O. A certain generalized
Pochhammer symbol and its applications to hypergeometric functions. Applied
Mathematics and Computation, 2014, \textbf{226} : 484-491] is defined. Several
potentially useful properties of the unified family such as generating
function, integral representation, Laplace transform and Mellin transform
are obtained. Besides, the unified Bessel, modified Bessel, spherical Bessel
and Bessel-Clifford functions are given as a series of Bessel functions.
Furthermore, the derivatives, recurrence relations and partial differential
equation of the so-called unified family are found. Moreover, the Mellin
transform of the products of the unified Bessel functions are obtained.
Besides, a three-fold integral representation is given for unified Bessel
function. Some of the results which are obtained in this paper are new and
some of them coincide with the known results in special cases. 
}


\MRSubClass{33C10; 33C05; 44A10}

\section{\protect\bigskip Introduction}

Bessel function first arises in the investigation of a physical problem in
Daniel Bernoulli's analysis of the small oscillations of a uniform heavy
flexible chain \cite{A.G.M, B.G}. It appears when finding separable
solutions to Laplace's equation and the Helmholtz equation in cylindrical or
spherical coordinates. Whereas Laplace's equation governs problems in heat
conduction, in the distribution of potential in an electrostatic field and
in hydrodynamics in the irrotational motion of an incompressible fluid \cite%
{A.G.M, J.B, J.D, B.G, H.L},~\ Helmholtz equation governs problems in
acoustic and electromagnetic wave propagation \cite{J.D, D.J, L.R}. Besides,
Bessel function and modified Bessel function play an important role in the
analysis of microwave and optical transmission in waveguides, including
coaxial and fiber \cite{N.K.B, S.G.K, J.C.S}. Also, Bessel function appears
in the inverse problem in wave propagation with applications in medicine,
astronomy and acoustic imaging \cite{D.C.K}. On the other hand, spherical
Bessel function arises in all problems in three dimensions with spherical
symmetry involving the scattering of electromagnetic radiation \cite{L.B.J,
J.D, E.K}. In quantum mechanics, it comes out in the solution of the Schr%
\"{o}dinger wave equation for a particle in a central potential \cite{A.M.S}%
. Therefore, Bessel function is crucially important for many problems of
wave propagation and static potentials. When in solving problems in
cylindrical coordinate system, one obtains Bessel function of an integer
order, in spherical problems one obtains half-integer orders such as
electromagnetic waves in a cylindrical waveguide, pressure amplitudes of
inviscid rotational flows, heat conduction in a cylindrical object, modes of
vibration of a thin circular (or annular) artificial membrane, diffusion
problems on a lattice, solutions to the radial Schr\"{o}dinger equation (in
a spherical and cylindrical coordinates) for a free particle, solving for
patterns of a acoustical radiation, frequency-dependent friction in circular
pipelines, dynamics of floating bodies, angular resolution. Also, it appears
in other problems such as signal processing , frequency modulation
synthesis, Kaiser window or Bessel filter. While \ Bessel function has wide
range of applications in mathematical physics such as acoustics, radio
physics, hydrodynamics, atomic and nuclear physics,~Bessel-Clifford function
comes out asymptotic expressions for the Dirac delta function \cite{F.N}.
Regular and irregular Coulomb wave functions are expressed in terms of
Bessel-Clifford function \cite{A, A2}. Also, it has an applications in
quantum mechanics. Because of these facts, several researchers have studied
on some extensions and generalizations of Bessel \ and Bessel-Clifford
functions ~\cite{A.K, C, D.M, H.E, F.S, L.G, L, H.A.D, M, N.V.O}. In that
point, special functions play a remarkably important role. Pochhammer symbol
or shifted factorial function is one of the most important and useful
function in the theory of special functions. Familiar definition is given by%
\begin{align}
(\alpha )_{n}& =\alpha (\alpha +1)(\alpha +2)...(\alpha +n-1) \\
(\alpha )_{0}& =1,~\alpha \neq 0.  \notag
\end{align}%
Also, factorial function $(\alpha )_{n}~$can be expressed in terms of a
ratio of gamma functions as%
\begin{equation}
(\alpha )_{n}:=\frac{\Gamma (\alpha +n)}{\Gamma (\alpha )}
\end{equation}%
and using the above definition, one can get

\begin{equation}
(\alpha )_{2n}=2^{2n}(\dfrac{\alpha }{2})_{n}(\dfrac{\alpha +1}{2})_{n}
\end{equation}%
where 
\begin{equation}
\Gamma (\alpha )=\int\limits_{0}^{\infty }t^{\alpha -1}e^{-t}dt,~\func{Re}%
(\alpha )>0
\end{equation}%
is the usual Euler's Gamma function. Another important function is the
generalized hypergeometric function which is defined by means of Pochhammer
symbols as 
\begin{equation}
_{p}F_{q}(a_{1,}a_{2,}...,a_{p};~b_{1,}b_{2},...,~b_{q};z~)=\sum%
\limits_{k=0}^{\infty }\frac{(a_{1})_{k}(a_{2})_{k}...(a_{p})_{k}}{%
(b_{1})_{k}(b_{2})_{k}...(b_{q})_{k}}\frac{z^{k}}{k!},~(p\leq q).
\end{equation}%
Note that, hypergeometric series given by (1.5)\ converges absolutely for $%
p\leq q$~\cite{N.M.S}.~Substituting $p=0$~and $q=1,~a_{2}=1+\nu ~$~and
replacing $z$~with $-\frac{z^{2}\text{~}}{4},~$then it is reduced to%
\begin{equation*}
_{0}F_{1}(-;1+\nu ;-\frac{z^{2}}{4})=\sum\limits_{k=0}^{\infty }\frac{1}{%
(1+\nu )_{k}}\frac{(-\frac{z^{2}}{4})}{k!}.
\end{equation*}%
The case $p=1$~and $q=1$, it is reduced to the confluent hypergeometric
function which is 
\begin{equation*}
_{1}F_{1}(a;b;z)=\sum\limits_{k=0}^{\infty }\frac{(a)_{k}}{(b)_{k}}\frac{%
z^{k}}{k!}.
\end{equation*}%
Chaudhry et al. \cite{C.Q.S} extended the confluent hypergeometric function
as 
\begin{align*}
F_{p}(a,b;c;z)& :=\sum\limits_{n=0}^{\infty }\frac{B_{p}(b+n,c-b)}{B(b,c-b)}%
(a)_{n}\frac{z^{n}}{n!} \\
(\left\vert z\right\vert & <1,~p\geq 0;~\func{Re}(c)>\func{Re}(b)>0)
\end{align*}%
where $(a)_{n}$ denotes the Pochhammer symbol given by (1.1) and 
\begin{align*}
B_{p}(x,y)& :=\int\limits_{0}^{1}t^{x-1}(1-t)^{y-1}e^{-\frac{p}{t(1-t)}}dt,
\\
(\func{Re}(p)& >0,~\func{Re}(x)>0,~\func{Re}(y)>0)
\end{align*}%
is the extended Euler's Beta function. In the case $p=0,$ extended Beta
function is reduced to usual Beta function which is defined by%
\begin{align*}
B(x,y)& =\int\limits_{0}^{1}t^{x-1}(1-t)^{y-1}dt, \\
(\func{Re}(x)& >0,~\func{Re}(y)>0).
\end{align*}%
An extension of the generalized hypergeometric function $_{r}F_{s}$~of $~r$%
~numerator parameters $a_{1},...,a_{r}$ and $s$ denominator parameters $%
b_{1},...,b_{s}$~was defined by Srivastava et al. in \cite{S.A.O} as 
\begin{equation*}
_{r}F_{s}[%
\begin{tabular}{l}
$(a_{1},\rho ),~a_{2},...,a_{r};$ \\ 
$\ \ \ \ \ \ \ \ \ \ \ \ b_{1},...,b_{s};$%
\end{tabular}%
\ \ \ z]:=\sum\limits_{n=0}^{\infty }\frac{(a_{1},\rho
)_{n}(a_{2})_{n}...(a_{r})_{n}}{(b_{1})_{n}(b_{2})_{n}...(b_{s})_{n}}\frac{%
z^{n}}{n!}
\end{equation*}%
where 
\begin{equation*}
a_{j}\in 
%TCIMACRO{\U{2102} }%
%BeginExpansion
\mathbb{C}
%EndExpansion
~(j=1,...,r)~\text{and }~b_{j}\in 
%TCIMACRO{\U{2102} }%
%BeginExpansion
\mathbb{C}
%EndExpansion
/%
%TCIMACRO{\U{2124} }%
%BeginExpansion
\mathbb{Z}
%EndExpansion
_{0}^{-}~(j=1,...,s),~%
%TCIMACRO{\U{2124} }%
%BeginExpansion
\mathbb{Z}
%EndExpansion
_{0}^{-}:=\{0,-1,-2,...\}.
\end{equation*}%
In particular, the corresponding extension of the confluent hypergeometric
function $\ _{1}F_{1}~$is given by 
\begin{equation*}
_{1}F_{1}[(a,\rho );c;z]:=\sum\limits_{n=0}^{\infty }\frac{(a,\rho )_{n}}{%
(c)_{n}}\frac{z^{n}}{n!}.
\end{equation*}%
Note that, in our main Theorems the above extension of the hypergeometric
function is used. Another kind of generalized and extended hypergeometric
function was introduced in \cite{R.S} as%
\begin{equation*}
_{u}F_{\nu }[%
\begin{tabular}{l}
$(a_{0};p,\{K_{l}\}_{l\in 
%TCIMACRO{\U{2115} }%
%BeginExpansion
\mathbb{N}
%EndExpansion
_{0}}),~a_{2},...,a_{u};$ \\ 
$\ \ \ \ \ \ \ \ \ \ \ \ \ \ \ \ \ \ \ \ \ \ \ \ \ \ b_{1},...,b_{\nu };$%
\end{tabular}%
\ \ \ z]:=\sum\limits_{n=0}^{\infty }\frac{(a_{0};p,\{K_{l}\}_{l\in 
%TCIMACRO{\U{2115} }%
%BeginExpansion
\mathbb{N}
%EndExpansion
_{0}})_{n}~(a_{2})_{n}...(a_{u})_{n}}{(b_{1})_{n}(b_{2})_{n}...(b_{\nu })_{n}%
}\frac{z^{n}}{n!}
\end{equation*}%
where 
\begin{equation*}
(\lambda ;p,\{K_{l}\}_{l\in 
%TCIMACRO{\U{2115} }%
%BeginExpansion
\mathbb{N}
%EndExpansion
_{0}})_{\nu }:=\frac{\Gamma _{p}^{(\{K_{l}\}_{l\in 
%TCIMACRO{\U{2115} }%
%BeginExpansion
\mathbb{N}
%EndExpansion
_{0}})}(\lambda +\nu )}{\Gamma _{p}^{(\{K_{l}\}_{l\in 
%TCIMACRO{\U{2115} }%
%BeginExpansion
\mathbb{N}
%EndExpansion
_{0}})}(\lambda )},\lambda ,\nu \in 
%TCIMACRO{\U{2102}}%
%BeginExpansion
\mathbb{C}%
%EndExpansion
\end{equation*}%
is the generalized and the extended Pochhammer symbol of $(\lambda )_{n}.~$%
Here, $\Gamma _{p}^{(\{K_{l}\}_{l\in 
%TCIMACRO{\U{2115} }%
%BeginExpansion
\mathbb{N}
%EndExpansion
_{0}})}(z)$ is the extended Gamma function which is given by%
\begin{align*}
\Gamma _{p}^{(\{K_{l}\}_{l\in 
%TCIMACRO{\U{2115} }%
%BeginExpansion
\mathbb{N}
%EndExpansion
_{0}})}(z)& :=\int\limits_{0}^{\infty }t^{z-1}\circleddash (\{K_{l}\}_{l\in 
%TCIMACRO{\U{2115} }%
%BeginExpansion
\mathbb{N}
%EndExpansion
_{0}};-t-\frac{p}{t})dt \\
(\func{Re}(z)& >0,~\func{Re}(p)\geq 0)
\end{align*}%
where $\circleddash (\{K_{l}\}_{l\in 
%TCIMACRO{\U{2115} }%
%BeginExpansion
\mathbb{N}
%EndExpansion
_{0}};z)~$ is given by \cite{S.P.C}%
\begin{equation*}
\circleddash (\{K_{l}\}_{l\in 
%TCIMACRO{\U{2115} }%
%BeginExpansion
\mathbb{N}
%EndExpansion
_{0}};z):=\{%
\begin{tabular}{l}
$\sum\limits_{l=0}^{\infty }K_{l}\frac{z^{l}}{l!}~~(|z|<R;~R>0,~K_{0}:=1)$
\\ 
$M_{0}~z^{\omega }\exp (z)[1+O(\frac{1}{|z|})]~(|z|\rightarrow \infty
;~M_{0}>0;~w\in 
%TCIMACRO{\U{2102} }%
%BeginExpansion
\mathbb{C}
%EndExpansion
).$%
\end{tabular}%
\ \ \ \ \ \ \ \ \ \ .
\end{equation*}%
Besides, multiple Gaussian hypergeometric function was studied in \cite{H.K}%
. Moreover, using the extended Pochhammer symbols, some properties of the
generalized and extended hypergeometric polynomials were obtained in \cite%
{R2, R.N}. On the other hand, generalization of Gamma, Beta and
hypergeometric functions were introduced and studied in \cite{E.M.A}%
.~Considering the generalized Beta function, generating functions for the
Gauss hypergeometric functions were introduced in \cite{S.P.J}. Extended
incomplete gamma function was obtained in \cite{C.S}. Finally, using the
generalized Pochhammer function ,which was introduced by Srivastava et al.%
\cite{S.A.O},~generalized Mittag-Leffler function was introduced and some
properties were presented in \cite{B.Y}.

The main idea of the present paper is to define unification of Bessel,
modified Bessel, spherical Bessel and Bessel-Clifford functions by means of
the generalized Pochhammer symbol \cite{S.A.O}. Before proceeding, some
facts related with the Bessel, modified Bessel, spherical Bessel and
Bessel-Clifford functions are presented. Bessel \ and modified Bessel
functions can be expressed in terms of the $_{0}F_{1}(-;1+\nu;-\frac{z^{2}}{4%
})$~and $_{1}F_{1}(a;b;z)$~functions as%
\begin{align*}
J_{\nu}(z) & =\frac{(\frac{z}{2})^{\nu}}{\Gamma(\nu+1)}\sum\limits_{k=0}^{%
\infty}\frac{(-1)^{k}}{(\nu+1)_{k}}\frac{(\frac{z^{2}}{4})^{k}}{k!}=\frac{(%
\frac{z}{2})^{\nu}}{\Gamma(\nu+1)}~_{0}F_{1}(-;1+\nu;-\frac{z^{2}}{4}), \\
I_{\nu}(z) & =\frac{(\frac{z}{2})^{\nu}}{\Gamma(\nu+1)}\sum\limits_{k=0}^{%
\infty}\frac{1}{(\nu+1)_{k}}\frac{(\frac{z^{2}}{4})^{k}}{k!}=\frac{(\frac {z%
}{2})^{\nu}}{\Gamma(\nu+1)}~_{0}F_{1}(-;1+\nu;\frac{z^{2}}{4}),
\end{align*}
and 
\begin{align*}
J_{\nu}(z) & =\frac{(\frac{z}{2})^{\nu}}{\Gamma(\nu+1)}e^{-iz}~_{1}F_{1}(\nu+%
\frac{1}{2};2\nu+1;2iz), \\
I_{\nu}(z) & =\frac{(\frac{z}{2})^{\nu}}{\Gamma(\nu+1)}e^{-z}~_{1}F_{1}(\nu+%
\frac{1}{2};2\nu+1;2z).
\end{align*}
Recurrence relations satisfied by usual Bessel and modified Bessel functions
are given by \cite{R}%
\begin{align*}
\frac{2\nu}{z}J_{\nu}(z) & =J_{\nu-1}(z)+J_{\nu+1}(z), \\
2\frac{d}{dz}J_{\nu}(z) & =J_{\nu-1}(z)-J_{\nu+1}(z), \\
\frac{2\nu}{z}I_{\nu}(z) & =I_{\nu-1}(z)-I_{\nu+1}(z), \\
2\frac{d}{dz}I_{\nu}(z) & =I_{\nu-1}(z)+I_{\nu+1}(z).
\end{align*}
Besides, spherical Bessel function of the first kind is defined by means of
the Bessel function as follows%
\begin{equation*}
j_{\nu}(z):=\sqrt{\frac{\pi}{2z}}J_{\nu+\frac{1}{2}}(z),~\func{Re}(\nu)>-%
\frac{3}{2}.
\end{equation*}
The following recurrence relations are satisfied by the spherical Bessel
function \cite{G.H}%
\begin{align*}
j_{\nu-1}(z)+j_{\nu+1}(z) & =\frac{2\nu+1}{z}j_{\nu}(z), \\
\nu j_{\nu-1}(z)-(\nu+1)~j_{\nu+1}(z) & =(2\nu+1)\frac{d}{dz}j_{\nu}(z), \\
\frac{d}{dz}[z^{\nu+1}j_{\nu}(z)] & =z^{\nu+1}j_{\nu-1}(z), \\
\frac{d}{dz}[z^{-\nu}j_{\nu}(z)] & =-z^{-\nu}j_{\nu+1}(z), \\
(\nu-1)~j_{\nu-1}(z)-(\nu+2)~j_{\nu+1}(z) & =z(\frac{d}{dz}j_{\nu -1}(z)+%
\frac{d}{dz}j_{\nu+1}(z)).
\end{align*}
Moreover, Bessel-Clifford function of the first kind is defined by means of $%
_{0}F_{1}(-;n+1;z)$ as 
\begin{equation*}
C_{\nu}(z):=\frac{1}{\Gamma(n+1)}~_{0}F_{1}(-;n+1;z)
\end{equation*}
which is a particular case of Wright function 
\begin{equation*}
\Phi(\rho,\beta;z)=\sum\limits_{k=0}^{\infty}\frac{z^{k}}{k!\Gamma(\rho
k+\beta)},~\rho>-1~\text{and~}\beta\in%
%TCIMACRO{\U{2102} }%
%BeginExpansion
\mathbb{C}
%EndExpansion
.
\end{equation*}
The connection between Bessel-Clifford function and modified \ Bessel
function is given by \cite{D.G.M}%
\begin{equation*}
C_{\nu}(z)=z^{-\frac{\nu}{2}}I_{\nu}(2\sqrt{z}),~\func{Re}(\nu)>0.
\end{equation*}
Recurrence relations satisfied by Bessel-Clifford function are given by \cite%
{A.G}%
\begin{align*}
\frac{d}{dz}C_{\nu}(z) & =C_{\nu+1}(z), \\
zC_{\nu+2}(z)+(\nu+1)C_{\nu+1}(z) & =C_{\nu}(z), \\
z(2\nu+4)\frac{d}{dz}C_{\nu+1}(z)+2z^{2}\frac{d^{2}}{dz^{2}}C_{\nu+1}(z) &
=2z\frac{d^{2}}{dz^{2}}C_{\nu-1}(z), \\
(\nu+1)z\frac{d}{dz}C_{\nu+1}(z)+(\nu+1)^{2}C_{\nu+1}(z) & =(\nu+1)\frac {d}{%
dz}C_{\nu-1}(z).
\end{align*}

Recently, Srivastava et al. \cite{S.A.O} generalized the Pochhammer function
as 
\begin{equation*}
(\lambda ;\rho )_{\nu }:=\{%
\begin{tabular}{l}
$\dfrac{\Gamma _{\rho }(\lambda +\nu )}{\Gamma (\lambda )}~(\func{Re}(\rho
)>0;~\lambda ,~\nu \in 
%TCIMACRO{\U{2102} }%
%BeginExpansion
\mathbb{C}
%EndExpansion
)$ \\ 
$(\lambda )_{\nu }~\ \ \ \ \ ~\ \ ~~\ (\rho =0;~\lambda ,~\nu \in 
%TCIMACRO{\U{2102} }%
%BeginExpansion
\mathbb{C}
%EndExpansion
)$%
\end{tabular}%
\ \ \ \ \ \ \ \ \ 
\end{equation*}%
where \ $\Gamma _{\rho }$~\ is the extended Gamma function, which was
introduced by Chaudhry and Zubair \cite{C.Z} as 
\begin{equation*}
\Gamma _{\rho }(x)=\int\limits_{0}^{\infty }t^{x-1}e^{-t-\frac{\rho }{t}}dt,~%
\func{Re}(\rho )>0
\end{equation*}%
and hence $(\lambda ;\rho )_{\nu }$~\ is \ defined by 
\begin{align}
(\lambda ;\rho )_{\nu }& =\frac{1}{\Gamma (\lambda )}\int\limits_{0}^{\infty
}t^{\lambda +\nu -1}e^{-t-\dfrac{\rho }{t}}dt \\
(\func{Re}(\nu )& >0;~\func{Re}(\lambda +\nu )>0~\text{when }\rho =0). 
\notag
\end{align}%
It was proved that \cite{S.A.O} 
\begin{equation}
(\lambda ;\rho )_{\nu +\mu }=(\lambda )_{\nu }(\lambda +\nu ;\rho )_{\mu
},~\lambda ,~\mu ,~\nu \in 
%TCIMACRO{\U{2102} }%
%BeginExpansion
\mathbb{C}
%EndExpansion
.
\end{equation}%
\ 

In the light of these definitions and generalizations, unification of four
parameter Bessel function is defined by \ 
\begin{align}
G_{\nu }^{(b,c)}(z;\rho )& :=\sum\limits_{k=0}^{\infty }\frac{%
(-b)^{k}(c;\rho )_{2k+\nu }}{\Gamma (\nu +k+1)~\Gamma (\nu +2k+1)}\frac{(%
\frac{z}{2})^{2k+\nu }}{k!}, \\
(z,c,\nu & \in 
%TCIMACRO{\U{2102} }%
%BeginExpansion
\mathbb{C}
%EndExpansion
,~\func{Re}(\nu )>-1,~\func{Re}(\rho )>0).  \notag
\end{align}

The case $b=1$ and $b=-1,~$generalized three parameter Bessel function of
the first kind $J_{\nu}^{(c)}(z;\rho)~$and generalized modified three
parameter Bessel function of the first kind $I_{\nu}^{(c)}(z;\rho)$ are
introduced by 
\begin{align*}
J_{\nu}^{(c)}(z;\rho) & :=\sum\limits_{k=0}^{\infty}\frac{%
(-1)^{k}(c;\rho)_{2k+\nu}}{\Gamma(\nu+k+1)\Gamma(\nu+2k+1)}\frac{(\frac{z}{2}%
)^{2k+\nu}}{k!}, \\
I_{\nu}^{(c)}(z;\rho) & :=\sum\limits_{k=0}^{\infty}\frac{(c;\rho)_{2k+\nu}}{%
\Gamma(\nu+k+1)\Gamma(\nu+2k+1)}\frac{(\frac{z}{2})^{2k+\nu}}{k!}, \\
(z,c,\nu & \in%
%TCIMACRO{\U{2102} }%
%BeginExpansion
\mathbb{C}
%EndExpansion
,~\func{Re}(\nu)>-1,~\func{Re}(\rho)>0)
\end{align*}
respectively. Letting $~c=1$~and $\rho=0,$~$J_{\nu}^{(c)}(z;\rho)$~is
reduced to usual Bessel function of the first kind $J_{\nu}(z)$~and ~$%
I_{\nu}^{(c)}(z;\rho)$~is reduced to usual modified Bessel function of the
first kind $I_{\nu}(z)$. Furthermore, generalized four parameter spherical
Bessel function is introduced by%
\begin{align*}
g_{\nu}^{(b,c)}(z;\rho) & :=\sqrt{\frac{\pi}{2z}~}~G_{\nu+\frac{1}{2}}^{(b,c-%
\frac{1}{2})}(z;\rho), \\
(\func{Re}(\nu) & >-\frac{3}{2},~\func{Re}(c)>\frac{1}{2},~\func{Re}%
(\rho)>0~).
\end{align*}
The case $b=1,$ $c=\frac{3}{2}$ and $\rho=0,$ it is reduced to usual
spherical Bessel function of the first kind $j_{\nu}(z).~$Moreover,
generalized four parameter Bessel-Clifford function of the first kind is
defined by 
\begin{align*}
C_{\nu}^{(b,\lambda)}(z;\rho) & :=z^{-\frac{\nu}{2}}G_{\nu}^{(b,\lambda )}(2%
\sqrt{z};\rho), \\
(\func{Re}(\nu) & >-1,~\func{Re}(\lambda )>0,~\func{Re}(\rho)>0).
\end{align*}
The case $b=-1,$ $\lambda=1$~and $\rho=0,$ $C_{\nu}^{(b,\lambda)}(z;\rho)$
is reduced to usual Bessel-Clifford function of the first kind $C_{\nu}(z).$

The organization of the paper is as follows: In section 2, generating
function, integral representation, Laplace transform and Mellin transform
involving the unified four parameter Bessel function are obtained. Moreover,
the expansion of the unified four parameter Bessel function in terms of a
series of usual Bessel functions is presented. In section 3, derivative
properties, recurrence relation and partial differential equation of the
unified four parameter Bessel function are found. In Section 4, the Mellin
transforms involving the products of the unified four parameter Bessel
function are obtained. In Section 5, a three-fold integral representation
formula for the unified four parameter Bessel function is given. In Section
6, generalized four parameter spherical Bessel and Bessel-Clifford functions
are introduced by means of the unified four parameter Bessel function and
some properties are obtained such as generating function, integral
representation, Laplace transform, Mellin transform, series in terms of
usual Bessel functions, recurrence relation and partial differential
equation. Some of the corresponding results of the mentioned Theorems are
new and some of them coincide with the usual cases. Finally, the special
cases of the unified four parameter Bessel function are given as a table and
\ the graphics of the generalized two parameter Bessel and the spherical
Bessel functions are drawn for some special cases.

\section{Unified Four Parameter Bessel Function}

In the following Lemma, the relation between $\ G_{-\nu}^{(b,c)}(z;\rho)$%
~and $G_{\nu}^{(b,c)}(z;\rho)$ is obtained$:$

\begin{lemma}
Let $~\nu \in 
%TCIMACRO{\U{2124} }%
%BeginExpansion
\mathbb{Z}
%EndExpansion
.~$Then the following relation is satisfied by the unified four parameter
Bessel function%
\begin{equation}
G_{-\nu }^{(b,c)}(z;\rho )=(-b)^{\nu }G_{\nu }^{(b,c)}(z;\rho ).
\end{equation}
\end{lemma}

\begin{proof}
It is clear that by (1.8), we have 
\begin{equation*}
G_{\nu}^{(b,c)}(z;\rho)=\sum\limits_{k=0}^{\infty}\frac{(-b)^{k}(c;%
\rho)_{2k+\nu}}{\Gamma(\nu+k+1)\Gamma(\nu+2k+1)}\frac{(\frac{z}{2})^{2k+\nu}%
}{k!}.
\end{equation*}
Substituting $-\nu$ instead of $\nu$ yields%
\begin{equation*}
G_{-\nu}^{(b,c)}(z;\rho)=\sum\limits_{k=\nu}^{\infty}\frac{%
(-b)^{k}(c;\rho)_{2k-\nu}}{\Gamma(-\nu+k+1)~\Gamma(-\nu+2k+1)}\frac{(\frac{z%
}{2})^{2k-\nu}}{k!}.
\end{equation*}
Taking $~\nu+k$ ~for $\ k,$ we get%
\begin{align*}
G_{-\nu}^{(b,c)}(z;\rho) & =(-b)^{\nu}\sum\limits_{k=0}^{\infty}\frac{%
(-b)^{k}(c;\rho)_{2k+\nu}}{\Gamma(\nu+2k+1)~\Gamma(\nu+k+1)}\frac {(\frac{z}{%
2})^{2k+\nu}}{k!}, \\
& =(-b)^{\nu}G_{\nu}^{(b,c)}(z;\rho),
\end{align*}
which completes the proof.
\end{proof}

Taking $b=1$~and $b=-1~$~in Lemma 2.1, the relations of the generalized
three parameter Bessel and modified Bessel functions of the first kind \ are
obtained, respectively:

\begin{corollary}
Let $~\nu\in%
%TCIMACRO{\U{2124} }%
%BeginExpansion
\mathbb{Z}
%EndExpansion
.~$The following relations are satisfied by the generalized three parameter
Bessel and modified Bessel functions of the first kind%
\begin{align*}
J_{-\nu}^{(c)}(z;\rho) & =(-1)^{\nu}J_{\nu}^{(c)}(z;\rho), \\
I_{-\nu}^{(c)}(z;\rho) & =I_{\nu}^{(c)}(z;\rho).
\end{align*}
\end{corollary}

Letting $c=1$ and $\rho =0$~in Corollary 2.1.1, the relations of the usual
Bessel functions are given as follows:

\begin{corollary}
\cite{R} \ Let $~\nu\in%
%TCIMACRO{\U{2124} }%
%BeginExpansion
\mathbb{Z}
%EndExpansion
.~$The following relations are satisfied by the usual Bessel and modified
Bessel functions%
\begin{align*}
J_{-\nu}(z) & =(-1)^{\nu}J_{\nu}(z), \\
I_{-\nu}(z) & =I_{\nu}(z).
\end{align*}
\end{corollary}

In the following theorem, the generating function of the unified four
parameter Bessel function is given in terms of the generalized confluent
hypergeometric function:

\begin{theorem}
For $t\neq0$ and for all finite $z,~n\in%
%TCIMACRO{\U{2124} }%
%BeginExpansion
\mathbb{Z}
%EndExpansion
,$ we have%
\begin{equation*}
_{1}F_{1}((c;\rho),1;(t-\frac{b}{t})\frac{z}{2})=\sum\limits_{n=-\infty
}^{\infty}G_{n}^{(b,c)}(z;\rho)t^{n}.
\end{equation*}
\end{theorem}

\begin{proof}
It is clear that we have 
\begin{equation*}
\sum\limits_{n=-\infty}^{\infty}G_{n}^{(b,c)}(z;\rho)t^{n}=\sum
\limits_{n=-\infty}^{-1}G_{n}^{(b,c)}(z;\rho)t^{n}+\sum\limits_{n=0}^{\infty
}G_{n}^{(b,c)}(z;\rho)t^{n}.
\end{equation*}
Taking $-n-1$~instead of $n$~in the first summation of the right hand side,
we have%
\begin{equation*}
\sum\limits_{n=-\infty}^{\infty}G_{n}^{(b,c)}(z;\rho)t^{n}=\sum%
\limits_{n=0}^{\infty}G_{-n-1}^{(b,c)}(z;\rho)t^{-n-1}+\sum\limits_{n=0}^{%
\infty}G_{n}^{(b,c)}(z;\rho)t^{n}.
\end{equation*}
Now, using Lemma 2.1, we get%
\begin{equation*}
\sum\limits_{n=-\infty}^{\infty}G_{n}^{(b,c)}(z;\rho)t^{n}=\sum%
\limits_{n=0}^{\infty}(-b)^{n+1}G_{n+1}^{(b,c)}(z;\rho)t^{-n-1}+\sum%
\limits_{n=0}^{\infty }G_{n}^{(b,c)}(z;\rho)t^{n}.
\end{equation*}
Plugging the series definitions of the unified Bessel function into right
hand side, we have%
\begin{align*}
\sum\limits_{n=-\infty}^{\infty}G_{n}^{(b,c)}(z;\rho)t^{n} & =\sum
\limits_{n=0}^{\infty}\sum\limits_{k=0}^{\infty}\frac{(-b)^{n+1}(-b)^{k}(c;%
\rho)_{2k+n+1}}{(n+k+1)!k!(n+2k+1)!}(\frac{z}{2})^{2k+n+1}t^{-n-1} \\
& +\sum\limits_{n=0}^{\infty}\sum\limits_{k=0}^{\infty}\frac{%
(-b)^{k}(c;\rho)_{2k+n}}{(n+k)!k!(n+2k)!}(\frac{z}{2})^{2k+n}t^{n}.
\end{align*}
Letting $n-2k$ instead of $n,$ we get%
\begin{align*}
\sum\limits_{n=-\infty}^{\infty}G_{n}^{(b,c)}(z;\rho)t^{n} & =\sum
\limits_{n=0}^{\infty}\sum\limits_{k=0}^{[\dfrac{n}{2}]}\frac{%
(-b)^{n-k+1}(c;\rho)_{n+1}}{(n-k+1)!k!(n+1)!}(\frac{z}{2})^{n+1}t^{-n+2k-1}
\\
& +\sum\limits_{n=0}^{\infty}\sum\limits_{k=0}^{[\dfrac{n}{2}]}\frac {%
(-b)^{k}(c;\rho)_{n}}{(n-k)!k!n!}(\frac{z}{2})^{n}t^{n-2k}.
\end{align*}
Taking $n-1$~instead of $n$ in the first summation of the right side, we have%
\begin{align*}
\sum\limits_{n=-\infty}^{\infty}G_{n}^{(b,c)}(z;\rho)t^{n} & =\sum
\limits_{n=1}^{\infty}\sum\limits_{k=0}^{[\dfrac{n-1}{2}]}\frac{%
(-b)^{n-k}(c;\rho)_{n}}{(n-k)!k!n!}(\frac{z}{2})^{n}t^{-n+2k} \\
& +(c;\rho)_{0}+\sum\limits_{n=1}^{\infty}\sum\limits_{k=0}^{[\dfrac{n}{2}]}%
\frac{(-b)^{k}(c;\rho)_{n}}{(n-k)!k!n!}(\frac{z}{2})^{n}t^{n-2k}.
\end{align*}
Taking into consideration of the following fact, which was proved in \cite{R}%
, (Lemma 12, page 112-113)%
\begin{equation*}
\sum\limits_{n=1}^{\infty}\sum\limits_{k=0}^{[\dfrac{n-1}{2}]}A(n-k,n)+\sum
\limits_{n=1}^{\infty}\sum\limits_{k=0}^{[\dfrac{n}{2}]}A(k,n)=\sum
\limits_{n=1}^{\infty}\sum\limits_{k=0}^{n}A(k,n)
\end{equation*}
we have%
\begin{equation*}
\sum\limits_{n=-\infty}^{\infty}G_{n}^{(b,c)}(z;\rho)t^{n}=(c;\rho)_{0}+\sum%
\limits_{n=1}^{\infty}\sum\limits_{k=0}^{n}\frac{(-b)^{k}(c;\rho)_{n}}{%
(n-k)!k!n!}(\frac{z}{2})^{n}t^{n-2k}.
\end{equation*}
Substituting the expansion of $(t-\frac{b}{t})^{n},$we get%
\begin{align*}
\sum\limits_{n=-\infty}^{\infty}G_{n}^{(b,c)}(z;\rho)t^{n} & =\sum
\limits_{n=0}^{\infty}\frac{(c;\rho)_{n}}{(1)_{n}}\frac{(t-\frac{b}{t})^{n}(%
\frac{z}{2})^{n}}{n!}, \\
& =~_{1}F_{1}((c;\rho),1;(t-\frac{b}{t})\frac{z}{2}).
\end{align*}
\end{proof}

Substituting $b=1~$and $b=-1$ in Theorem 2.2, the generating functions of
the generalized three parameter Bessel and modified Bessel functions of the
first kind are obtained, respectively:

\begin{corollary}
For $t\neq0$ and for all finite $z~$and $n\in%
%TCIMACRO{\U{2124} }%
%BeginExpansion
\mathbb{Z}
%EndExpansion
,~~$the generating functions of the generalized three parameter Bessel and
modified Bessel functions of the first kind are given by%
\begin{align*}
_{1}F_{1}((c;\rho),1;(t-\frac{1}{t})\frac{z}{2}) & =\sum\limits_{n=-\infty
}^{\infty}J_{n}^{(c)}(z;\rho)t^{n}, \\
_{1}F_{1}((c;\rho),1;(t+\frac{1}{t})\frac{z}{2}) & =\sum\limits_{n=-\infty
}^{\infty}I_{n}^{(c)}(z;\rho)t^{n}.
\end{align*}
\end{corollary}

Taking $c=1$ and $\rho =0$ in Corollary 2.2.1 and considering the facts $%
J_{n}^{(1)}(z;0)=J_{n}(z),~I_{n}^{(1)}(z;0)=I_{n}(z)$, the following
Corollary is obtained:

\begin{corollary}
\cite{R} \ For $t\neq0$ and for all finite $z~$and $\ n\in%
%TCIMACRO{\U{2124} }%
%BeginExpansion
\mathbb{Z}
%EndExpansion
,~$generating functions of the usual Bessel and modified Bessel functions
are given by 
\begin{align*}
e^{\frac{z}{2}(t-\tfrac{1}{t})} &
=\sum\limits_{n=-\infty}^{\infty}J_{n}(z)t^{n}, \\
e^{\frac{z}{2}(t+\tfrac{1}{t})} &
=\sum\limits_{n=-\infty}^{\infty}I_{n}(z)t^{n}.
\end{align*}
\end{corollary}

In the following theorem, the integral representation of the unified four
parameter Bessel function is presented:

\begin{theorem}
The integral formula satisfied by the unified four parameter Bessel function
is given by 
\begin{equation}
G_{\nu }^{(b,c)}(z;\rho )=\frac{(\frac{z}{2})^{\nu }}{[\Gamma (\nu
+1)]^{2}\Gamma (c)}\int\limits_{0}^{\infty }t^{c+\nu -1}e^{-t-\frac{\rho }{t}%
}~_{0}F_{3}(-;\nu +1,\frac{\nu +1}{2},\frac{\nu +2}{2};\frac{-bz^{2}t^{2}}{16%
})dt
\end{equation}%
where $\func{Re}(c)>0$~and $\func{Re}(\nu )>-1.$
\end{theorem}

\begin{proof}
Applying the generalized Pochhammer expansion into (1.8) and using the usual
Pochhammer function expansion (1.2), we have%
\begin{equation*}
G_{\nu}^{(b,c)}(z;\rho)=\frac{1}{[\Gamma(\nu+1)]^{2}}\sum\limits_{k=0}^{%
\infty}\frac{(-b)^{k}}{(\nu+1)_{k}(\nu+1)_{2k}k!}(\frac{z}{2})^{2k+\nu }\{%
\frac{1}{\Gamma(c)}\int\limits_{0}^{\infty}t^{c+2k+\nu-1}e^{-t-\frac{\rho }{t%
}}dt\}.
\end{equation*}
Interchanging the order of summation and integral under the conditions $%
\func{Re}(c)>0$~and $\func{Re}(\nu)>-1,$~we have%
\begin{equation*}
G_{\nu}^{(b,c)}(z;\rho)=\frac{(\frac{z}{2})^{\nu}}{[\Gamma(\nu+1)]^{2}%
\Gamma(c)}\int\limits_{0}^{\infty}\{\sum\limits_{k=0}^{\infty}\frac{1}{%
(\nu+1)_{k}(\nu+1)_{2k}k!}(-\frac{bz^{2}t^{2}}{4})^{k}\}t^{c+\nu -1}e^{-t-%
\frac{\rho}{t}}dt.
\end{equation*}
Letting $\nu+1$~in place of $\alpha$~and taking $k$~for $n$~in the
duplication formula (1.3), we have 
\begin{equation*}
G_{\nu}^{(b,c)}(z;\rho)=\frac{(\frac{z}{2})^{\nu}}{[\Gamma(\nu+1)]^{2}%
\Gamma(c)}\int\limits_{0}^{\infty}\{\sum\limits_{k=0}^{\infty}\frac{1}{%
(\nu+1)_{k}(\frac{\nu+1}{2})_{k}(\frac{\nu+2}{2})_{k}k!}(-\frac{bz^{2}t^{2}}{%
16})^{k}\}t^{c+\nu-1}e^{-t-\frac{\rho}{t}}dt.
\end{equation*}
Taking into consideration of (1.5), we have%
\begin{equation*}
G_{\nu}^{(b,c)}(z;\rho)=\frac{(\frac{z}{2})^{\nu}}{[\Gamma(\nu+1)]^{2}%
\Gamma(c)}\int\limits_{0}^{\infty}t^{c+\nu-1}e^{-t-\frac{\rho}{t}%
}~_{0}F_{3}(-;\nu+1,\frac{\nu+1}{2},\frac{\nu+2}{2};\frac{-bz^{2}t^{2}}{16}%
)dt.
\end{equation*}
\end{proof}

Letting $b=1$~and $b=-1$ in Theorem 2.3, the integral representations of the
generalized three parameter Bessel and modified Bessel functions of the
first kind are found, respectively:

\begin{corollary}
Integral representations satisfied by the generalized three parameter Bessel
and modified Bessel functions of the first kind are given by%
\begin{align*}
J_{\nu}^{(c)}(z;\rho) & =\frac{(\frac{z}{2})^{\nu}}{[\Gamma(\nu
+1)]^{2}\Gamma(c)}\int\limits_{0}^{\infty}t^{c+\nu-1}e^{-t-\frac{\rho}{t}%
}~_{0}F_{3}(-;\nu+1,\frac{\nu+1}{2},\frac{\nu+2}{2};\frac{-z^{2}t^{2}}{16}%
)dt, \\
I_{\nu}^{(c)}(z;\rho) & =\frac{(\frac{z}{2})^{\nu}}{[\Gamma(\nu
+1)]^{2}\Gamma(c)}\int\limits_{0}^{\infty}t^{c+\nu-1}e^{-t-\frac{\rho}{t}%
}~_{0}F_{3}(-;\nu+1,\frac{\nu+1}{2},\frac{\nu+2}{2};\frac{z^{2}t^{2}}{16})dt,
\end{align*}
where $\func{Re}(c)>0$ and $\func{Re}(\nu)>-1.$
\end{corollary}

Substituting $c=1$ and $\rho =0$ in Corollary 2.3.1, the following Corollary
is obtained:

\begin{corollary}
Integral formulas satisfied by the usual Bessel and modified Bessel
functions are given by 
\begin{align*}
J_{\nu}(z) & =\frac{(\frac{z}{2})^{\nu}}{[\Gamma(\nu+1)]^{2}}\int
\limits_{0}^{\infty}t^{\nu}e^{-t}~_{0}F_{3}(-;\nu+1,\frac{\nu+1}{2},\frac {%
\nu+2}{2};\frac{-z^{2}t^{2}}{16})dt, \\
I_{\nu}(z) & =\frac{(\frac{z}{2})^{\nu}}{[\Gamma(\nu+1)]^{2}}\int
\limits_{0}^{\infty}t^{\nu}e^{-t}~_{0}F_{3}(-;\nu+1,\frac{\nu+1}{2},\frac {%
\nu+2}{2};\frac{z^{2}t^{2}}{16})dt,
\end{align*}
where $\func{Re}(\nu)>-1.$
\end{corollary}

Note that, some integrals involving usual Bessel functions were obtained in 
\cite{P.S.M, M.A, J.P, M.G, M.R}.

Taking $t=\frac{u}{1-u}$~in Theorem 2.3, the following integral
representation is found:

\begin{corollary}
Integral representation satisfied by the unified four parameter Bessel
function is given by 
\begin{align*}
G_{\nu}^{(b,c)}(z;\rho) & =\frac{(\frac{z}{2})^{\nu}}{[\Gamma(\nu
+1)]^{2}\Gamma(c)}\int\limits_{0}^{1}u^{c+\nu-1}(1-u)^{-c-\nu-1}e^{\dfrac{%
-u^{2}-\rho(1-u)^{2}}{u(1-u)}} \\
& \times~_{0}F_{3}(-;\nu+1,\frac{\nu+1}{2},\frac{\nu+2}{2};\frac{-bz^{2}u^{2}%
}{(1-u)^{2}16})du
\end{align*}
where $\func{Re}(\nu)>-1~$and $\func{Re}(c)>0.$
\end{corollary}

\begin{theorem}
$~$The Laplace transform of the unified \ four parameter Bessel function is
given by%
\begin{equation*}
\mathcal{L}\{G_{\nu}^{(b,c)}(t;\rho)\}(s)=\frac{1}{s}\sum\limits_{k=0}^{%
\infty}\frac{(-b)^{k}(c;\rho)_{2k+\nu}~}{k!\Gamma(\nu+k+1)}(\frac{1}{2s}%
)^{2k+\nu}
\end{equation*}
where $\func{Re}(c)>0,~\func{Re}(\nu)>-1$ and $\func{Re}(s)>0.$
\end{theorem}

\begin{proof}
Laplace transform is given by%
\begin{equation*}
\mathcal{L}\{G_{\nu}^{(b,c)}(t;\rho)\}(s)=\int\limits_{0}^{\infty}G_{\nu
}^{(b,c)}(t;\rho)e^{-st}dt.
\end{equation*}
Substituting the series form of the unified four parameter Bessel function,
we have%
\begin{equation*}
\mathcal{L}\{G_{\nu}^{(b,c)}(t;\rho)\}(s)=\int\limits_{0}^{\infty}\{\sum%
\limits_{k=0}^{\infty}\frac{(-b)^{k}(c;\rho)_{2k+\nu}}{k!\Gamma
(\nu+k+1)\Gamma(\nu+2k+1)}(\frac{t}{2})^{2k+\nu}\}e^{-st}dt.
\end{equation*}
Replacing the order of summation and integral by the conditions $Re(c)>0$,~$%
Re(s)>0~~$and \ $\func{Re}(\nu)>-1,~$we get%
\begin{equation*}
\mathcal{L}\{G_{\nu}^{(b,c)}(t;\rho)\}(s)=\sum\limits_{k=0}^{\infty}\frac{%
(-b)^{k}(c;\rho)_{2k+\nu}}{k!\Gamma(\nu+k+1)\Gamma(\nu+2k+1)2^{2k+\nu}}%
\{\int\limits_{0}^{\infty}t^{2k+\nu}e^{-st}dt\}.
\end{equation*}
Making the substitution $st=u$~and using the Gamma functions for the above
integral, we get%
\begin{equation*}
\mathcal{L}\{G_{\nu}^{(b,c)}(t;\rho)\}(s)=\sum\limits_{k=0}^{\infty}\frac{%
(-b)^{k}(c;\rho)_{2k+\nu}~\Gamma(\nu+2k+1)}{k!\Gamma(\nu+k+1)\Gamma
(\nu+2k+1)2^{2k+\nu}s^{2k+\nu+1}}.
\end{equation*}
Finally, we have%
\begin{equation*}
\mathcal{L}\{G_{\nu}^{(b,c)}(t;\rho)\}(s)=\frac{1}{s}\sum\limits_{k=0}^{%
\infty}\frac{(-b)^{k}(c;\rho)_{2k+\nu}~}{k!\Gamma(\nu+k+1)}(\frac{1}{2s}%
)^{2k+\nu}.
\end{equation*}
\end{proof}

The case $b=1,$ $G_{\nu}^{(b,c)}(t;\rho)$ is reduced to $J_{\nu}^{(c)}(t;%
\rho)$~and when $b=-1,~G_{\nu}^{(b,c)}(t;\rho)$ is reduced to $I_{\nu
}^{(c)}(t;\rho).~$Therefore, the following Corollary is obtained:

\begin{corollary}
$~$The Laplace transforms of the generalized three parameter Bessel and
modified Bessel functions of the first kind are given by%
\begin{align*}
\mathcal{L}\{J_{\nu}^{(c)}(t;\rho)\}(s) & =\frac{1}{s}\sum\limits_{k=0}^{%
\infty}\frac{(-1)^{k}(c;\rho)_{2k+\nu}}{k!\Gamma(\nu+k+1)}(\frac{1}{2s}%
)^{2k+\nu}, \\
\mathcal{L}\{I_{\nu}^{(c)}(t;\rho)\}(s) & =\frac{1}{s}\sum\limits_{k=0}^{%
\infty}\frac{(c;\rho)_{2k+\nu}}{k!\Gamma(\nu+k+1)}(\frac{1}{2s})^{2k+\nu},
\end{align*}
where $\func{Re}(c)>0,~\func{Re}(\nu)>-1$ and $\func{Re}(s)>0.$
\end{corollary}

Since $J_{\nu}^{(c)}(z;\rho)$ is reduced to $J_{\nu}(z)$~and $%
I_{\nu}^{(c)}(z;\rho)$~is reduced to $I_{\nu}(z)~$when $c=1$ and $\rho=0$,
the following Corollary is presented for usual Bessel and modified Bessel
functions :

\begin{corollary}
$~$The Laplace transforms of the usual Bessel and modified Bessel functions
are given by 
\begin{align*}
\mathcal{L}\{J_{\nu}(t)\}(s) & =\frac{1}{s}\sum\limits_{k=0}^{\infty}\frac{%
(-1)^{k}\Gamma(2k+\nu+1)}{k!\Gamma(\nu+k+1)}(\frac{1}{2s})^{2k+\nu}, \\
\mathcal{L}\{I_{\nu}(t)\}(s) & =\frac{1}{s}\sum\limits_{k=0}^{\infty}\frac{%
\Gamma(2k+\nu+1)}{k!\Gamma(\nu+k+1)}(\frac{1}{2s})^{2k+\nu},
\end{align*}
where $\func{Re}(s)>0$ and $\func{Re}(\nu)>-1.$ \ \ \ \ \ \ \ \ \ \ \ \ \ \
\ \ \ \ \ \ \ \ \ \ \ \ \ \ \ \ \ \ \ \ \ \ \ \ \ \ \ \ \ \ \ \ \ \ \ \ \ \
\ \ \ \ \ \ \ \ \ \ \ \ \ \ \ \ \ \ \ \ \ \ \ \ \ \ \ \ \ \ \ \ \ \ \ \ \ \
\ \ \ \ \ \ \ \ \ \ \ \ \ \ \ \ \ \ \ \ \ \ \ \ \ \ \ \ \ \ \ \ \ \ \ \ \ \
\ \ \ \ \ \ \ \ \ \ \ \ \ \ \ \ \ \ \ \ \ \ \ \ \ \ \ \ \ \ \ \ \ \ \ \ \ \
\ \ \ \ \ \ \ \ \ \ \ \ \ \ \ \ \ \ \ \ \ \ \ \ \ \ \ \ \ \ \ \ \ \ \ \ \ \
\ \ \ \ \ \ \ \ \ \ \ \ \ \ \ \ \ \ \ \ \ \ \ \ \ \ \ \ \ \ \ \ \ \ \ \ \ \
\ \ \ \ \ \ \ \ \ \ \ \ \ \ \ \ \ \ \ \ \ \ \ \ \ \ \ \ \ \ \ \ 
\end{corollary}

Taking $\nu =0$~~in Corollary 2.4.2 and using the fact that \cite{R} 
\begin{equation*}
(1-z)^{-\alpha }=\sum\limits_{n=0}^{\infty }\frac{(\alpha )_{n}}{n!}z^{n}
\end{equation*}%
and applying 
\begin{equation*}
(2k)!=2^{2k}(\frac{1}{2})_{k}k!,
\end{equation*}%
Laplace transform of the $J_{0}(t)$~and $I_{0}(t)$ are found$.$

\begin{corollary}
\cite{M.O.T} \ The Laplace transforms of the $J_{0}(t)$~and $I_{0}(t)$~are
given by%
\begin{align*}
\mathcal{L}\{J_{0}(t)\}(s) & =\frac{1}{\sqrt{s^{2}+1}},~\func{Re}(s)>0, \\
\mathcal{L}\{I_{0}(t)\}(s) & =\frac{1}{\sqrt{s^{2}-1}},~\func{Re}(s)>0.
\end{align*}
\ 
\end{corollary}

Note that the Laplace transform of the modified Bessel functions of the
second kind was obtained in \cite{H.S}.

\begin{theorem}
$~$Mellin transform involving the unified four parameter Bessel function \
is given by 
\begin{align*}
& \mathcal{M}\{e^{-z}G_{\nu}^{(b,c)}(z;\rho);s\} \\
& =\frac{1}{2^{\nu}[\Gamma(\nu+1)]^{2}}\sum\limits_{k=0}^{\infty}\frac{%
(-b)^{k}(c;\rho)_{\nu+2k}~~\Gamma(s+\nu+2k)~}{(\nu+1)_{k}~(\frac{\nu +1}{2}%
)_{k}~(\frac{\nu+2}{2})_{k}~16^{k}k!},
\end{align*}
where $~\func{Re}(\nu)>-1$ and $s+\nu\neq\{0,-1,-2,...\}.$
\end{theorem}

\begin{proof}
The Mellin transform is given by%
\begin{equation*}
\mathcal{M}\{e^{-z}G_{\nu}^{(b,c)}(z;\rho);s\}=\int\limits_{0}^{%
\infty}e^{-z}z^{s-1}G_{\nu}^{(b,c)}(z;\rho)dz.
\end{equation*}
Substituting equation (2.2) instead of $G_{\nu}^{(b,c)}(z;\rho)$, we have%
\begin{align*}
& \mathcal{M}\{e^{-z}G_{\nu}^{(b,c)}(z;\rho);s\} \\
& =\int\limits_{0}^{\infty}e^{-z}z^{s-1}[\frac{z^{\nu}}{2^{\nu}[\Gamma
(\nu+1)]^{2}\Gamma(c)}\int\limits_{0}^{\infty}t^{c+\nu-1}e^{-t-\frac{\rho}{t}%
}~_{0}F_{3}(-;\nu+1,\frac{\nu+1}{2},\frac{\nu+2}{2};\frac{-bz^{2}t^{2}}{16}%
)dt]dz.
\end{align*}
Inserting the series forms of hypergeometric function, we have 
\begin{align*}
& \mathcal{M}\{e^{-z}G_{\nu}^{(b,c)}(z;\rho);s\} \\
& =\int\limits_{0}^{\infty}\int\limits_{0}^{\infty}\frac{z^{s+\nu-1}e^{-z}}{%
2^{\nu}[\Gamma(\nu+1)]^{2}\Gamma(c)}t^{c+\nu-1}e^{-t-\frac{\rho}{t}%
}(~\sum\limits_{k=0}^{\infty}\frac{(-1)^{k}b^{k}z^{2k}t^{2k}}{(\nu+1)_{k}(%
\frac{\nu+1}{2})_{k}(\frac{\nu+2}{2})_{k}16^{k}k!})dtdz \\
& =\frac{1}{2^{\nu}[\Gamma(\nu+1)]^{2}\Gamma(c)}\int\limits_{0}^{\infty}\int%
\limits_{0}^{\infty}z^{s+\nu-1}t^{c+\nu-1}e^{-t-\frac{\rho}{t}%
-z}(\sum\limits_{k=0}^{\infty}\frac{(-1)^{k}b^{k}z^{2k}t^{2k}}{(\nu+1)_{k}(%
\frac{\nu+1}{2})_{k}(\frac{\nu+2}{2})_{k}16^{k}k!})dtdz.
\end{align*}
Interchanging the order of integral and summation which is satisfied under
the conditions of the Theorem, we get%
\begin{align*}
& \mathcal{M}\{e^{-z}G_{\nu}^{(b,c)}(z;\rho);s\} \\
& =\frac{1}{2^{\nu}[\Gamma(\nu+1)]^{2}}\int\limits_{0}^{\infty}(\sum%
\limits_{k=0}^{\infty}\frac{(-1)^{k}b^{k}z^{2k+s+\nu-1}e^{-z}}{(\nu+1)_{k}(%
\frac{\nu+1}{2})_{k}(\frac{\nu+2}{2})_{k}16^{k}k!})(\frac {1}{\Gamma(c)}%
\int\limits_{0}^{\infty}t^{c+\nu+2k-1}e^{-t-\frac{\rho}{t}}dt)dz.
\end{align*}
Taking into consideration of the generalized Pochhammer function expansion $%
(c;\rho)_{\nu+2k},$ we get%
\begin{equation*}
\mathcal{M}\{e^{-z}G_{\nu}^{(b,c)}(z;\rho);s\}=\frac{1}{2^{\nu}[\Gamma
(\nu+1)]^{2}}\int\limits_{0}^{\infty}(\sum\limits_{k=0}^{\infty}\frac {%
(-1)^{k}(c;\rho)_{\nu+2k}~b^{k}z^{2k+s+\nu-1}e^{-z}}{(\nu+1)_{k}(\frac{\nu +1%
}{2})_{k}(\frac{\nu+2}{2})_{k}16^{k}k!})dz.
\end{equation*}
Changing the order of integral and summation under the conditions $\func{Re}%
(\nu)>-1,~s+\nu\neq\{0,-1,-2,...\}~$and $\func{Re}(\rho)>0,$~we get%
\begin{equation*}
\mathcal{M}\{e^{-z}G_{\nu}^{(b,c)}(z;\rho);s\}=\frac{1}{2^{\nu}[\Gamma
(\nu+1)]^{2}}\sum\limits_{k=0}^{\infty}\frac{(-1)^{k}(c;\rho)_{\nu+2k}~b^{k}%
}{(\nu+1)_{k}(\frac{\nu+1}{2})_{k}(\frac{\nu+2}{2})_{k}16^{k}k!}%
(\int\limits_{0}^{\infty}z^{2k+s+\nu-1}e^{-z}dz).
\end{equation*}
Using the definition of the Gamma function and taking into consideration of
(1.2), we have%
\begin{align*}
& \mathcal{M}\{e^{-z}G_{\nu}^{(b,c)}(z;\rho);s\} \\
& =\frac{1}{2^{\nu}[\Gamma(\nu+1)]^{2}}\sum\limits_{k=0}^{\infty}\frac{%
(-b)^{k}(c;\rho)_{\nu+2k}~~\Gamma(s+\nu+2k)~}{(\nu+1)_{k}~(\frac{\nu +1}{2}%
)_{k}~(\frac{\nu+2}{2})_{k}~16^{k}k!}.
\end{align*}
\end{proof}

The Mellin transforms of the generalized three parameter Bessel and modified
Bessel functions of the first kind can be obtained by substituting $~b=1$%
~and $b=-1$ in Theorem 2.5, respectively. Thus, the following result is
clear:

\begin{corollary}
$~$The Mellin transforms involving the generalized three parameter Bessel
and modified Bessel functions of the first kind are given by 
\begin{align*}
\mathcal{M}\{e^{-z}J_{\nu}^{(c)}(z;\rho);s\} & =\frac{1}{2^{\nu}[\Gamma
(\nu+1)]^{2}}\sum\limits_{k=0}^{\infty}\frac{(-1)^{k}(c;\rho)_{\nu+2k}~~%
\Gamma(s+\nu+2k)~}{(\nu+1)_{k}~(\frac{\nu+1}{2})_{k}~(\frac{\nu+2}{2}%
)_{k}~16^{k}k!}, \\
\mathcal{M}\{e^{-z}I_{\nu}^{(c)}(z;\rho);s\} & =\frac{1}{2^{\nu}[\Gamma
(\nu+1)]^{2}}\sum\limits_{k=0}^{\infty}\frac{(c;\rho)_{\nu+2k}~~\Gamma
(s+\nu+2k)~}{(\nu+1)_{k}~(\frac{\nu+1}{2})_{k}~(\frac{\nu+2}{2})_{k}~16^{k}k!%
},
\end{align*}
where~$\func{Re}(\nu)>-1,$ $\func{Re}(c)>0~$and $s+\nu \neq\{0,-1,-2,...\}.$
\end{corollary}

Substituting $c=1$~and $\rho =0$~in Corollary 2.5.1,~ the Mellin transforms
involving the usual Bessel and modified Bessel functions are obtained. Thus,
the following Corollary is given:

\begin{corollary}
$~$Mellin transforms involving the usual Bessel and modified Bessel
functions are given by%
\begin{align*}
\mathcal{M}\{e^{-z}J_{\nu}(z);s\} & =\frac{1}{2^{\nu}}\sum\limits_{k=0}^{%
\infty}\frac{(-1)^{k}\Gamma(s+\nu+2k)}{\Gamma(\nu+k+1)2^{2k}k!}, \\
\mathcal{M}\{e^{-z}I_{\nu}(z);s\} & =\frac{1}{2^{\nu}}\sum\limits_{k=0}^{%
\infty}\frac{\Gamma(s+\nu+2k)}{\Gamma(\nu+k+1)2^{2k}k!}
\end{align*}
where$~\func{Re}(\nu)>-1$ and $s+\nu\neq\{0,-1,-2,...\}.$
\end{corollary}

Letting $\nu =0$~\ in Corollary 2.5.2 and using the fact that 
\begin{equation*}
(s)_{2k}=2^{2k}(\frac{s}{2})_{k}(\frac{s+1}{2})_{k},
\end{equation*}%
the following result is obtained:

\begin{corollary}
\cite{M.O.T} The Mellin transforms involving the usual Bessel function $%
e^{-z}J_{0}(z)$~and modified Bessel function $e^{-z}I_{0}(z)$~are given by 
\begin{align*}
\mathcal{M}\{e^{-z}J_{0}(z);s\} & =\Gamma(s)~_{2}F_{1}(\frac{s}{2},\frac{s+1%
}{2};1;-1), \\
\mathcal{M}\{e^{-z}I_{0}(z);s\} & =\Gamma(s)~_{2}F_{1}(\frac{s}{2},\frac{s+1%
}{2};1;1)
\end{align*}
where $\func{Re}(s)>0.$
\end{corollary}

Note that the Mellin transforms of products of Bessel and generalized
hypergeometric functions were obtained in \cite{R.M}.

\begin{corollary}
Taking $s=1$~in Theorem 2.5, the following integral formula is obtained%
\begin{equation*}
\int\limits_{0}^{\infty }e^{-z}G_{\nu }^{(b,c)}(z;\rho )dz=\frac{1}{2^{\nu
}[\Gamma (\nu +1)]^{2}}\sum\limits_{k=0}^{\infty }\frac{(-b)^{k}(c;\rho
)_{\nu +2k}~~\Gamma (\nu +2k+1)~}{(\nu +1)_{k}~(\frac{\nu +1}{2})_{k}~(\frac{%
\nu +2}{2})_{k}~16^{k}k!}
\end{equation*}%
where $\func{Re}(\nu )>-1$~and $\func{Re}(c)>0$.
\end{corollary}

Taking $s=1$ in Theorem 2.5 and substituting $b=1$ and $b=-1$~respectively,
the following Corollaries are presented:

\begin{corollary}
$~$Infinite integrals involving the generalized three parameter Bessel and
modified Bessel functions of the first kind are given by%
\begin{align*}
\int\limits_{0}^{\infty}e^{-z}J_{\nu}^{(c)}(z;\rho)dz & =\frac{1}{2^{\nu
}[\Gamma(\nu+1)]^{2}}\sum\limits_{k=0}^{\infty}\frac{(-1)^{k}(c;\rho)_{\nu
+2k}~~\Gamma(\nu+2k+1)~}{(\nu+1)_{k}~(\frac{\nu+1}{2})_{k}~(\frac{\nu+2}{2}%
)_{k}~16^{k}k!}, \\
\int\limits_{0}^{\infty}e^{-z}I_{\nu}^{(c)}(z;\rho)dz & =\frac{1}{2^{\nu
}[\Gamma(\nu+1)]^{2}}\sum\limits_{k=0}^{\infty}\frac{(c;\rho)_{\nu+2k}~~%
\Gamma(\nu+2k+1)~}{(\nu+1)_{k}~(\frac{\nu+1}{2})_{k}~(\frac{\nu+2}{2}%
)_{k}~16^{k}k!},
\end{align*}
where $\func{Re}(\nu)>-1$~and~ $\func{Re}(c)>0.$
\end{corollary}

Integrals involving usual Bessel and modified Bessel functions are given in
the following Corollary:

\begin{corollary}
$~$Infinite integrals involving the usual Bessel and modified Bessel
functions are given by 
\begin{align*}
\int\limits_{0}^{\infty}e^{-z}J_{\nu}(z)dz & =\frac{1}{2^{\nu}[\Gamma
(\nu+1)]^{2}}\sum\limits_{k=0}^{\infty}\frac{(-1)^{k}~~[\Gamma(\nu+2k+1]^{2}~%
}{(\nu+1)_{k}~(\frac{\nu+1}{2})_{k}~(\frac{\nu+2}{2})_{k}~16^{k}k!}, \\
\int\limits_{0}^{\infty}e^{-z}I_{\nu}(z)dz & =\frac{1}{2^{\nu}[\Gamma
(\nu+1)]^{2}}\sum\limits_{k=0}^{\infty}\frac{~~[\Gamma(\nu+2k+1)]^{2}~}{%
(\nu+1)_{k}~(\frac{\nu+1}{2})_{k}~(\frac{\nu+2}{2})_{k}~16^{k}k!},
\end{align*}
where $\func{Re}(\nu)>-1.$
\end{corollary}

Besides, the Mellin transforms satisfied by unified four parameter Bessel,
generalized three parameter Bessel and modified Bessel functions of the
first kind, usual and modified Bessel functions are obtained in the
following Theorem and corresponding Corollaries, as well.

\begin{theorem}
$~$For the unified four parameter Bessel function, the Mellin transform is
given by%
\begin{align*}
\mathcal{M}\{G_{\nu}^{(b,c)}(z;\rho);s\} & =\frac{\Gamma(s)\Gamma(c+\nu +s)}{%
[\Gamma(\nu+1)]^{2}\Gamma(c)}\sum\limits_{m=0}^{\infty}\frac {%
(\nu+2m)~\Gamma(v+m)}{m!}J_{\nu+2m}(z) \\
& \times~_{2}F_{3}(\frac{c+\nu+s}{2},\frac{c+\nu+s+1}{2};\nu+1,\frac{\nu +1}{%
2},\frac{\nu+2}{2};\frac{-bz^{2}}{4})
\end{align*}
where $\func{Re}(\nu)>-1,~\func{Re}(c)>0,$~$c+\nu+s,~\nu
+m\neq\{0,-1,-2,...\}.$
\end{theorem}

\begin{proof}
The Mellin transform is given by%
\begin{equation*}
\mathcal{M}\{G_{\nu}^{(b,c)}(z;\rho);s\}=\int\limits_{0}^{\infty}\rho
^{s-1}G_{\nu}^{(b,c)}(z;\rho)d\rho.
\end{equation*}
Substituting equation (2.2) for $G_{\nu}^{(b,c)}(z;\rho),$ we have%
\begin{align*}
& \mathcal{M}\{G_{\nu}^{(b,c)}(z;\rho);s\} \\
& =\int\limits_{0}^{\infty}\rho^{s-1}[\frac{z^{\nu}}{2^{\nu}[\Gamma
(\nu+1)]^{2}\Gamma(c)}\int\limits_{0}^{\infty}t^{c+\nu-1}e^{-t-\frac{\rho}{t}%
}~_{0}F_{3}(-;\nu+1,\frac{\nu+1}{2},\frac{\nu+2}{2};\frac{-bz^{2}t^{2}}{16}%
)dt]d\rho.
\end{align*}
Interchanging the order of integrals which is valid due to the conditions of
the Theorem, we get 
\begin{align*}
\mathcal{M}\{G_{\nu}^{(b,c)}(z;\rho);s\} & =\frac{z^{\nu}}{%
2^{\nu}[\Gamma(\nu+1)]^{2}\Gamma(c)}\int\limits_{0}^{\infty}t^{c+%
\nu-1}e^{-t}~_{0}F_{3}(-;\nu+1,\frac{\nu+1}{2},\frac{\nu+2}{2};\frac{%
-bz^{2}t^{2}}{16}) \\
& \times(\int\limits_{0}^{\infty}\rho^{s-1}e^{-\frac{\rho}{t}}d\rho)dt.
\end{align*}
Letting $\rho=ut$ and using the Gamma function, which is presented in (1.4),
in above integral we have%
\begin{align*}
& \mathcal{M}\{G_{\nu}^{(b,c)}(z;\rho);s\} \\
& =\frac{\Gamma(s)z^{\nu}}{2^{\nu}[\Gamma(\nu+1)]^{2}\Gamma(c)}%
\int\limits_{0}^{\infty}t^{c+\nu+s-1}e^{-t}~_{0}F_{3}(-;\nu+1,\frac{\nu+1}{2}%
,\frac{\nu+2}{2};\frac{-bz^{2}t^{2}}{16})dt.
\end{align*}
Now, writing the expansion of the hypergeometric series, we have%
\begin{align*}
& \mathcal{M}\{G_{\nu}^{(b,c)}(z;\rho);s\} \\
& =\frac{\Gamma(s)z^{\nu}}{2^{\nu}[\Gamma(\nu+1)]^{2}\Gamma(c)}%
\int\limits_{0}^{\infty}t^{c+\nu+s-1}e^{-t}(\sum\limits_{k=0}^{\infty}\frac{%
(-b)^{k}~z^{2k}t^{2k}}{(\nu+1)_{k}(\frac{\nu+1}{2})_{k}(\frac{\nu+2}{2}%
)_{k}2^{4k}k!})dt.
\end{align*}
Interchanging the order of integral and summation which is true for $~\func{%
Re}(c)>0,~\func{Re}(\nu)>-1,~\func{Re}(s)>0,$~we get%
\begin{align*}
& \mathcal{M}\{G_{\nu}^{(b,c)}(z;\rho);s\} \\
& =\frac{\Gamma(s)z^{\nu}}{2^{\nu}[\Gamma(\nu+1)]^{2}\Gamma(c)}%
\sum\limits_{k=0}^{\infty}(\int\limits_{0}^{\infty}t^{c+\nu+s+2k-1}e^{-t}dt)%
\frac{(-b)^{k}~z^{2k}}{(\nu+1)_{k}(\frac{\nu+1}{2})_{k}(\frac{\nu +2}{2}%
)_{k}2^{4k}k!}.
\end{align*}
Using the Gamma function, we have%
\begin{align*}
& \mathcal{M}\{G_{\nu}^{(b,c)}(z;\rho);s\} \\
& =\frac{\Gamma(s)z^{\nu}}{2^{\nu}[\Gamma(\nu+1)]^{2}\Gamma(c)}%
\sum\limits_{k=0}^{\infty}\frac{\Gamma(c+\nu+s+2k)(-b)^{k}}{(\nu+1)_{k}(%
\frac{\nu+1}{2})_{k}(\frac{\nu+2}{2})_{k}}\frac{z^{2k}}{2^{4k}k!}.
\end{align*}
Applying (1.2) for the Gamma function \ $\Gamma(c+\nu+s+2k)$, we have%
\begin{align*}
& \mathcal{M}\{G_{\nu}^{(b,c)}(z;\rho);s\} \\
& =\frac{\Gamma(s)z^{\nu}}{2^{\nu}[\Gamma(\nu+1)]^{2}\Gamma(c)}%
\sum\limits_{k=0}^{\infty}\frac{(c+\nu+s)_{2k}~\Gamma(c+\nu+s)}{(\nu +1)_{k}(%
\frac{\nu+1}{2})_{k}(\frac{\nu+2}{2})_{k}}\frac{(-b)^{k}z^{2k}}{2^{4k}k!}.
\end{align*}
By (1.3), instead of $(c+\nu+s)_{2k},$ we can write $2^{2k}(\frac{c+\nu+s}{2}%
)_{k}(\frac{c+\nu+s+1}{2})_{k}.$ Thus, we have%
\begin{align*}
\mathcal{M}\{G_{\nu}^{(b,c)}(z;\rho);s\} & =\frac{\Gamma(s)\Gamma
(c+\nu+s)z^{\nu}}{2^{\nu}[\Gamma(\nu+1)]^{2}\Gamma(c)}\sum\limits_{k=0}^{%
\infty}\frac{2^{2k}(\frac{c+\nu+s}{2})_{k}(\frac{c+\nu+s+1}{2})_{k}}{%
(\nu+1)_{k}(\frac{\nu+1}{2})_{k}(\frac{\nu+2}{2})_{k}}\frac{(-b)^{k}z^{2k}}{%
2^{4k}k!} \\
& =\frac{\Gamma(s)\Gamma(c+\nu+s)z^{\nu}}{2^{\nu}[\Gamma(\nu+1)]^{2}\Gamma(c)%
}\sum\limits_{k=0}^{\infty}\frac{(\frac{c+\nu+s}{2})_{k}(\frac {c+\nu+s+1}{2}%
)_{k}}{(\nu+1)_{k}(\frac{\nu+1}{2})_{k}(\frac{\nu+2}{2})_{k}}\frac{(-\frac{%
bz^{2}}{4})^{k}}{k!}.
\end{align*}
In the last step, using the hypergeometric expansion (1.5) and%
\begin{equation*}
(\frac{z}{2})^{\nu}=\sum\limits_{m=0}^{\infty}\frac{(\nu+2m)(\nu+m-1)!}{m!}%
J_{\nu+2m}(z)
\end{equation*}
in \cite{W}, we have%
\begin{align*}
\mathcal{M}\{G_{\nu}^{(b,c)}(z;\rho);s\} & =\frac{\Gamma(s)\Gamma(c+\nu +s)}{%
[\Gamma(\nu+1)]^{2}\Gamma(c)}\sum\limits_{m=0}^{\infty}\frac {%
(\nu+2m)~\Gamma(\nu+m)}{m!}J_{\nu+2m}(z) \\
& \times~_{2}F_{3}(\frac{c+\nu+s}{2},\frac{c+\nu+s+1}{2};\nu+1,\frac{\nu +1}{%
2},\frac{\nu+2}{2};-\frac{bz^{2}}{4}).
\end{align*}
\end{proof}

The Mellin transform of the generalized three parameter Bessel and modified
Bessel functions of the first kind can be obtained by substituting $b=1$ and 
$b=-1$ in Theorem 2.6, respectively. Thus, the following result is clear:

\begin{corollary}
$~$For the generalized three parameter Bessel and modified Bessel functions
of the first kind, the Mellin transforms are given by%
\begin{align*}
\mathcal{M}\{J_{\nu}^{(c)}(z;\rho);s\} & =\frac{\Gamma(s)\Gamma(c+\nu +s)}{%
[\Gamma(\nu+1)]^{2}\Gamma(c)}\sum\limits_{m=0}^{\infty}\frac {%
(\nu+2m)~\Gamma(\nu+m)}{m!}J_{\nu+2m}(z) \\
& \times~_{2}F_{3}(\frac{c+\nu+s}{2},\frac{c+\nu+s+1}{2};\nu+1,\frac{\nu +1}{%
2},\frac{\nu+2}{2};-\frac{z^{2}}{4}), \\
\mathcal{M}\{I_{\nu}^{(c)}(z;\rho);s\} & =\frac{\Gamma(s)\Gamma(c+\nu +s)}{%
[\Gamma(\nu+1)]^{2}\Gamma(c)}\sum\limits_{m=0}^{\infty}\frac {%
(\nu+2m)~\Gamma(\nu+m)}{m!}J_{\nu+2m}(z) \\
& \times~_{2}F_{3}(\frac{c+\nu+s}{2},\frac{c+\nu+s+1}{2};\nu+1,\frac{\nu +1}{%
2},\frac{\nu+2}{2};\frac{z^{2}}{4}),
\end{align*}
where $\func{Re}(\nu)>-1,~\func{Re}(c)>0$,$~\func{Re}(s)>0,~c+\nu+s,~\nu+m%
\neq\{0,-1,-2,...\}.$
\end{corollary}

Taking $c=1$ in Corollary 2.6.1, it is reduced the following Corollary:

\begin{corollary}
$~$For the functions $J_{\nu}^{(1,1)}(z;\rho)$ and $I_{\nu}^{(-1,1)}(z;%
\rho), $ the Mellin transform is given by%
\begin{align*}
\mathcal{M}\{J_{\nu}^{(1,1)}(z;\rho);s\} & =\frac{\Gamma(s)\Gamma(1+\nu +s)}{%
[\Gamma(\nu+1)]^{2}}\sum\limits_{m=0}^{\infty}\frac{(\nu+2m)~\Gamma (\nu+m)}{%
m!}J_{\nu+2m}(z) \\
& \times~_{2}F_{3}(\frac{\nu+s+1}{2},\frac{\nu+s+2}{2};\nu+1,\frac{\nu+1}{2},%
\frac{\nu+2}{2};-\frac{z^{2}}{4}), \\
\mathcal{M}\{I_{\nu}^{(-1,1)}(z;\rho);s\} & =\frac{\Gamma(s)\Gamma(1+\nu +s)%
}{[\Gamma(\nu+1)]^{2}}\sum\limits_{m=0}^{\infty}\frac{(\nu+2m)~\Gamma (\nu+m)%
}{m!}J_{\nu+2m}(z) \\
& \times~_{2}F_{3}(\frac{\nu+s+1}{2},\frac{\nu+s+2}{2};\nu+1,\frac{\nu+1}{2},%
\frac{\nu+2}{2};\frac{z^{2}}{4}),
\end{align*}
where $\func{Re}(\nu)>-1,~\func{Re}(s)>0,~\nu+m\neq
\{0,-1,-2,...\},~\nu+s\neq\{-1,-2,...\}.$
\end{corollary}

Letting $s=1$~in Theorem 2.6, the following integral representation of the
unified four parameter Bessel function is valid:

\begin{corollary}
$~$Infinite integral of the unified four parameter Bessel function is given
by 
\begin{align*}
\int\limits_{0}^{\infty}G_{\nu}^{(b,c)}(z;\rho)d\rho & =\frac{\Gamma
(c+\nu+1)}{[\Gamma(\nu+1)]^{2}\Gamma(c)}\sum\limits_{m=0}^{\infty}\frac {%
(\nu+2m)~\Gamma(\nu+m)}{m!}J_{\nu+2m}(z) \\
& \times~_{2}F_{3}(\frac{c+\nu+1}{2},\frac{c+\nu+2}{2};\nu+1,\frac{\nu+1}{2},%
\frac{\nu+2}{2};-\frac{bz^{2}}{4}),
\end{align*}
where $\func{Re}(\nu)>-1,~\func{Re}(c)>0,~\nu+m\neq
\{0,-1,-2,...\},~c+\nu\neq\{-1,-2,...\}.$
\end{corollary}

Substituting $b=1$~and $b=-1~$in Corollary 2.6.3$,$ the following Corollary
is found:

\begin{corollary}
$~$Infinite integrals satisfied by the generalized three parameter Bessel
and modified Bessel functions of the first kind are given by%
\begin{align*}
\int\limits_{0}^{\infty}J_{\nu}^{(c)}(z;\rho)d\rho & =\frac{\Gamma(c+\nu +1)%
}{[\Gamma(\nu+1)]^{2}\Gamma(c)}\sum\limits_{m=0}^{\infty}\frac {%
(\nu+2m)~\Gamma(\nu+m)}{m!}J_{\nu+2m}(z) \\
& \times~_{2}F_{3}(\frac{c+\nu+1}{2},\frac{c+\nu+2}{2};\nu+1,\frac{\nu+1}{2},%
\frac{\nu+2}{2};-\frac{z^{2}}{4}), \\
\int\limits_{0}^{\infty}I_{\nu}^{(c)}(z;\rho)d\rho & =\frac{\Gamma(c+\nu +1)%
}{[\Gamma(\nu+1)]^{2}\Gamma(c)}\sum\limits_{m=0}^{\infty}\frac {%
(\nu+2m)~\Gamma(\nu+m)}{m!}J_{\nu+2m}(z) \\
& \times~_{2}F_{3}(\frac{c+\nu+1}{2},\frac{c+\nu+2}{2};\nu+1,\frac{\nu+1}{2},%
\frac{\nu+2}{2};\frac{z^{2}}{4}),
\end{align*}
where $\func{Re}(\nu)>-1,~\func{Re}(c)>0,~\nu+m\neq
\{0,-1,-2,...\},~c+\nu\neq\{-1,-2,...\}.$
\end{corollary}

Letting $c=1$ in Corollary 2.6.4, the integral representations of the
functions $J_{\nu }^{(1,1)}(z;\rho )$ and $I_{\nu }^{(1,1)}(z;\rho )$ are
presented, respectively:

\begin{corollary}
$~$For the functions $J_{\nu}^{(1,1)}(z;\rho)$ and $I_{\nu}^{(-1,1)}(z;\rho
),$~the following integrals are valid%
\begin{align*}
\int\limits_{0}^{\infty}J_{\nu}^{(1,1)}(z;\rho)d\rho & =\frac{\Gamma(\nu +2)%
}{[\Gamma(\nu+1)]^{2}}\sum\limits_{m=0}^{\infty}\frac{(\nu+2m)~\Gamma (\nu+m)%
}{m!}J_{\nu+2m}(z) \\
& \times~~_{2}F_{3}(\frac{\nu+2}{2},\frac{\nu+3}{2};\nu+1,\frac{\nu+1}{2},%
\frac{\nu+2}{2};-\frac{z^{2}}{4}), \\
\int\limits_{0}^{\infty}I_{\nu}^{(-1,1)}(z;\rho)d\rho & =\frac{\Gamma(\nu +2)%
}{[\Gamma(\nu+1)]^{2}}\sum\limits_{m=0}^{\infty}\frac{(\nu+2m)~\Gamma (\nu+m)%
}{m!}J_{\nu+2m}(z) \\
& \times~_{2}F_{3}(\frac{\nu+2}{2},\frac{\nu+3}{2};\nu+1,\frac{\nu+1}{2},%
\frac{\nu+2}{2};\frac{z^{2}}{4}),
\end{align*}
where $\func{Re}(\nu)>-1~$and $\nu+m\neq\{0,-1,-2,...\}.$
\end{corollary}

The expansion of the unified four parameter Bessel function is introduced as
a series of Bessel functions in the following Theorem:

\begin{theorem}
Unified four parameter Bessel function is given by a double series of \
Bessel functions%
\begin{align*}
& G_{\nu}^{(b,c)}(z;\rho) \\
& =\frac{\Gamma(\mu+1)}{[\Gamma(\nu+1)]^{2}}\sum\limits_{m=0}^{\infty}\sum%
\limits_{n=0}^{\infty}(\frac{z}{2})^{\nu-\mu+m+n}\frac{(-1)^{n}J_{\mu
+m+n}(z)(-m-n)_{n}(-b)^{n}(c;\rho)_{\nu+2n}(\mu+1)_{n}}{(m+n)!n!(\nu
+1)_{n}(\nu+1)_{2n}}
\end{align*}
where $\mu\neq\nu$ and $\func{Re}(\nu)>-1,~\func{Re}(\mu)>-1.$
\end{theorem}

\begin{proof}
Taking into consideration of the definition of the unified four parameter
Bessel function, we have%
\begin{equation*}
G_{\nu}^{(b,c)}(z;\rho)=\sum\limits_{n=0}^{\infty}\frac{(-b)^{n}(c;\rho
)_{\nu+2n}}{n!~\Gamma(n+\nu+1)~\Gamma(2n+\nu+1)}(\frac{z}{2})^{\nu-\mu +n}(%
\frac{z}{2})^{\mu+n}.
\end{equation*}
Using Gegenbauer expansion \cite{W} (page 143)%
\begin{equation*}
(\frac{z}{2})^{\mu+n}=\Gamma(\mu+n+1)\sum\limits_{p=0}^{\infty}\frac{(\frac {%
z}{2})^{p}}{p!}J_{\mu+n+p}(z),
\end{equation*}
we have%
\begin{equation*}
G_{\nu}^{(b,c)}(z;\rho)=\sum\limits_{n=0}^{\infty}\frac{(-b)^{n}(c;\rho
)_{\nu+2n}}{n!~\Gamma(n+\nu+1)~\Gamma(2n+\nu+1)}(\frac{z}{2})^{\nu-\mu
+n}~\Gamma(\mu+n+1)\sum\limits_{p=0}^{\infty}\frac{(\frac{z}{2})^{p}}{p!}%
J_{\mu+n+p}(z).
\end{equation*}
Taking $m-n$~for $p,$ we have 
\begin{align*}
G_{\nu}^{(b,c)}(z;\rho) & =\sum\limits_{m=0}^{\infty}(\frac{z}{2})^{\nu
-\mu+m}J_{\mu+m}(z)\{\sum\limits_{n=0}^{\infty}\frac{(-b)^{n}(c;\rho)_{\nu
+2n}\Gamma(\mu+n+1)}{n!(m-n)!~\Gamma(n+\nu+1)~\Gamma(2n+\nu+1)}\} \\
& =\sum\limits_{m=0}^{\infty}(\frac{z}{2})^{\nu-\mu+m}J_{\mu+m}(z)\{\sum%
\limits_{n=0}^{\infty}\frac{(-1)^{n}m!b^{n}(c;\rho)_{\nu+2n}\Gamma(\mu+n+1)}{%
(m-n)!n!\Gamma(n+\nu+1)\Gamma(2n+\nu+1)m!}.
\end{align*}
Now, using the expansion of 
\begin{equation*}
(-m)_{n}=\frac{(-1)^{n}m!}{(m-n)!},
\end{equation*}
and $(\nu+1)_{n}$~and $(\nu+1)_{2n},$we get%
\begin{equation*}
G_{\nu}^{(b,c)}(z;\rho)=\sum\limits_{m=0}^{\infty}(\frac{z}{2})^{\nu-\mu
+m~}J_{\mu+m}(z)~\frac{\Gamma(\mu+1)}{[\Gamma(\nu+1)]^{2}m!}\sum
\limits_{n=0}^{m}\frac{(-m)_{n}b^{n}(c;\rho)_{\nu+2n}(\mu+1)_{n}}{%
n!(\nu+1)_{n}(\nu+1)_{2n}}.
\end{equation*}
Applying the Cauchy product for the series~\cite{R}, we get%
\begin{equation*}
G_{\nu}^{(b,c)}(z;\rho)=\frac{\Gamma(\mu+1)}{[\Gamma(\nu+1)]^{2}}%
\sum\limits_{m=0}^{\infty}\sum\limits_{n=0}^{\infty}(\frac{z}{2})^{\nu
-\mu+m+n}\frac{(-1)^{n}J_{\mu+m+n}(z)(-m-n)_{n}(-b)^{n}(c;\rho)_{\nu+2n}(%
\mu+1)_{n}}{(m+n)!n!(\nu+1)_{n}(\nu+1)_{2n}}.
\end{equation*}
\end{proof}

The expansions of the generalized three parameter Bessel and modified Bessel
functions of the first kind are given by follows:

\begin{corollary}
The series expansions of the generalized three parameter Bessel and modified
Bessel functions of the first kind are given by 
\begin{align*}
J_{\nu}^{(c)}(z;\rho) & =\frac{\Gamma(\mu+1)}{[\Gamma(\nu+1)]^{2}}%
\sum\limits_{m=0}^{\infty}\sum\limits_{n=0}^{\infty}(\frac{z}{2})^{\nu
-\mu+m+n}\frac{J_{\mu+m+n}(z)(-m-n)_{n}(c;\rho)_{\nu+2n}(\mu+1)_{n}}{%
(m+n)!n!(\nu+1)_{n}(\nu+1)_{2n}}, \\
I_{\nu}^{(c)}(z;\rho) & =\frac{\Gamma(\mu+1)}{[\Gamma(\nu+1)]^{2}}%
\sum\limits_{m=0}^{\infty}\sum\limits_{n=0}^{\infty}(\frac{z}{2})^{\nu
-\mu+m+n}\frac{(-1)^{n}J_{\mu+m+n}(z)(-m-n)_{n}(c;\rho)_{\nu+2n}(\mu+1)_{n}}{%
(m+n)!n!(\nu+1)_{n}(\nu+1)_{2n}},
\end{align*}
where $\mu\neq\nu$~and $\func{Re}(\nu)>-1,~\func{Re}(\mu)>-1.$
\end{corollary}

Substituting $c=1$ and \ $\rho =0$ in Corollary 2.7.1 and then using
Chu-Vandermode Theorem~\cite{G.A.R} 
\begin{equation*}
\sum\limits_{n=0}^{m}\frac{(-m)_{n}(\mu +1)_{n}}{(\nu +1)_{n}n!}=\frac{(\nu
-\mu )_{m}}{(\nu +1)_{m}},
\end{equation*}%
the expansions of the usual Bessel and modified Bessel functions can be
obtained.

\begin{corollary}
\cite{W} \ Series expansions satisfied by the usual and modified Bessel
functions are given by%
\begin{align*}
J_{\nu}(z) & =\frac{\Gamma(\mu+1)}{\Gamma(\nu-\mu)}\sum\limits_{m=0}^{%
\infty}(\frac{z}{2})^{\nu-\mu+m}\frac{\Gamma(\nu-\mu+m)}{\Gamma(\nu +m+1)m!}%
J_{\mu+m}(z),~\mu\neq\nu,~\func{Re}(\mu)>-1,~\nu-\mu \neq\{0,-1,-2,...\}, \\
I_{\nu}(z) & =\frac{\Gamma(\mu+1)}{[\Gamma(\nu+1)]}\sum\limits_{m=0}^{%
\infty}\sum\limits_{n=0}^{\infty}(\frac{z}{2})^{\nu-\mu+m+n}\frac {%
(-1)^{n}J_{\mu+m+n}(z)(-m-n)_{n}(\mu+1)_{n}}{\Gamma(m+n+1)n!(\nu+1)_{n}}%
,~\mu\neq\nu,~\func{Re}(\nu)>-1,~\func{Re}(\mu)>-1.
\end{align*}
\end{corollary}

\section{Derivative Properties, Recurrence Relation and Partial Differential
Equation of the Unified Four Parameter Bessel Function}

Derivative properties and recurrence relations of the unified four parameter
Bessel function are obtained in the following Theorems.

\begin{theorem}
A differential recurrence formula satisfied by $G_{\nu }^{(b,c)}(z;\rho )$
is given by 
\begin{equation}
\frac{\partial }{\partial z}[z^{-\nu +1}\frac{\partial }{\partial z}(z^{\nu
}G_{\nu }^{(b,c-1)}(z;\rho ))]=(c-1)G_{\nu -1}^{(b,c)}(z;\rho ).
\end{equation}
\end{theorem}

\begin{proof}
Substituting $c-1~$in (1.8) and multiplying with $z^{\nu}$ yields%
\begin{equation*}
\sum\limits_{k=0}^{\infty}\frac{(-b)^{k}(c-1;\rho)_{2k+\nu}}{\Gamma
(\nu+k+1)~\Gamma(\nu+2k+1)}\frac{z^{2k+2\nu}}{2^{2k+\nu}~k!}.
\end{equation*}
Taking derivative with respect to $z$ of $z^{\nu}G_{\nu}^{(b,c-1)}(z;\rho),$
we get%
\begin{equation*}
\frac{\partial}{\partial z}[z^{\nu}G_{\nu}^{(b,c-1)}(z;\rho)]=\sum
\limits_{k=0}^{\infty}\frac{(-b)^{k}(c-1;\rho)_{2k+\nu}}{\Gamma(\nu
+k)~\Gamma(\nu+2k+1)}\frac{z^{2k+2\nu-1}}{2^{2k+\nu-1}k!}.
\end{equation*}
Then, multiplying the last series with $z^{-\nu+1}$~and differentiating with
respect to $z$~gives%
\begin{equation*}
\frac{\partial}{\partial z}[z^{-\nu+1}\frac{\partial}{\partial z}[z^{\nu
}G_{\nu}^{(b,c-1)}(z;\rho)]]=\sum\limits_{k=0}^{\infty}\frac{%
(-b)^{k}(c-1;\rho)_{2k+\nu}~z^{2k+\nu-1}}{\Gamma(\nu+k)~%
\Gamma(v+2k)~2^{2k+v-1}k!}.
\end{equation*}
Using (1.7),~we get%
\begin{equation*}
\frac{\partial}{\partial z}[z^{-\nu+1}\frac{\partial}{\partial z}[z^{\nu
}G_{\nu}^{(b,c-1)}(z;\rho)]]=\sum\limits_{k=0}^{\infty}\frac{%
(-b)^{k}(c-1)_{\nu}(c+\nu-1;\rho)_{2k}~z^{2k+\nu-1}}{\Gamma(\nu+k)~\Gamma
(v+2k)~2^{2k+v-1}k!}.
\end{equation*}
Multiplying and dividing with the term $(c)_{\nu-1}$~of the right hand side,
we have%
\begin{equation*}
\frac{\partial}{\partial z}[z^{-\nu+1}\frac{\partial}{\partial z}[z^{\nu
}G_{\nu}^{(b,c-1)}(z;\rho)]]=\frac{(c-1)_{\nu}}{(c)_{\nu-1}}\sum
\limits_{k=0}^{\infty}\frac{(-b)^{k}(c)_{\nu-1}(c+\nu-1;\rho)_{2k}~z^{2k+%
\nu-1}}{\Gamma(\nu+k)~\Gamma(v+2k)~2^{2k+v-1}k!}.
\end{equation*}
Taking into consideration of the unified Bessel function and equation (1.7),
yields%
\begin{equation*}
\frac{\partial}{\partial z}[z^{-\nu+1}\frac{\partial}{\partial z}[z^{\nu
}G_{\nu}^{(b,c-1)}(z;\rho)]]=\frac{(c-1)_{\nu}}{(c)_{\nu-1}}%
G_{\nu-1}^{(b,c)}(z;\rho).
\end{equation*}
By (1.2), we get the result%
\begin{equation*}
\frac{\partial}{\partial z}[z^{-\nu+1}\frac{\partial}{\partial z}[z^{\nu
}G_{\nu}^{(b,c-1)}(z;\rho)]]=(c-1)G_{\nu-1}^{(b,c)}(z;\rho).
\end{equation*}
\end{proof}

The following Corollary is presented for the functions $J_{\nu}^{(c)}(z;%
\rho) $ and $I_{\nu}^{(c)}(z;\rho).$

\begin{corollary}
Recurrence relations satisfied by the generalized three parameter Bessel and
modified Bessel functions of the first kind are given by 
\begin{align*}
\frac{\partial}{\partial z}[z^{-\nu+1}\frac{\partial}{\partial z}(z^{\nu
}J_{\nu}^{(c-1)}(z;\rho))] & =(c-1)J_{\nu-1}^{(c)}(z;\rho), \\
\frac{\partial}{\partial z}[z^{-\nu+1}\frac{\partial}{\partial z}(z^{\nu
}I_{\nu}^{(c-1)}(z;\rho))] & =(c-1)I_{\nu-1}^{(c)}(z;\rho).
\end{align*}
\end{corollary}

\begin{theorem}
Recurrence relation satisfied by the unified four parameter Bessel function
is given by 
\begin{equation}
\frac{\partial }{\partial z}[z^{\nu +1}\frac{\partial }{\partial z}(z^{-\nu
}G_{\nu }^{(b,c-1)}(z;\rho ))]=-b(c-1)G_{\nu +1}^{(b,c)}(z;\rho ).
\end{equation}
\end{theorem}

\begin{proof}
Applying the similar process that is used in the previous theorem, one can
get the result.
\end{proof}

Substituting $b=1$ and $b=-1~$respectively in Theorem 3.2, the following
recurrences are obtained:

\begin{corollary}
Recurrence formulas satisfied by the generalized three parameter Bessel and
modified Bessel functions of the first kind are given by%
\begin{align*}
\frac{\partial}{\partial z}[z^{\nu+1}\frac{\partial}{\partial z}(z^{-\nu
}J_{\nu}^{(c-1)}(z;\rho))] & =-(c-1)J_{\nu+1}^{(c)}(z;\rho), \\
\frac{\partial}{\partial z}[z^{\nu+1}\frac{\partial}{\partial z}(z^{-\nu
}I_{\nu}^{(c-1)}(z;\rho))] & =(c-1)I_{\nu+1}^{(c)}(z;\rho).
\end{align*}
\end{corollary}

\begin{theorem}
For $c\neq 1,$~we have%
\begin{equation}
\frac{\partial }{\partial \rho }[G_{\nu }^{(b,c)}(z;\rho )]=-\frac{1}{c-1}%
G_{\nu }^{(b,c-1)}(z;\rho ).
\end{equation}
\end{theorem}

\begin{proof}
Taking derivative with respect to $\rho$~in (2.2), the desired result is
obtained.
\end{proof}

Taking $b=1$ and $b=-1~$in Theorem 3.3,$~$the derivatives of the generalized
three parameter Bessel and modified Bessel functions of the first kind are
calculated, respectively:

\begin{corollary}
For $c\neq1,$~we have%
\begin{align*}
\frac{\partial}{\partial\rho}[J_{\nu}^{(c)}(z;\rho)] & =-\frac{1}{c-1}%
J_{\nu}^{(c-1)}(z;\rho), \\
\frac{\partial}{\partial\rho}[I_{\nu}^{(c)}(z;\rho)] & =-\frac{1}{c-1}%
I_{\nu}^{(c-1)}(z;\rho).
\end{align*}
\end{corollary}

\begin{theorem}
The recurrence relation satisfied by the unified four parameter Bessel
function is 
\begin{align}
& (2-\nu )\frac{\partial }{\partial z}G_{\nu -1}^{(b,c-1)}(z;\rho )+z(\frac{%
\partial ^{2}}{\partial z^{2}}G_{\nu -1}^{(b,c-1)}(z;\rho )+b\frac{\partial
^{2}}{\partial z^{2}}G_{\nu +1}^{(b,c-1)}(z;\rho )) \\
& =-b(\nu +2)\frac{\partial }{\partial z}G_{\nu +1}^{(b,c-1)}(z;\rho ). 
\notag
\end{align}
\end{theorem}

\begin{proof}
Taking $\nu+1$ in place of $\nu~$in (3.1) and letting $\nu-1$~instead of $\nu$
in (3.2), then comparing these relations, one can get (3.4).
\end{proof}

$G_{\nu-1}^{(b,c-1)}(z;\rho)$ is reduced to $J_{\nu-1}^{(c-1)}(z;\rho)~$and $%
G_{\nu+1}^{(b,c-1)}(z;\rho)$~is reduced to $J_{\nu+1}^{(c-1)}(z;\rho)~$when $%
b=1.~$Besides, $G_{\nu+1}^{(b,c-1)}(z;\rho)$~is reduced to $%
I_{\nu+1}^{(c-1)}(z;\rho)$ and $G_{\nu-1}^{(b,c-1)}(z;\rho)$ is reduced to $%
I_{\nu -1}^{(c-1)}(z;\rho)$ when $b=-1.~$And so, the following Corollary is
found:

\begin{corollary}
Recurrence relations satisfied by the generalized three parameter Bessel and
modified Bessel functions of the first kind are given by 
\begin{align*}
& (2-\nu)\frac{\partial}{\partial z}J_{\nu-1}^{(c-1)}(z;\rho)+z(\frac {%
\partial^{2}}{\partial z^{2}}J_{\nu-1}^{(c-1)}(z;\rho)+\frac{\partial^{2}}{%
\partial z^{2}}J_{\nu+1}^{(c-1)}(z;\rho)) \\
& =-(\nu+2)\frac{\partial}{\partial z}J_{\nu+1}^{(c-1)}(z;\rho), \\
& (2-\nu)\frac{\partial}{\partial z}I_{\nu-1}^{(c-1)}(z;\rho)+z(\frac {%
\partial^{2}}{\partial z^{2}}I_{\nu-1}^{(c-1)}(z;\rho)-\frac{\partial^{2}}{%
\partial z^{2}}I_{\nu+1}^{(c-1)}(z;\rho)) \\
& =(\nu+2)\frac{\partial}{\partial z}I_{\nu+1}^{(c-1)}(z;\rho).
\end{align*}
\end{corollary}

To obtain the recurrence relations satisfied by the usual Bessel and
modified Bessel functions, one can substitute $c=2$ and $\rho =0$ in
Corollary 3.4.1. In that case, $J_{\nu -1}^{(c-1)}(z;\rho )$ is reduced to $%
J_{\nu -1}(z)$~and $J_{\nu +1}^{(c-1)}(z;\rho )$ is reduced to $J_{\nu
+1}(z).$ Under the similar substitutions, $I_{\nu -1}^{(c-1)}(z;\rho )$~is
reduced to $I_{\nu -1}(z)$ and $I_{\nu +1}^{(c-1)}(z;\rho )$ is reduced to $%
I_{\nu +1}(z).~$Besides, considering the equations (1.1) and (1.8) in \cite%
{R}~(page 111), the recurrence relation of the usual Bessel function is
obtained. The similar process can be applied to find out the recurrence
relation of the modified Bessel function.

\begin{corollary}
\ Recurrence relations satisfied by the usual Bessel and modified Bessel
functions are given by 
\begin{align*}
& (2-\nu)\frac{d}{dz}J_{\nu-1}(z)+z(\frac{d^{2}}{dz^{2}}J_{\nu-1}(z)+\frac{%
d^{2}}{dz^{2}}J_{\nu+1}(z)) \\
& =-(\nu+2)\frac{d}{dz}J_{\nu+1}(z), \\
& (2-\nu)\frac{d}{dz}I_{\nu-1}(z)+z(\frac{d^{2}}{dz^{2}}I_{\nu-1}(z)-\frac{%
d^{2}}{dz^{2}}I_{\nu+1}(z)) \\
& =(\nu+2)\frac{d}{dz}I_{\nu+1}(z).
\end{align*}
\end{corollary}

\begin{proof}
Taking derivative with respect to $z$ in the following recurrence formula 
\cite{R}%
\begin{equation*}
2\nu~J_{\nu}(z)=z[J_{\nu-1}(z)+J_{\nu+1}(z)]
\end{equation*}
it is directly obtained that 
\begin{equation*}
2\nu\frac{d}{dz}J_{\nu}(z)=J_{\nu-1}(z)+J_{\nu+1}(z)+z\frac{d}{dz}J_{\nu
-1}(z)+z\frac{d}{dz}J_{\nu+1}(z).
\end{equation*}
Now, inserting the derivative of $\ J_{\nu}(z),$ 
\begin{equation*}
\frac{d}{dz}J_{\nu}(z)=\frac{1}{2}[J_{\nu-1}(z)-J_{\nu+1}(z)]
\end{equation*}
in the above recurrence formula and differentiating with respect to $z$,~
result is obtained as 
\begin{equation*}
(2-\nu)\frac{d}{dz}J_{\nu-1}(z)+z(\frac{d^{2}}{dz^{2}}J_{\nu-1}(z)+\frac {%
d^{2}}{dz^{2}}J_{\nu+1}(z))=-(\nu+2)\frac{d}{dz}J_{\nu+1}(z).
\end{equation*}
Proof of the recurrence formula satisfied by modified Bessel function can be
obtained as a similar way by means of the recurrence formulas given in
Section 1.
\end{proof}

In the following Theorem, the partial differential equation satisfied by the
unified four parameter Bessel function is obtained.

\begin{theorem}
Partial differential equation satisfied by the unified four parameter Bessel
function is given by 
\begin{align}
& \frac{\partial ^{6}}{\partial z^{4}\partial \rho ^{2}}G_{\nu
}^{(b,c)}(z;\rho )+\frac{5}{z}\frac{\partial ^{5}}{\partial z^{3}\partial
\rho ^{2}}G_{\nu }^{(b,c)}(z;\rho )+\frac{(\nu -2)^{2}}{z^{2}}\frac{\partial
^{4}}{\partial z^{2}\partial \rho ^{2}}G_{\nu }^{(b,c)}(z;\rho ) \\
& =-\frac{b}{z^{2}}G_{\nu }^{(b,c)}(z;\rho ).  \notag
\end{align}
\end{theorem}

\begin{proof}
Taking $\nu-1$~in place of $\nu$~in (3.2) yields%
\begin{equation}
\frac{\partial}{\partial z}[z^{\nu}\frac{\partial}{\partial z}(z^{-\nu
+1}G_{\nu-1}^{(b,c-1)}(z;\rho))]=-b(c-1)G_{\nu}^{(b,c)}(z;\rho)  \tag{3.6}
\end{equation}
and letting $c-1$ in place of $c$ in (3.1),~we get%
\begin{equation}
\frac{1}{(c-2)}\frac{\partial}{\partial z}[z^{-\nu+1}\frac{\partial}{%
\partial z}(z^{\nu}G_{\nu}^{(b,c-2)}(z;\rho))]=G_{\nu-1}^{(b,c-1)}(z;\rho). 
\tag{3.7}
\end{equation}
In equation (3.3), substituting $c-1$ in place of $c,$~we get%
\begin{equation}
G_{\nu}^{(b,c-2)}(z;\rho)=-(c-2)\frac{\partial}{\partial\rho}%
G_{\nu}^{(b,c-1)}(z;\rho).  \tag{3.8}
\end{equation}
Using (3.3) for $G_{\nu}^{(b,c-1)}(z;\rho),$~we get 
\begin{equation}
G_{\nu}^{(b,c-2)}(z;\rho)  =-(c-2)\frac{\partial}{\partial\rho}[-(c-1)\frac{%
\partial}{\partial\rho}G_{\nu}^{(b,c)}(z;\rho)]  \notag \\
 =(c-1)(c-2)\frac{\partial^{2}}{\partial\rho^{2}}G_{\nu}^{(b,c)}(z;\rho). 
\tag{3.9}
\end{equation}
Obtaining $G_{\nu-1}^{(b,c-1)}(z;\rho)$ from (3.1) and inserting that into
(3.6), we get%
\begin{equation}
 \frac{\partial}{\partial z}[z^{\nu}\frac{\partial}{\partial z}(z^{-\nu +1}%
\frac{1}{c-2}\frac{\partial}{\partial z}[z^{-\nu+1}\frac{\partial}{\partial z%
}(z^{\nu}G_{\nu}^{(b,c-2)}(z;\rho))])]  \tag{3.10} \\
 =-b(c-1)G_{\nu}^{(b,c)}(z;\rho).  \notag
\end{equation}
Plugging (3.9) into (3.10), we have%
\begin{align}
& \frac{\partial}{\partial z}[z^{\nu}\frac{\partial}{\partial z}(z^{-\nu +1}%
\frac{1}{c-2}\frac{\partial}{\partial z}[z^{-\nu+1}\frac{\partial}{\partial z%
}(z^{\nu}(c-1)(c-2)\frac{\partial^{2}}{\partial\rho^{2}}G_{\nu}^{(b,c)}(z;%
\rho))])]  \tag{3.11} \\
& =-b(c-1)G_{\nu}^{(b,c)}(z;\rho).  \notag
\end{align}
Applying the derivatives, we get%
\begin{equation*}
\frac{\partial^{6}}{\partial z^{4}\partial\rho^{2}}G_{\nu}^{(b,c)}(z;\rho)+%
\frac{5}{z}\frac{\partial^{5}}{\partial z^{3}\partial\rho^{2}}G_{\nu
}^{(b,c)}(z;\rho)+\frac{(\nu-2)^{2}}{z^{2}}\frac{\partial^{4}}{\partial
z^{2}\partial\rho^{2}}G_{\nu}^{(b,c)}(z;\rho)+\frac{b}{z^{2}}%
G_{\nu}^{(b,c)}(z;\rho)=0.
\end{equation*}
\end{proof}

Taking $b=1,$ $b=-1$ and $b=1,~c=1$ and \ $b=-1$~and $c=1~$in Theorem 3.5,
the following Corollary can be presented:

\begin{corollary}
Partial differential equations satisfied by the generalized three parameter
Bessel function and modified Bessel function of the first kind and the
functions $J_{\nu}^{(1,1)}(z;\rho)$ and $I_{\nu}^{(-1,1)}(z;\rho)$ are given
by%
\begin{align*}
\frac{\partial^{6}}{\partial z^{4}\partial\rho^{2}}J_{\nu}^{(c)}(z;\rho )+%
\frac{5}{z}\frac{\partial^{5}}{\partial z^{3}\partial\rho^{2}}%
J_{\nu}^{(c)}(z;\rho)+\frac{(\nu-2)^{2}}{z^{2}}\frac{\partial^{4}}{\partial
z^{2}\partial\rho^{2}}J_{\nu}^{(c)}(z;\rho)+\frac{1}{z^{2}}%
J_{\nu}^{(c)}(z;\rho) & =0, \\
\frac{\partial^{6}}{\partial z^{4}\partial\rho^{2}}I_{\nu}^{(c)}(z;\rho )+%
\frac{5}{z}\frac{\partial^{5}}{\partial z^{3}\partial\rho^{2}}%
I_{\nu}^{(c)}(z;\rho)+\frac{(\nu-2)^{2}}{z^{2}}\frac{\partial^{4}}{\partial
z^{2}\partial\rho^{2}}I_{\nu}^{(c)}(z;\rho)-\frac{1}{z^{2}}%
I_{\nu}^{(c)}(z;\rho) & =0, \\
\frac{\partial^{6}}{\partial z^{4}\partial\rho^{2}}J_{\nu}^{(1,1)}(z;\rho)+%
\frac{5}{z}\frac{\partial^{5}}{\partial z^{3}\partial\rho^{2}}J_{\nu
}^{(1,1)}(z;\rho)+\frac{(\nu-2)^{2}}{z^{2}}\frac{\partial^{4}}{\partial
z^{2}\partial\rho^{2}}J_{\nu}^{(1,1)}(z;\rho)+\frac{1}{z^{2}}%
J_{\nu}^{(1,1)}(z;\rho) & =0, \\
\frac{\partial^{6}}{\partial z^{4}\partial\rho^{2}}I_{\nu}^{(-1,1)}(z;\rho)+%
\frac{5}{z}\frac{\partial^{5}}{\partial z^{3}\partial\rho^{2}}I_{\nu
}^{(-1,1)}(z;\rho)+\frac{(\nu-2)^{2}}{z^{2}}\frac{\partial^{4}}{\partial
z^{2}\partial\rho^{2}}I_{\nu}^{(-1,1)}(z;\rho)-\frac{1}{z^{2}}%
I_{\nu}^{(-1,1)}(z;\rho) & =0,
\end{align*}
respectively.
\end{corollary}

\begin{theorem}
Let $\alpha \in 
%TCIMACRO{\U{211d} }%
%BeginExpansion
\mathbb{R}
%EndExpansion
$ and $u=G_{\nu }^{(b,c)}(\alpha z;\rho ).$ Then the partial differential
equation satisfied by $u$ is given by 
\begin{equation}
z^{4}\frac{\partial ^{6}}{\partial z^{4}\partial \rho ^{2}}u+5z^{3}\frac{%
\partial ^{5}}{\partial z^{3}\partial \rho ^{2}}u+(\nu -2)^{2}z^{2}\frac{%
\partial ^{4}}{\partial z^{2}\partial \rho ^{2}}u+b\alpha ^{2}z^{2}u=0. 
\tag{3.12}
\end{equation}
\end{theorem}

\begin{proof}
Taking into consideration of the following partial derivatives 
\begin{equation*}
\frac{\partial u}{\partial z}=\alpha\frac{\partial G_{\nu}^{(b,c)}(\alpha
z;\rho)}{\partial z},~\frac{\partial^{2}u}{\partial u^{2}}=\alpha^{2}\frac{%
\partial^{2}G_{\nu}^{(b,c)}(\alpha z;\rho)}{\partial z^{2}},~\frac{%
\partial^{3}u}{\partial z^{3}}=\alpha^{3}\frac{\partial^{3}G_{\nu
}^{(b,c)}(\alpha z;\rho)}{\partial z^{3}},~\frac{\partial^{4}u}{\partial
z^{4}}=\alpha^{4}\frac{\partial^{4}G_{\nu}^{(b,c)}(\alpha z;\rho)}{\partial
z^{4}},
\end{equation*}
in the partial differential equation (3.5), we get 
\begin{align*}
& (\alpha z)^{4}\frac{\partial^{6}}{\partial z^{4}\partial\rho^{2}}G_{\nu
}^{(b,c)}(\alpha z;\rho)+5(\alpha z)^{3}\frac{\partial^{5}}{\partial
z^{3}\partial\rho^{2}}G_{\nu}^{(b,c)}(\alpha z;\rho)+(\nu-2)^{2}(\alpha
z)^{2}\frac{\partial^{4}}{\partial z^{2}\partial\rho^{2}}G_{\nu}^{(b,c)}(%
\alpha z;\rho) \\
& =-b(\alpha z)^{2}G_{\nu}^{(b,c)}(\alpha z;\rho).
\end{align*}
Therefore, $u=G_{\nu}^{(b,c)}(\alpha z;\rho)$ is a solution of (3.12).
\end{proof}

\section{ Mellin Transform of Products of Two Unified Four Parameter Bessel
Functions}

In this Section, the Mellin transform involving products of two unified four
parameter Bessel functions are obtained. To obtain this, integral
representation of the product of the unified four parameter Bessel functions
should be calculated.

\begin{theorem}
$~$Let $\alpha,\beta\in%
%TCIMACRO{\U{211d} }%
%BeginExpansion
\mathbb{R}
%EndExpansion
.~$Double integral representation \ satisfied by the multiplication of the
unified four parameter Bessel functions is given by 
\begin{align*}
G_{\nu}^{(b,c)}(\alpha z;\rho)G_{w}^{(b,c)}(\beta z;\rho) & =\frac{(\alpha
z)^{\nu}(\beta z)^{w}}{2^{\nu+w}[\Gamma(c)]^{2}~[\Gamma(\nu+1)]^{2}~[%
\Gamma(w+1)]^{2}}\int\limits_{0}^{\infty}\int\limits_{0}^{\infty}t^{c+\nu
-1}u^{c+w-1}e^{-t-u-\frac{\rho}{t}-\frac{\rho}{u}} \\
& \times~_{0}F_{3}(-;\nu+1,\frac{\nu+1}{2},\frac{\nu+2}{2};\frac{-b\alpha
^{2}z^{2}t^{2}}{16}) \\
& ~\times~_{0}F_{3}(-;w+1,\frac{w+1}{2},\frac{w+2}{2};\frac{%
-b\beta^{2}z^{2}u^{2}}{16})dtdu,
\end{align*}
where $\func{Re}(c)>0,~\func{Re}(\nu)>-1$ and $\func{Re}(w)>-1.$
\end{theorem}

\begin{proof}
Considering series definition of the unified four parameter Bessel function,
it can be written as 
\begin{equation*}
G_{\nu}^{(b,c)}(\alpha z;\rho)G_{w}^{(b,c)}(\beta
z;\rho)=\sum\limits_{k,l=0}^{\infty}\frac{(-b)^{k+l}(c;\rho)_{2k+\nu}(c;%
\rho)_{2l+w}(\alpha z)^{2k+\nu }(\beta z)^{2l+w}}{k!l!~\Gamma(k+\nu+1)~%
\Gamma(l+w+1)~\Gamma(2k+\nu +1)~\Gamma(2l+w+1)~2^{2k+\nu}2^{2l+w}}.
\end{equation*}
Using the generalized Pochhammer function for $(c;\rho)_{2k+\nu}$ and $%
(c;\rho)_{2l+w}$, it is clear that%
\begin{align*}
G_{\nu}^{(b,c)}(\alpha z;\rho)G_{w}^{(b,c)}(\beta z;\rho) & =\sum
\limits_{k,l=0}^{\infty}\frac{(-b)^{k+l}(\alpha z)^{2k+\nu}(\beta z)^{2l+w}}{%
k!l!~\Gamma(k+\nu+1)~\Gamma(l+w+1)~\Gamma(2k+\nu+1)~\Gamma(2l+w+1)~2^{2k+\nu
}2^{2l+w}} \\
& \times~\frac{1}{[\Gamma(c)]^{2}}~\int\limits_{0}^{\infty}\int
\limits_{0}^{\infty}t^{c+2k+\nu-1}u^{c+2l+w-1}e^{-t-u-\frac{\rho}{t}-\frac{%
\rho}{u}}dtdu.
\end{align*}
Now, using (1.2) yields%
\begin{align*}
G_{\nu}^{(b,c)}(\alpha z;\rho)G_{w}^{(b,c)}(\beta z;\rho) & =\frac {\Gamma(%
\frac{\nu+1}{2})\Gamma(\frac{\nu+2}{2})\Gamma(\frac{w+1}{2})\Gamma(\frac{w+2%
}{2})(\frac{\alpha z}{2})^{\nu}(\frac{\beta z}{2})^{w}}{\Gamma(\nu+1)%
\Gamma(w+1)} \\
& \times\sum\limits_{k,l=0}^{\infty}\frac{(\frac{-bz^{2}\alpha^{2}}{16})^{k}(%
\frac{-bz^{2}\beta^{2}}{16})^{l}}{\Gamma(k+\nu+1)\Gamma(l+w+1)\Gamma (\frac{%
\nu+1}{2}+k)\Gamma(\frac{\nu+2}{2}+k)\Gamma(\frac{w+1}{2}+l)\Gamma(\frac{w+2%
}{2}+l)k!l!} \\
& \times\frac{1}{[\Gamma(c)]^{2}}~\int\limits_{0}^{\infty}\int\limits_{0}^{%
\infty}t^{c+2k+\nu-1}u^{c+2l+w-1}e^{-t-u-\frac{\rho}{t}-\frac{\rho}{u}}dtdu.
\end{align*}
Recalling the fact that given double power series%
\begin{equation*}
\sum\limits_{m,n=0}^{\infty}\frac{\prod\limits_{j=1}^{A}\Gamma(a_{j}+m%
\vartheta_{j}+n\varphi_{j})\prod\limits_{j=1}^{B}\Gamma(b_{j}+m\psi
_{j})\prod\limits_{j=1}^{B^{^{\prime}}}\Gamma(b_{j}^{^{\prime}}+n\psi
_{j}^{^{\prime}})}{\prod\limits_{j=1}^{C}\Gamma(c_{j}+m\delta_{j}+n%
\varepsilon_{j})\prod\limits_{j=1}^{D}\Gamma(d_{j}+m\eta_{j})\prod
\limits_{j=1}^{D^{\prime}}\Gamma(d_{j}^{^{\prime}}+n\eta_{j}^{^{\prime}})}%
\frac{x^{m}y^{n}}{m!n!}
\end{equation*}
converges absolutely for all complex $x$ and $y$ when 
\begin{align*}
1+\sum\limits_{j=1}^{C}\delta_{j}+\sum\limits_{j=1}^{D}\eta_{j}-\sum
\limits_{j=1}^{A}\vartheta_{j}-\sum\limits_{j=1}^{B}\psi_{j} & >0, \\
1+\sum\limits_{j=1}^{C}\varepsilon_{j}+\sum\limits_{j=1}^{D^{\prime}}\eta
_{j}^{\prime}-\sum\limits_{j=1}^{A}\varphi_{j}-\sum\limits_{j=1}^{B^{%
\prime}}\psi_{j}^{^{\prime}} & >0,
\end{align*}
(see \cite{H.S.D} in (3.10)). Comparing the given power series with the
double series of the right hand side of $G_{\nu}^{(b,c)}(\alpha
z;\rho)G_{w}^{(b,c)}(\beta z;\rho)$ and replacing the order of the integrals
and summations where $\delta_{j}=0,$ $\sum\limits_{j=1}^{3}\eta_{j}=3,~%
\vartheta _{j}=0,~\psi_{j}~=0~$and $\varepsilon_{j}=0,~\sum\limits_{j=1}^{3}%
\eta _{j}^{^{\prime}}=3,~\varphi_{j}=0,~\psi_{j}^{^{\prime}}=0,~x=\frac {%
-bz^{2}\alpha^{2}}{16},~y=\frac{-bz^{2}\beta^{2}}{16},~$the following is
obtained 
\begin{align*}
G_{\nu}^{(b,c)}(\alpha z;\rho)G_{w}^{(b,c)}(\beta z;\rho) & =\frac {\Gamma(%
\frac{\nu+1}{2})\Gamma(\frac{\nu+2}{2})\Gamma(\frac{w+1}{2})\Gamma(\frac{w+2%
}{2})(\frac{\alpha z}{2})^{\nu}(\frac{\beta z}{2})^{w}}{[\Gamma(c)]^{2}%
\Gamma(\nu+1)\Gamma(w+1)} \\
&
\times\int\limits_{0}^{\infty}\int\limits_{0}^{\infty}\{\sum%
\limits_{k,l=0}^{\infty}\frac{(\frac{-bz^{2}\alpha^{2}}{16})^{k}(\frac{%
-bz^{2}\beta^{2}}{16})^{l}t^{2k}u^{2l}}{\Gamma(k+\nu+1)\Gamma(l+w+1)\Gamma(%
\frac{\nu+1}{2}+k)\Gamma(\frac{\nu+2}{2}+k)\Gamma(\frac{w+1}{2}+l)\Gamma(%
\frac{w+2}{2}+l)k!l!}\} \\
& \times~t^{c+\nu-1}u^{c+w-1}e^{-t-u-\frac{\rho}{t}-\frac{\rho}{u}}dtdu.
\end{align*}
After some calculations in series, summations can be written as%
\begin{align*}
& \frac{\Gamma(\frac{\nu+1}{2})\Gamma(\frac{\nu+2}{2})\Gamma(\frac{w+1}{2}%
)\Gamma(\frac{w+2}{2})(\frac{\alpha z}{2})^{\nu}(\frac{\beta z}{2})^{w}}{%
\Gamma(\nu+1)\Gamma(w+1)} \\
\times & \sum\limits_{k,l=0}^{\infty}\frac{(\frac{-bz^{2}\alpha^{2}}{16}%
)^{k}(\frac{-bz^{2}\beta^{2}}{16})^{l}t^{2k}u^{2l}}{\Gamma(k+\nu
+1)\Gamma(l+w+1)\Gamma(\frac{\nu+1}{2}+k)\Gamma(\frac{\nu+2}{2}+k)\Gamma (%
\frac{w+1}{2}+l)\Gamma(\frac{w+2}{2}+l)k!l!} \\
& =\frac{(\alpha z)^{\nu}(\beta z)^{w}}{2^{\nu+w}[\Gamma(\nu+1)]^{2}[%
\Gamma(w+1)]^{2}}~_{0}F_{3}(-;\nu+1,\frac{\nu+1}{2},\frac{\nu+2}{2};\frac{%
-b\alpha^{2}z^{2}t^{2}}{16}) \\
& \times~_{0}F_{3}(-;w+1,\frac{w+1}{2},\frac{w+2}{2};\frac{%
-b\beta^{2}z^{2}u^{2}}{16}).
\end{align*}
Substituting the last relation inside of the integrals and using usual
Pochhammer function expansions, the desired result is obtained as%
\begin{align*}
G_{\nu}^{(b,c)}(\alpha z;\rho)G_{w}^{(b,c)}(\beta z;\rho) & =\frac{(\alpha
z)^{\nu}(\beta z)^{w}}{2^{\nu+w}[\Gamma(c)]^{2}[\Gamma(\nu+1)]^{2}[%
\Gamma(w+1)]^{2}}\int\limits_{0}^{\infty}\int\limits_{0}^{\infty}t^{c+\nu
-1}u^{c+w-1}e^{-t-u-\frac{\rho}{t}-\frac{\rho}{u}} \\
& \times~_{0}F_{3}(-;\nu+1,\frac{\nu+1}{2},\frac{\nu+2}{2};\frac{-b\alpha
^{2}z^{2}t^{2}}{16}) \\
& \times~_{0}F_{3}(-;w+1,\frac{w+1}{2},\frac{w+2}{2};\frac{%
-b\beta^{2}z^{2}u^{2}}{16})dtdu.
\end{align*}
\end{proof}

Letting $b=1,$ $b=-1,~b=1,~c=1,~\rho =0~$\ and $b=-1,~c=1,~\rho =0$ in
Theorem 4.1, the following Corollaries are obtained:

\begin{corollary}
$~$The integral representations satisfied by the functions~ $%
J_{\nu}^{(c)}(\alpha z;\rho)$,~$J_{w}^{(c)}(\beta z;\rho)$ are given by 
\begin{align*}
J_{\nu}^{(c)}(\alpha z;\rho)J_{w}^{(c)}(\beta z;\rho) & =\frac{(\alpha
z)^{\nu}(\beta z)^{w}}{2^{\nu+w}[\Gamma(c)]^{2}~[\Gamma(\nu+1)]^{2}[%
\Gamma(w+1)]^{2}}\int\limits_{0}^{\infty}\int\limits_{0}^{\infty}t^{c+\nu
-1}u^{c+w-1}e^{-t-u-\frac{\rho}{t}-\frac{\rho}{u}} \\
& \times~_{0}F_{3}(-;\nu+1,\frac{\nu+1}{2},\frac{\nu+2}{2};\frac{-\alpha
^{2}z^{2}t^{2}}{16}) \\
& \times~_{0}F_{3}(-;w+1,\frac{w+1}{2},\frac{w+2}{2};\frac{%
-\beta^{2}z^{2}u^{2}}{16})dtdu, \\
I_{\nu}^{(c)}(\alpha z;\rho)I_{w}^{(c)}(\beta z;\rho) & =\frac{(\alpha
z)^{\nu}(\beta z)^{w}}{2^{\nu+w}[\Gamma(c)]^{2}[\Gamma(\nu+1)]^{2}[%
\Gamma(w+1)]^{2}}\int\limits_{0}^{\infty}\int\limits_{0}^{\infty}t^{c+\nu
-1}u^{c+w-1}e^{-t-u-\frac{\rho}{t}-\frac{\rho}{u}} \\
& \times~_{0}F_{3}(-;\nu+1,\frac{\nu+1}{2},\frac{\nu+2}{2};\frac{\alpha
^{2}z^{2}t^{2}}{16}) \\
& \times~_{0}F_{3}(-;w+1,\frac{w+1}{2},\frac{w+2}{2};\frac{%
\beta^{2}z^{2}u^{2}}{16})dtdu,
\end{align*}
where \ $\func{Re}(c)>0,~\func{Re}(\nu)>-1$ and $\func{Re}(w)>-1.$
\end{corollary}

\begin{corollary}
$~$Integral representations satisfied by the usual Bessel and modified
Bessel functions are given by%
\begin{align*}
J_{\nu}(\alpha z)J_{w}(\beta z) & =\frac{(\alpha z)^{\nu}(\beta z)^{w}}{%
2^{\nu+w}[\Gamma(\nu+1)]^{2}[\Gamma(w+1)]^{2}}\int\limits_{0}^{\infty}\int%
\limits_{0}^{\infty}t^{\nu}u^{w}e^{-t-u} \\
& \times~_{0}F_{3}(-;\nu+1,\frac{\nu+1}{2},\frac{\nu+2}{2};\frac{-\alpha
^{2}z^{2}t^{2}}{16}) \\
& \times~_{0}F_{3}(-;w+1,\frac{w+1}{2},\frac{w+2}{2};\frac{%
-\beta^{2}z^{2}u^{2}}{16})dtdu, \\
I_{\nu}(\alpha z)I_{w}(\beta z) & =\frac{(\alpha z)^{\nu}(\beta z)^{w}}{%
2^{\nu+w}[\Gamma(\nu+1)]^{2}[\Gamma(w+1)]^{2}}\int\limits_{0}^{\infty}\int%
\limits_{0}^{\infty}t^{\nu}u^{w}e^{-t-u} \\
& \times~_{0}F_{3}(-;\nu+1,\frac{\nu+1}{2},\frac{\nu+2}{2};\frac{\alpha
^{2}z^{2}t^{2}}{16}) \\
& \times~_{0}F_{3}(-;w+1,\frac{w+1}{2},\frac{w+2}{2};\frac{%
\beta^{2}z^{2}u^{2}}{16})dtdu,
\end{align*}
where \ $\func{Re}(\nu)>-1$ and $\func{Re}(w)>-1.$
\end{corollary}

In the following Theorem, the Mellin transform of the $zG_{\nu}^{(b,c)}(%
\alpha z;\rho)G_{w}^{(b,c)}(\beta z;\rho)$ is calculated$.$

\begin{theorem}
$~$The Mellin transform of $ze^{-z}G_{\nu}^{(b,c)}(\alpha
z;\rho)G_{w}^{(b,c)}(\beta z;\rho)~$ is given by%
\begin{align*}
& \mathcal{M}\{ze^{-z}G_{\nu}^{(b,c)}(\alpha z;\rho)G_{w}^{(b,c)}(\beta
z;\rho);s\} \\
& =\frac{\alpha^{\nu}\beta^{w}}{2^{\nu+w}[\Gamma(\nu+1)]^{2}[\Gamma
(w+1)]^{2}}\sum\limits_{k,l=0}^{\infty}\frac{(-b)^{k+l}\alpha^{2k}\beta
^{2l}(c;\rho)_{\nu+2k}~(c;\rho)_{w+2l}~\Gamma(s+\nu+w+2k+2l+1)}{(\nu
+1)_{k}~(w+1)_{l}~(\frac{\nu+1}{2})_{k}~(\frac{w+1}{2})_{l}~(\frac{\nu+2}{2}%
)_{k}~(\frac{w+2}{2})_{l}~k!l!2^{4k+4l}},
\end{align*}
where \ $\func{Re}(\nu)>-1$ and $\func{Re}(w)>-1.$
\end{theorem}

\begin{proof}
Mellin transform can be written%
\begin{equation*}
\mathcal{M}\{ze^{-z}G_{\nu}^{(b,c)}(\alpha z;\rho)G_{w}^{(b,c)}(\beta
z;\rho);s\}=\int\limits_{0}^{\infty}z^{s-1}ze^{-z}G_{\nu}^{(b,c)}(\alpha
z;\rho)G_{w}^{(b,c)}(\beta z;\rho)dz.
\end{equation*}
Inserting the integral representation of the $G_{\nu}^{(b,c)}(\alpha
z;\rho)G_{w}^{(b,c)}(\beta z;\rho)$~and expanding the hypergeometric series
and using the expansion of the generalized Pochhammer functions$,$ we have%
\begin{align*}
& \mathcal{M}\{ze^{-z}G_{\nu}^{(b,c)}(\alpha z;\rho)G_{w}^{(b,c)}(\beta
z;\rho);s\} \\
& =\frac{\alpha^{\nu}\beta^{w}}{2^{\nu+w}[\Gamma(\nu+1)]^{2}[\Gamma
(w+1)]^{2}[\Gamma(c)]^{2}}\int\limits_{0}^{\infty}\int\limits_{0}^{\infty}%
\int\limits_{0}^{\infty}e^{-z}z^{s+\nu+w}t^{c+\nu-1}u^{c+w-1}e^{-t-u-\frac {%
\rho}{t}-\frac{\rho}{u}} \\
& \times~\sum\limits_{k,l=0}^{\infty}\frac{(-b)^{k+l}\alpha^{2k}\beta
^{2l}t^{2k}u^{2l}}{(\nu+1)_{k}~(w+1)_{l}~(\frac{\nu+1}{2})_{k}~(\frac{w+1}{2}%
)_{l}~(\frac{\nu+2}{2})_{k}~(\frac{w+2}{2})_{l}~k!l!}\frac{z^{2k+2l}}{%
16^{k+l}}dtdudz.
\end{align*}
Replacing the order of integrals and summations, we get%
\begin{align*}
& \mathcal{M}\{ze^{-z}G_{\nu}^{(b,c)}(\alpha z;\rho)G_{w}^{(b,c)}(\beta
z;\rho);s\} \\
& =\frac{\alpha^{\nu}\beta^{w}}{2^{\nu+w}[\Gamma(\nu+1)]^{2}[\Gamma
(w+1)]^{2}[\Gamma(c)]^{2}}\sum\limits_{k,l=0}^{\infty}\{\int\limits_{0}^{%
\infty}\int\limits_{0}^{\infty}\int\limits_{0}^{\infty}t^{c+\nu +2k-1}e^{-t-%
\frac{\rho}{t}}u^{2l+c+w-1}e^{-u-\frac{\rho}{u}}z^{s+\nu
+w+2k+2l}e^{-z}dtdudz\} \\
& \times\frac{(-b)^{k+l}\alpha^{2k}\beta^{2l}}{(\nu+1)_{k}~(w+1)_{l}~(\frac{%
\nu+1}{2})_{k}~(\frac{w+1}{2})_{l}~(\frac{\nu+2}{2})_{k}~(\frac {w+2}{2}%
)_{l}~k!l!(16)^{k+l}}.
\end{align*}
Now, computing the integrals in terms of the generalized Pochhammer
functions and Gamma function, we have%
\begin{align*}
& \mathcal{M}\{ze^{-z}G_{\nu}^{(b,c)}(\alpha z;\rho)G_{w}^{(b,c)}(\beta
z;\rho);s\} \\
& =\frac{\alpha^{\nu}\beta^{w}}{2^{\nu+w}[\Gamma(\nu+1)]^{2}[\Gamma
(w+1)]^{2}}\sum\limits_{k,l=0}^{\infty}\frac{(-b)^{k+l}\alpha^{2k}\beta
^{2l}~(c;\rho)_{\nu+2k}~(c;\rho)_{w+2l}~\Gamma(s+\nu+w+2k+2l+1)}{(\nu
+1)_{k}~(w+1)_{l}~(\frac{\nu+1}{2})_{k}~(\frac{w+1}{2})_{l}~(\frac{\nu+2}{2}%
)_{k}~(\frac{w+2}{2})_{l}~k!l!2^{4k+4l}}.
\end{align*}
\end{proof}

Taking $s=1$~in Theorem 4.2, the following integral representation formula
is obtained.

\begin{corollary}
$~$Infinite integral involving of the unified four parameter Bessel
functions is given by 
\begin{align*}
& \int\limits_{0}^{\infty}ze^{-z}G_{\nu}^{(b,c)}(\alpha
z;\rho)G_{w}^{(b,c)}(\beta z;\rho)dz \\
& =\frac{\alpha^{\nu}\beta^{w}}{2^{\nu+w}[\Gamma(\nu+1)]^{2}[\Gamma
(w+1)]^{2}}\sum\limits_{k,l=0}^{\infty}\frac{(-b)^{k+l}\alpha^{2k}\beta
^{2l}~(c;\rho)_{\nu+2k}~(c;\rho)_{w+2l}~\Gamma(\nu+w+2k+2l+2)}{(\nu
+1)_{k}~(w+1)_{l}~(\frac{\nu+1}{2})_{k}~(\frac{w+1}{2})_{l}~(\frac{\nu+2}{2}%
)_{k}~(\frac{w+2}{2})_{l}~k!l!2^{4k+4l}},
\end{align*}
where $\func{Re}(\nu)>-1$ and $\func{Re}(w)>-1.$
\end{corollary}

Taking $b=1,$ $b=-1,$ $b=1,~c=1,~\rho =0$ and $b=-1,~c=1,~\rho =0$ in
Corollary 4.2.1, the following integral representations are obtained:

\begin{corollary}
$~$The integral representations involving the multiplication of the
generalized three parameter Bessel and modified Bessel functions of the
first kind are given by%
\begin{align*}
& \int\limits_{0}^{\infty}ze^{-z}J_{\nu}^{(c)}(\alpha
z;\rho)J_{w}^{(c)}(\beta z;\rho)dz \\
& =\frac{\alpha^{\nu}\beta^{w}}{2^{\nu+w}[\Gamma(\nu+1)]^{2}[\Gamma
(w+1)]^{2}}\sum\limits_{k,l=0}^{\infty}\frac{(-1)^{k+l}\alpha^{2k}\beta
^{2l}~(c;\rho)_{\nu+2k}~(c;\rho)_{w+2l}~\Gamma(\nu+w+2k+2l+2)}{(\nu
+1)_{k}~(w+1)_{l}~(\frac{\nu+1}{2})_{k}~(\frac{w+1}{2})_{l}~(\frac{\nu+2}{2}%
)_{k}~(\frac{w+2}{2})_{l}~k!l!2^{4k+4l}}, \\
& \int\limits_{0}^{\infty}ze^{-z}I_{\nu}^{(c)}(\alpha
z;\rho)I_{w}^{(c)}(\beta z;\rho)dz \\
& =\frac{\alpha^{\nu}\beta^{w}}{2^{\nu+w}[\Gamma(\nu+1)]^{2}[\Gamma
(w+1)]^{2}}\sum\limits_{k,l=0}^{\infty}\frac{\alpha^{2k}\beta^{2l}~(c;%
\rho)_{\nu+2k}~(c;\rho)_{w+2l}~\Gamma(\nu+w+2k+2l+2)}{(\nu+1)_{k}~(w+1)_{l}~(%
\frac{\nu+1}{2})_{k}~(\frac{w+1}{2})_{l}~(\frac{\nu+2}{2})_{k}~(\frac{w+2}{2}%
)_{l}~k!l!2^{4k+4l}},
\end{align*}
where $\func{Re}(\nu)>-1$ and $\func{Re}(w)>-1.$
\end{corollary}

\begin{corollary}
Integrals involving $J_{\nu}^{(1,1)}(\alpha z;\rho)$, $J_{w}^{(1,1)}(\beta
z;\rho),~I_{\nu}^{(-1,1)}(\alpha z;\rho)~$and $I_{w}^{(-1,1)}(\beta z;\rho)$
are given by%
\begin{align*}
& \int\limits_{0}^{\infty}ze^{-z}J_{\nu}^{(1,1)}(\alpha
z;\rho)J_{w}^{(1,1)}(\beta z;\rho)dz \\
& =\frac{\alpha^{\nu}\beta^{w}}{2^{\nu+w}[\Gamma(\nu+1)]^{2}[\Gamma
(w+1)]^{2}}\sum\limits_{k,l=0}^{\infty}\frac{(-1)^{k+l}\alpha^{2k}\beta
^{2l}~(1;\rho)_{\nu+2k}~(1;\rho)_{w+2l}~\Gamma(\nu+w+2k+2l+2)}{(\nu
+1)_{k}~(w+1)_{l}~(\frac{\nu+1}{2})_{k}~(\frac{w+1}{2})_{l}~(\frac{\nu+2}{2}%
)_{k}~(\frac{w+2}{2})_{l}~k!l!2^{4k+4l}}, \\
& \int\limits_{0}^{\infty}ze^{-z}I_{\nu}^{(-1,1)}(\alpha z;\rho
)I_{w}^{(-1,1)}(\beta z;\rho)dz \\
& =\frac{\alpha^{\nu}\beta^{w}}{2^{\nu+w}[\Gamma(\nu+1)]^{2}[\Gamma
(w+1)]^{2}}\sum\limits_{k,l=0}^{\infty}\frac{\alpha^{2k}\beta^{2l}~(1;%
\rho)_{\nu+2k}~(1;\rho)_{w+2l}~\Gamma(\nu+w+2k+2l+2)}{(\nu+1)_{k}~(w+1)_{l}~(%
\frac{\nu+1}{2})_{k}~(\frac{w+1}{2})_{l}~(\frac{\nu+2}{2})_{k}~(\frac{w+2}{2}%
)_{l}~k!l!2^{4k+4l}},
\end{align*}
where $\func{Re}(\nu)>-1$ and $\func{Re}(w)>-1.$
\end{corollary}

For the products of usual Bessel and modified Bessel functions, the
following Corollary can be written.

\begin{corollary}
$~$Integral representations satisfied by the usual Bessel and modified
Bessel functions $J_{\nu}(\alpha z)$, $J_{w}(\beta z)$ and $I_{\nu}(\alpha
z),~I_{w}(\beta z)$ are given by 
\begin{align*}
& \int\limits_{0}^{\infty}ze^{-z}J_{\nu}(\alpha z)J_{w}(\beta z)dz \\
& =\frac{\alpha^{\nu}\beta^{w}}{2^{\nu+w}[\Gamma(\nu+1)]^{2}[\Gamma
(w+1)]^{2}}\sum\limits_{k,l=0}^{\infty}\frac{(-1)^{k+l}\alpha^{2k}\beta
^{2l}~\Gamma(\nu+2k+1)~\Gamma(w+2l+1)~\Gamma(\nu+w+2k+2l+2)}{(\nu
+1)_{k}~(w+1)_{l}~(\frac{\nu+1}{2})_{k}~(\frac{w+1}{2})_{l}~(\frac{\nu+2}{2}%
)_{k}~(\frac{w+2}{2})_{l}~k!l!2^{4k+4l}}, \\
& \int\limits_{0}^{\infty}ze^{-z}I_{\nu}(\alpha z)I_{w}(\beta z)dz \\
& =\frac{\alpha^{\nu}\beta^{w}}{2^{\nu+w}[\Gamma(\nu+1)]^{2}[\Gamma
(w+1)]^{2}}\sum\limits_{k,l=0}^{\infty}\frac{\alpha^{2k}\beta^{2l}~\Gamma
(\nu+2k+1)~\Gamma(w+2l+1)~\Gamma(\nu+w+2k+2l+2)}{(\nu+1)_{k}~(w+1)_{l}~(%
\frac{\nu+1}{2})_{k}~(\frac{w+1}{2})_{l}~(\frac{\nu+2}{2})_{k}~(\frac {w+2}{2%
})_{l}~k!l!2^{4k+4l}},
\end{align*}
where $\func{Re}(\nu)>-1$ and $\func{Re}(w)>-1.$
\end{corollary}

Note that, the Mellin transform of the products of Bessel functions was
obtained in \cite{A.R.M}. Recalling that the integrals 
\begin{equation*}
\int\limits_{0}^{\infty }f(x)J_{\nu }(\alpha x)J_{w}(\beta x)dx
\end{equation*}%
arises in problems involving particle motion in an unbounded rotating fluid 
\cite{A.D, J.H}~in magnetohydrodynamic flow, crack problems in elasticity 
\cite{M.T} and distortions of nearly circular lipid domains where $v,w\in 
%TCIMACRO{\U{2124} }%
%BeginExpansion
\mathbb{Z}
%EndExpansion
^{+}$ and $\alpha ,\beta \in 
%TCIMACRO{\U{211d} }%
%BeginExpansion
\mathbb{R}
%EndExpansion
^{+}$ \cite{S.M}. Taking $~f(x)=xe^{-x}$~the result of the corresponding
Theorem coincides with this integral.

\section{A Three-Fold Integral Formula for the Unified Four Parameter Bessel
Function}

In the following Theorem, a three-fold integral formula for the unified four
parameter Bessel function is given:

\begin{theorem}
A three-fold \ integral formula satisfied by the unified four parameter
Bessel function is given by 
\begin{align*}
G_{\nu}^{(b,c)}(z;\rho) & =\frac{(\frac{z}{2})^{\nu}}{\Gamma(c)\pi
\lbrack\Gamma(\nu+\frac{1}{2})]^{2}}\int\limits_{0}^{1}\int\limits_{0}^{1}%
\int\limits_{0}^{1}t^{-\frac{1}{2}}u^{-\frac{1}{2}}s^{c+\nu-1}(1-t)^{\nu -%
\frac{1}{2}}(1-u)^{\nu-\frac{1}{2}}(1-s)^{-c-\nu-1} \\
& \times~e^{\frac{-s^{2}-\rho(1-s)^{2}}{s(1-s)}}~_{0}F_{3}(-;\frac{1}{4},%
\frac{3}{4},\frac{1}{2};\frac{-bz^{2}s^{2}u^{2}t}{16(1-s)^{2}})dsdtdu,
\end{align*}
where \ $\func{Re}(c)>0,~\func{Re}(\nu)>-\frac{1}{2}$ and $\func{Re}(s)>0.$
\end{theorem}

\begin{proof}
Considering series expansion of the unified four parameter Bessel function as%
\begin{equation*}
G_{\nu}^{(b,c)}(z;\rho)=\sum\limits_{k=0}^{\infty}\frac{2^{4k}(-b)^{k}(c;%
\rho)_{2k+\nu}(\frac{z}{2})^{\nu}z^{2k}\Gamma(k+\frac{1}{2})\Gamma (\nu+%
\frac{1}{2})\Gamma(2k+\frac{1}{2})\Gamma(\nu+\frac{1}{2})}{%
2^{4k}2^{2k}k!\Gamma(k+\frac{1}{2})\Gamma(k+\nu+1)\Gamma(\nu+\frac{1}{2}%
)\Gamma(2k+\nu+1)\Gamma(2k+\frac{1}{2})\Gamma(\nu+\frac{1}{2})},
\end{equation*}
then using Beta function 
\begin{equation*}
B(x;y)=\frac{\Gamma(x)\Gamma(y)}{\Gamma(x+y)}
\end{equation*}
and applying duplication formula for the Gamma function \cite{L.G,H.K} 
\begin{equation*}
\sqrt{\pi}\Gamma(2k+1)=2^{2k}k!\Gamma(k+\frac{1}{2}),
\end{equation*}
we have 
\begin{equation*}
G_{\nu}^{(b,c)}(z;\rho)=\sum\limits_{k=0}^{\infty}\frac{2^{4k}(-b)^{k}(c;%
\rho)_{2k+\nu}(\frac{z}{2})^{\nu}z^{2k}B(k+\frac{1}{2};\nu+\frac{1}{2})B(2k+%
\frac{1}{2};\nu+\frac{1}{2})}{2^{4k}\sqrt{\pi}\Gamma(2k+1)\Gamma (\nu+\frac{1%
}{2})\Gamma(2k+\frac{1}{2})\Gamma(\nu+\frac{1}{2})}.
\end{equation*}
Again, using duplication formula, we get%
\begin{equation*}
G_{\nu}^{(b,c)}(z;\rho)=\sum\limits_{k=0}^{\infty}\frac{2^{4k}(-b)^{k}(c;%
\rho)_{2k+\nu}(\frac{z}{2})^{\nu}z^{2k}B(k+\frac{1}{2};\nu+\frac{1}{2})B(2k+%
\frac{1}{2};\nu+\frac{1}{2})}{\sqrt{\pi}\Gamma(4k+1)\sqrt{\pi}[\Gamma(\nu+%
\frac{1}{2})]^{2}}.
\end{equation*}
Now, inserting the Beta functions, 
\begin{align*}
B(k+\frac{1}{2};\nu+\frac{1}{2}) & =\int\limits_{0}^{1}t^{k-\frac{1}{2}%
}(1-t)^{\nu-\frac{1}{2}}dt, \\
B(2k+\frac{1}{2};\nu+\frac{1}{2}) & =\int\limits_{0}^{1}u^{2k-\frac{1}{2}%
}(1-u)^{\nu-\frac{1}{2}}du
\end{align*}
we have%
\begin{align*}
G_{\nu}^{(b,c)}(z;\rho) & =\sum\limits_{k=0}^{\infty}\frac{%
2^{4k}(-b)^{k}(c;\rho)_{2k+\nu}(\frac{z}{2})^{\nu}z^{2k}}{%
\pi\Gamma(4k+1)[\Gamma (\nu+\frac{1}{2})]^{2}}\int\limits_{0}^{1}t^{k-\frac{1%
}{2}}(1-t)^{\nu-\frac {1}{2}}dt\int\limits_{0}^{1}u^{2k-\frac{1}{2}%
}(1-u)^{\nu-\frac{1}{2}}du, \\
& =\sum\limits_{k=0}^{\infty}\frac{(-b)^{k}(c;\rho)_{2k+\nu~}z^{2k}2^{4k}(%
\frac{z}{2})^{\nu}}{\pi\Gamma(4k+1)[\Gamma(\nu+\frac{1}{2})]^{2}}%
\int\limits_{0}^{1}\int\limits_{0}^{1}t^{k-\frac{1}{2}}u^{2k-\frac{1}{2}%
}(1-t)^{\nu-\frac{1}{2}}(1-u)^{\nu-\frac{1}{2}}dtdu, \\
& =\frac{(\frac{z}{2})^{\nu}}{\pi\lbrack\Gamma(\nu+\frac{1}{2})]^{2}}%
\sum\limits_{k=0}^{\infty}\frac{(-b)^{k}(c;\rho)_{2k+\nu~}z^{2k}2^{4k}}{%
\Gamma(4k+1)}\int\limits_{0}^{1}\int\limits_{0}^{1}t^{k-\frac{1}{2}}u^{2k-%
\frac{1}{2}}(1-t)^{\nu-\frac{1}{2}}(1-u)^{\nu-\frac{1}{2}}dtdu.
\end{align*}
Inserting the generalized Pochhammer function%
\begin{equation*}
(c;\rho)_{2k+\nu~}=\frac{1}{\Gamma(c)}\int\limits_{0}^{\infty}y^{c+2k+\nu
-1}e^{-y-\frac{\rho}{y}}dy,
\end{equation*}
we have%
\begin{align*}
G_{\nu}^{(b,c)}(z;\rho) & =\frac{(\frac{z}{2})^{\nu}}{\pi\lbrack\Gamma (\nu+%
\frac{1}{2})]^{2}}\sum\limits_{k=0}^{\infty}\frac{(-b)^{k}z^{2k}2^{4k}}{%
\Gamma(4k+1)}\frac{1}{\Gamma(c)}(\int\limits_{0}^{\infty}y^{c+2k+\nu
-1}e^{-y-\frac{\rho}{y}}dy) \\
& \times\int\limits_{0}^{1}\int\limits_{0}^{1}t^{k-\frac{1}{2}}u^{2k-\frac {1%
}{2}}(1-t)^{\nu-\frac{1}{2}}(1-u)^{\nu-\frac{1}{2}}dtdu.
\end{align*}
Making substitution $y=\frac{s}{1-s}$ in the first integral on the right
side, we have%
\begin{align*}
G_{\nu}^{(b,c)}(z;\rho) & =\frac{(\frac{z}{2})^{\nu}}{\pi\Gamma
(c)[\Gamma(\nu+\frac{1}{2})]^{2}}\sum\limits_{k=0}^{\infty}\frac {%
(-b)^{k}z^{2k}2^{4k}}{\Gamma(4k+1)} \\
& \times\int\limits_{0}^{1}\int\limits_{0}^{1}\int\limits_{0}^{1}t^{k-\frac{1%
}{2}}u^{2k-\frac{1}{2}}s^{c+2k+\nu-1}(1-t)^{\nu-\frac{1}{2}}(1-u)^{\nu-\frac{%
1}{2}}(1-s)^{-c-2k-\nu-1} \\
& \times e^{\dfrac{-s^{2}-\rho(1-s)^{2}}{s(1-s)}}dsdtdu.
\end{align*}
Taking into consideration of the Legendre's duplication formula for the
function $\Gamma(4k+1)$%
\begin{equation*}
\Gamma(4k+1)=\frac{1}{\sqrt{\pi}}\Gamma(2k+\frac{1}{2})2^{4k}\Gamma(2k+1)
\end{equation*}
$~$and applying (1.2) and (1.3), we have%
\begin{equation*}
\Gamma(4k+1)=2^{4k}2^{2k}~(\frac{1}{4})_{k}~(\frac{3}{4})_{k}~2^{2k}(\frac {1%
}{2})_{k}(1)_{k}.
\end{equation*}
Now, replacing the order of summation and integrals, we get%
\begin{align*}
G_{\nu}^{(b,c)}(z;\rho) & =\frac{(\frac{z}{2})^{\nu}}{\pi\Gamma
(c)[\Gamma(\nu+\frac{1}{2})]^{2}}\int\limits_{0}^{1}\int\limits_{0}^{1}\int%
\limits_{0}^{1}\{\sum\limits_{k=0}^{\infty}\frac{1}{(\frac{1}{4})_{k}(\frac{3%
}{4})_{k}(\frac{1}{2})_{k}~k!}(\frac{-bz^{2}s^{2}u^{2}t}{16(1-s)^{2}})^{k}\}
\\
& \times~t^{-\frac{1}{2}}u^{-\frac{1}{2}}s^{c+\nu-1}(1-t)^{\nu-\frac{1}{2}%
}(1-u)^{\nu-\frac{1}{2}}(1-s)^{-c-\nu-1}~e^{\dfrac{-s^{2}-\rho(1-s)^{2}}{%
s(1-s)}}dsdtdu.
\end{align*}
Using hypergeometric series expansion, the result is obtained%
\begin{align*}
G_{\nu}^{(b,c)}(z;\rho) & =\frac{(\frac{z}{2})^{\nu}}{\Gamma(c)\pi
\lbrack\Gamma(\nu+\frac{1}{2})]^{2}}\int\limits_{0}^{1}\int\limits_{0}^{1}%
\int\limits_{0}^{1}t^{-\frac{1}{2}}u^{-\frac{1}{2}}s^{c+\nu-1}(1-t)^{\nu -%
\frac{1}{2}}(1-u)^{\nu-\frac{1}{2}}(1-s)^{-c-\nu-1} \\
& \times~e^{\frac{-s^{2}-\rho(1-s)^{2}}{s(1-s)}}~_{0}F_{3}(-;\frac{1}{4},%
\frac{3}{4},\frac{1}{2};\frac{-bz^{2}s^{2}u^{2}t}{16(1-s)^{2}})dsdtdu.
\end{align*}
\end{proof}

Letting $b=1$ and $b=-1$ in Theorem 5.1, the following triple integral
representations for the generalized three parameter Bessel and modified
Bessel functions of the first kind are obtained:

\begin{corollary}
For the generalized three parameter Bessel and modified Bessel functions of
the first kind, we have%
\begin{align*}
J_{\nu}^{(c)}(z;\rho) & =\frac{(\frac{z}{2})^{\nu}}{\Gamma(c)\pi
\lbrack\Gamma(\nu+\frac{1}{2})]^{2}}\int\limits_{0}^{1}\int\limits_{0}^{1}%
\int\limits_{0}^{1}t^{-\frac{1}{2}}u^{-\frac{1}{2}}s^{c+\nu-1}(1-t)^{\nu -%
\frac{1}{2}}(1-u)^{\nu-\frac{1}{2}}(1-s)^{-c-\nu-1} \\
& \times e^{\frac{-s^{2}-\rho(1-s)^{2}}{s(1-s)}}~_{0}F_{3}(-;\frac{1}{4},%
\frac{3}{4},\frac{1}{2};\frac{-z^{2}s^{2}u^{2}t}{16(1-s)^{2}})dsdtdu, \\
I_{\nu}^{(c)}(z;\rho) & =\frac{(\frac{z}{2})^{\nu}}{\Gamma(c)\pi
\lbrack\Gamma(\nu+\frac{1}{2})]^{2}}\int\limits_{0}^{1}\int\limits_{0}^{1}%
\int\limits_{0}^{1}t^{-\frac{1}{2}}u^{-\frac{1}{2}}s^{c+\nu-1}(1-t)^{\nu -%
\frac{1}{2}}(1-u)^{\nu-\frac{1}{2}}(1-s)^{-c-\nu-1} \\
& \times~e^{\frac{-s^{2}-\rho(1-s)^{2}}{s(1-s)}}~_{0}F_{3}(-;\frac{1}{4},%
\frac{3}{4},\frac{1}{2};\frac{z^{2}s^{2}u^{2}t}{16(1-s)^{2}})dsdtdu,
\end{align*}
where $\func{Re}(c)>0$ and $\func{Re}(\nu)>-\frac{1}{2}.$
\end{corollary}

Taking $b=1,~c=1~$and $\rho =0$ and $b=-1,~c=1~$and $\rho =0$ in Corollary
5.1.1, the integral representations of the usual Bessel and modified Bessel
functions are given in the following Corollary.

\begin{corollary}
Triple integral formulas satisfied by the usual Bessel and modified Bessel
functions are given by%
\begin{align*}
J_{\nu}(z) & =\frac{(\frac{z}{2})^{\nu}}{\pi\lbrack\Gamma(\nu+\frac{1}{2}%
)]^{2}}\int\limits_{0}^{1}\int\limits_{0}^{1}\int\limits_{0}^{1}t^{-\frac{1}{%
2}}u^{-\frac{1}{2}}s^{\nu}(1-t)^{\nu-\frac{1}{2}}(1-u)^{\nu -\frac{1}{2}%
}(1-s)^{-\nu-2} \\
& \times e^{\frac{-s}{(1-s)}}~_{0}F_{3}(-;\frac{1}{4},\frac{3}{4},\frac{1}{2}%
;\frac{-z^{2}s^{2}u^{2}t}{16(1-s)^{2}})dsdtdu, \\
I_{\nu}(z) & =\frac{(\frac{z}{2})^{\nu}}{\pi\lbrack\Gamma(\nu+\frac{1}{2}%
)]^{2}}\int\limits_{0}^{1}\int\limits_{0}^{1}\int\limits_{0}^{1}t^{-\frac{1}{%
2}}u^{-\frac{1}{2}}s^{\nu}(1-t)^{\nu-\frac{1}{2}}(1-u)^{\nu -\frac{1}{2}%
}(1-s)^{-\nu-2} \\
& \times~e^{\frac{-s}{(1-s)}}~_{0}F_{3}(-;\frac{1}{4},\frac{3}{4},\frac{1}{2}%
;\frac{z^{2}s^{2}u^{2}t}{16(1-s)^{2}})dsdtdu,
\end{align*}
where $\func{Re}(\nu)>-\frac{1}{2}.$
\end{corollary}

\section{Generalized Four Parameter Spherical Bessel and Bessel-Clifford
Functions}

In this Section, generalized four parameter spherical Bessel and
Bessel-Clifford functions are defined by 
\begin{align*}
g_{\nu }^{(b,c)}(z;\rho )& :=\sqrt{\frac{\pi }{2z}}~G_{\nu +\frac{1}{2}%
}^{(b,c-\frac{1}{2})}(z;\rho ) \\
(\func{Re}(\nu )& >-\frac{3}{2},~\func{Re}(c)>\frac{1}{2},~\func{Re}(\rho
)>0), \\
C_{\nu }^{(b,\lambda )}(z;\rho )& :=z^{-\frac{\nu }{2}}G_{\nu }^{(b,\lambda
)}(2\sqrt{z};\rho ) \\
(\func{Re}(\nu )& >-1,~\func{Re}(\lambda )>0,~\func{Re}(\rho )>0),
\end{align*}%
respectively. Taking into consideration of the definition of the unified
four parameter Bessel function, the mentioned generalized functions can be
written as%
\begin{align*}
g_{\nu }^{(b,c)}(z;\rho )& =\frac{\sqrt{\pi }}{2}~\sum\limits_{k=0}^{\infty }%
\frac{(-b)^{k}(c-\frac{1}{2};\rho )_{2k+\nu +\frac{1}{2}}~}{k!~\Gamma (\nu
+k+\frac{3}{2})~\Gamma (\nu +2k+\frac{3}{2})}(\frac{z}{2})^{2k+\nu }, \\
(\func{Re}(\nu )& >-\frac{3}{2},~\func{Re}(c)>\frac{1}{2},~\func{Re}(\rho
)>0), \\
C_{\nu }^{(b,\lambda )}(z;\rho )& =\sum\limits_{k=0}^{\infty }\frac{%
(-b)^{k}(\lambda ;\rho )_{2k+\nu }~}{k!~\Gamma (\nu +k+1)~\Gamma (\nu +2k+1)}%
(\sqrt{z})^{2k}, \\
(\func{Re}(\nu )& >-1,~\func{Re}(\lambda )>0,~\func{Re}(\rho )>0).
\end{align*}%
Taking $b=1,c=\frac{3}{2}~$and $\rho =0,$ $g_{\nu }^{(b,c)}(z;\rho )$ is
reduced to usual spherical Bessel function of the first kind $j_{\nu }(z).~$%
Letting~$b=-1,~\lambda =1$~and $\rho =0,$~$C_{\nu }^{(b,\lambda )}(z;\rho )$%
~is reduced to usual Bessel-Clifford function of the first kind $C_{\nu
}(z).~$Since the proofs of the corresponding Theorems for the generalized
four parameter spherical Bessel and Bessel-Clifford functions are similar
with the unified four parameter Bessel function, details are omitted. Note
that using the monomiality principle, multiplicative and derivative
operators of the usual Bessel-Clifford function were introduced and studied
in \cite{I.P, G.D, H.M}.

\begin{lemma}
Let $\nu\in%
%TCIMACRO{\U{2124} }%
%BeginExpansion
\mathbb{Z}
%EndExpansion
.~$The relationship between $C_{-\nu}^{(b,\lambda)}(z;\rho)$ and $C_{\nu
}^{(b,\lambda)}(z;\rho)$ is given by 
\begin{equation*}
C_{-\nu}^{(b,\lambda)}(z;\rho)=(-b)^{\nu}z^{\nu}C_{\nu}^{(b,\lambda)}(z;%
\rho).
\end{equation*}
\end{lemma}

\begin{proof}
It can be directly seen from Lemma 2.1 in Section 2.
\end{proof}

\begin{corollary}
Let $\nu\in%
%TCIMACRO{\U{2124} }%
%BeginExpansion
\mathbb{Z}
%EndExpansion
.~$The relation between $C_{-\nu}(z)$ and $C_{\nu}(z)$ is given by 
\begin{equation*}
C_{-\nu}(z)=z^{\nu}C_{\nu}(z).
\end{equation*}
\end{corollary}

\begin{theorem}
Let $n\in%
%TCIMACRO{\U{2124} }%
%BeginExpansion
\mathbb{Z}
%EndExpansion
.~$For $t\neq0~$and for all finite $z,~$generating functions of the
generalized four parameter spherical Bessel and Bessel-Clifford functions
are given by%
\begin{align*}
\sum\limits_{n=-\infty}^{\infty}g_{n-\frac{1}{2}}^{(b,c+\frac{1}{2}%
)}(z;\rho)t^{n} & =\sqrt{\frac{\pi}{2z}}~_{1}F_{1}((c;\rho),1;(t-\frac{b}{t})%
\frac{z}{2}), \\
\sum\limits_{n=-\infty}^{\infty}C_{n}^{(b,\lambda)}(z;\rho)t^{n} &
=~_{1}F_{1}((\lambda,\rho),1;t-\frac{bz}{t}).
\end{align*}
\end{theorem}

\begin{proof}
To prove the generating function for the generalized four parameter
spherical Bessel function, we take into consideration of the following
relation 
\begin{equation*}
g_{n}^{(b,c)}(z;\rho)=\sqrt{\frac{\pi}{2z}}~G_{n+\frac{1}{2}}^{(b,c-\frac {1%
}{2})}(z;\rho).
\end{equation*}
Substituting $n-\frac{1}{2}$ for $n$ and $c+\frac{1}{2}$ for $c,$ we have%
\begin{equation*}
g_{n-\frac{1}{2}}^{(b,c+\frac{1}{2})}(z;\rho)=\sqrt{\frac{\pi}{2z}}%
G_{n}^{(b,c)}(z;\rho).
\end{equation*}
Taking summations on both sides of the equality yields%
\begin{equation*}
\sum\limits_{n=-\infty}^{\infty}g_{n-\frac{1}{2}}^{(b,c+\frac{1}{2}%
)}(z;\rho)t^{n}=\sqrt{\frac{\pi}{2z}}\sum\limits_{n=-\infty}^{%
\infty}G_{n}^{(b,c)}(z;\rho)t^{n}.
\end{equation*}
By Theorem 2.2 
\begin{equation*}
\sum\limits_{n=-\infty}^{\infty}g_{n-\frac{1}{2}}^{(b,c+\frac{1}{2}%
)}(z;\rho)t^{n}=\sqrt{\frac{\pi}{2z}}~_{1}F_{1}((c;\rho),1;(t-\frac{b}{t})%
\frac{z}{2}),
\end{equation*}
whence the result. Generating function for the generalized four parameter
Bessel-Clifford function can be shown as a similar process that is applied
in Theorem 2.2 in Section 2.
\end{proof}

\begin{corollary}
\cite{I.P, A.G} Let $n\in 
%TCIMACRO{\U{2124} }%
%BeginExpansion
\mathbb{Z}
%EndExpansion
.~$For $t\neq 0~$and for all finite $z,~$generating functions of the usual
spherical Bessel and Bessel-Clifford functions are given by%
\begin{align*}
\sum\limits_{n=-\infty }^{\infty }j_{n-\frac{1}{2}}(z)t^{n}& =\sqrt{\frac{%
\pi }{2z}}~e^{(t-\frac{1}{t})\frac{z}{2}}, \\
\sum\limits_{n=-\infty }^{\infty }C_{n}(z)t^{n}& =e^{t+\frac{z}{t}}.
\end{align*}
\end{corollary}

\begin{theorem}
Integral representations satisfied by the generalized four parameter
spherical Bessel and Bessel-Clifford functions are given by 
\begin{align*}
g_{\nu}^{(b,c)}(z;\rho) & =\frac{\sqrt{\pi}(\frac{z}{2})^{\nu}}{2[\Gamma
(\nu+\frac{3}{2})]^{2}\Gamma(c-\frac{1}{2})}\int\limits_{0}^{\infty}t^{c+%
\nu-1}e^{-t-\frac{\rho}{t}}~_{0}F_{3}(-;\frac{2\nu+3}{2},\frac{2\nu+3}{4},%
\frac{2\nu+5}{4};\frac{-bz^{2}t^{2}}{16})dt,~ \\
(\func{Re}(\nu) & >-\frac{3}{2},~\func{Re}(c)>\frac{1}{2}) \\
C_{\nu}^{(b,\lambda)}(z;\rho) & =\frac{1}{[\Gamma(\nu+1)]^{2}\Gamma (\lambda)%
}\int\limits_{0}^{\infty}t^{\lambda+\nu-1}e^{-t-\frac{\rho}{t}%
}~_{0}F_{3}(-;\nu+1,\frac{\nu+1}{2},\frac{\nu+2}{2};\frac{-bzt^{2}}{4})dt, \\
(\func{Re}(\nu) & >-1,~\func{Re}(\lambda)>0)
\end{align*}
respectively$.$
\end{theorem}

\begin{proof}
Proofs are clear from Theorem 2.3 in Section 2.
\end{proof}

\begin{corollary}
Integral representations satisfied by spherical Bessel and Bessel-Clifford
functions are given by 
\begin{align*}
j_{\nu}(z) & =\frac{\sqrt{\pi}(\frac{z}{2})^{\nu}}{2[\Gamma(\nu+\frac{3}{2}%
)]^{2}}\int\limits_{0}^{\infty}t^{\nu+\frac{1}{2}}e^{-t}~_{0}F_{3}(-;\frac{%
2\nu+3}{2},\frac{2\nu+3}{4},\frac{2\nu+5}{4};\frac{-z^{2}t^{2}}{16})dt,~%
\func{Re}(\nu)>-\frac{3}{2} \\
C_{\nu}(z) & =\frac{1}{[\Gamma(\nu+1)]^{2}}\int\limits_{0}^{\infty}t^{\nu
}e^{-t}~_{0}F_{3}(-;\nu+1,\frac{\nu+1}{2},\frac{\nu+2}{2};\frac{zt^{2}}{4}%
)dt,~\func{Re}(\nu)>-1.
\end{align*}
\end{corollary}

\begin{theorem}
The Laplace transforms of the generalized four parameter spherical Bessel
and Bessel-Clifford functions are given by%
\begin{align*}
\mathcal{L}\{g_{\nu}^{(b,c)}(t;\rho)\}(s) & =\frac{\sqrt{\pi}}{2s}%
\sum\limits_{k=0}^{\infty}\frac{(-b)^{k}(c-\frac{1}{2};\rho)_{2k+\nu+\frac {1%
}{2}}~\Gamma(\nu+2k+1)}{k!\Gamma(k+\nu+\frac{3}{2})\Gamma(2k+\nu+\frac{3}{2})%
}(\frac{1}{2s})^{2k+\nu}, \\
(\func{Re}(\nu) & >-\frac{3}{2},~\func{Re}(c)>\frac{1}{2}) \\
\mathcal{L}\{C_{\nu}^{(b,\lambda)}(t;\rho)\}(s) & =\sum\limits_{k=0}^{\infty}%
\frac{(-b)^{k}(\lambda,\rho)_{2k+\nu}}{\Gamma(\nu+k+1)~\Gamma
(\nu+2k+1)~s^{k+1}}, \\
(\func{Re}(\nu) & >-1,~\func{Re}(\lambda)>0).
\end{align*}
\end{theorem}

\begin{proof}
Laplace transforms of the generalized four parameter spherical Bessel and
Bessel-Clifford functions can be calculated in a similar way as in Theorem
2.4 in Section 2.
\end{proof}

\begin{corollary}
$~$The Laplace transforms of the spherical Bessel and Bessel-Clifford
functions \ are given by%
\begin{align*}
\mathcal{L}\{j_{\nu}(t)\}(s) & =\frac{\sqrt{\pi}}{2s}\sum\limits_{k=0}^{%
\infty}\frac{(-1)^{k}~\Gamma(\nu+2k+1)}{k!\Gamma(k+\nu+\frac{3}{2})}(\frac{1%
}{2s})^{2k+\nu},\func{Re}(\nu)>-\frac{3}{2}~ \\
\mathcal{L}\{C_{\nu}(t)\}(s) & =\sum\limits_{k=0}^{\infty}\frac{1}{%
\Gamma(\nu+k+1)~s^{k+1}},~\func{Re}(\nu)>-1.
\end{align*}
\end{corollary}

\begin{theorem}
The Mellin transforms of the generalized four parameter spherical Bessel and
Bessel-Clifford functions are given by%
\begin{align*}
\mathcal{M}\{e^{-z}g_{\nu}^{(b,c)}(z;\rho);s\} & =\frac{\sqrt{\pi}\Gamma (c)%
}{2^{\nu+1}[\Gamma(\nu+\frac{3}{2})]^{2}\Gamma(c-\frac{1}{2})}%
\sum\limits_{k=0}^{\infty}\frac{(-1)^{k}b^{k}(c;\rho)_{\nu+2k}~\Gamma
(s+\nu+2k)}{(\frac{2\nu+3}{2})_{k}(\frac{2\nu+3}{4})_{k}(\frac{2\nu+5}{4}%
)_{k}16^{k}k!}, \\
& \func{Re}(\nu)>\frac{-3}{2},~\func{Re}(c)>\frac{1}{2},~s+\nu\neq%
\{0,-1,-2,...\} \\
\mathcal{M}\{e^{-z}C_{\nu}^{(b,\lambda)}(z;\rho);s\} & =\frac{\Gamma (s)}{%
[\Gamma(\nu+1)]^{2}}\sum\limits_{k=0}^{\infty}\frac{(-b)^{k}(\lambda
;\rho)_{\nu+2k}~(s)_{k}}{(\nu+1)_{k}(\frac{\nu+1}{2})_{k}(\frac{\nu+2}{2}%
)_{k}2^{2k}k!}, \\
\func{Re}(\nu) & >-1,~\func{Re}(s)>0,~\func{Re}(\lambda)>0.
\end{align*}
\end{theorem}

\begin{proof}
Proofs of the corresponding Mellin transforms can be done by Theorem 2.5 in
Section 2.
\end{proof}

\begin{corollary}
The Mellin transforms of the spherical Bessel and Bessel-Clifford functions
are given by%
\begin{align*}
\mathcal{M}\{e^{-z}j_{\nu}(z);s\} & =\frac{\sqrt{\pi}\Gamma(\nu+s)}{%
2^{\nu+1}\Gamma(\nu+\frac{3}{2})}~_{2}F_{1}(\frac{\nu+s}{2},\frac{\nu+s+1}{2}%
;\frac{2\nu+3}{2};-1) \\
\func{Re}(\nu) & >\frac{-3}{2},~s+\nu\neq0,-1,-2,... \\
\mathcal{M}\{C_{\nu}(z);s\} & =\frac{\Gamma(s)}{\Gamma(\nu+1)}%
~_{1}F_{1}(s;\nu+1;1), \\
\func{Re}(\nu) & >-1,~\func{Re}(s)>0.
\end{align*}
\end{corollary}

The following Mellin transforms are also satisfied by the generalized four
parameter spherical Bessel and Bessel-Clifford functions, respectively:

\begin{theorem}
The Mellin transforms of the generalized four parameter spherical Bessel and
Bessel-Clifford functions are given by%
\begin{align*}
\mathcal{M}\{g_{\nu}^{(b,c)}(z;\rho);s\} & =\frac{\sqrt{\pi}\Gamma
(s)~\Gamma(c+\nu+s)}{2[\Gamma(\nu+\frac{3}{2})]^{2}~\Gamma(c-\frac{1}{2})}%
~\sum\limits_{m=0}^{\infty}\frac{(\nu+2m)\Gamma(\nu+m)}{m!}J_{\nu+2m}(z) \\
& \times~_{2}F_{3}(\frac{c+\nu+s}{2},~\frac{c+\nu+s+1}{2};\frac{2\nu+3}{2},%
\frac{2\nu+3}{4},~\frac{2\nu+5}{4};\frac{-bz^{2}}{4}), \\
(\func{Re}(\nu) & >-\frac{3}{2},~\func{Re}(c)>\frac{1}{2},~\func{Re}(s)>0) \\
\mathcal{M}\{C_{\nu}^{(b,\lambda)}(z;\rho);s\} & =\frac{\Gamma
(s)~\Gamma(\lambda+\nu+s)}{[\Gamma(\nu+1)]^{2}~\Gamma(\lambda)}~_{2}F_{3}(%
\frac{\lambda+\nu+s}{2},~\frac{\lambda+\nu+s+1}{2};\nu+1,~\frac{\nu +1}{2},~%
\frac{\nu+2}{2};-bz), \\
(\func{Re}(\nu) & >-1,~\func{Re}(\lambda )>0,~\func{Re}(s)>0).
\end{align*}
\end{theorem}

\begin{proof}
Proofs can be done in a similar way as Theorem 2.6 in Section 2.
\end{proof}

\begin{theorem}
Series expansions satisfied by the generalized four parameter spherical
Bessel and Bessel-Clifford functions are given by 
\begin{align*}
g_{\nu}^{(b,c)}(z;\rho) & =\frac{\sqrt{\pi}}{2}\frac{\Gamma(\mu+1)}{%
[\Gamma(\nu+\frac{3}{2})]^{2}}\sum\limits_{m=0}^{\infty}\sum\limits_{n=0}^{%
\infty}(\frac{z}{2})^{\nu-\mu+m+n}~J_{\mu+m+n}(z)\frac{(-m-n)_{n}b^{n}(c-%
\frac{1}{2};\rho)_{2n+\nu+\frac{1}{2}}(\mu+1)_{n}}{n!(\nu+\frac{3}{2}%
)_{n}(\nu+\frac{3}{2})_{2n}(m+n)!}, \\
~(\func{Re}(\mu) & >-1,~\func{Re}(\nu)>-\frac{3}{2},~\func{Re}(c)>\frac{1}{2}%
) \\
C_{\nu}^{(b,\lambda)}(z;\rho) & =\frac{1}{[\Gamma(\nu+1)]^{2}}%
\sum\limits_{m=0}^{\infty}\sum\limits_{n=0}^{\infty}\frac{%
z^{m}2^{n-m}J_{m+n}(z)(-m-n)_{n}(-b)^{n}(\lambda;\rho)_{\nu+2n}}{%
(m+n)!(-1)^{n}(\nu +1)_{n}(\nu+1)_{2n}}, \\
(\func{Re}(\nu) & >-1,~\func{Re}(\lambda)>0).
\end{align*}
\end{theorem}

\begin{proof}
Formulas can be calculated from Theorem 2.7 in Section 2.
\end{proof}

\begin{corollary}
For spherical Bessel and Bessel-Clifford functions, it can be seen that%
\begin{align*}
j_{\nu}(z) & =\frac{\sqrt{\pi}}{2}\frac{\Gamma(\mu+1)}{\Gamma(\nu+\frac {3}{2%
})}\sum\limits_{m=0}^{\infty}\sum\limits_{n=0}^{\infty}(\frac{z}{2}%
)^{\nu-\mu+m+n}~J_{\mu+m+n}(z)\frac{(-m-n)_{n}~(\mu+1)_{n}}{n!(\nu+\frac{3}{2%
})_{n}~(m+n)!},~\func{Re}(\nu)>-\frac{3}{2}, \\
C_{\nu}(z) & =\frac{1}{\Gamma(\nu+1)}\sum\limits_{m=0}^{\infty}\sum%
\limits_{n=0}^{\infty}\frac{z^{m}2^{n-m}J_{m+n}(z)(-m-n)_{n}}{%
(-1)^{n}(m+n)!(\nu+1)_{n}},~\func{Re}(\nu)>-1.
\end{align*}
\end{corollary}

\begin{theorem}
Recurrence relations satisfied by the generalized four parameter spherical
Bessel and Bessel-Clifford functions are given by 
\begin{align*}
\frac{\partial}{\partial z}[z^{\nu+\frac{3}{2}}\frac{\partial}{\partial z}%
[z^{-\nu}g_{\nu}^{(b,c-1)}(z;\rho)]] & =-bz^{\frac{1}{2}}(c-\frac{3}{2}%
)g_{\nu+1}^{(b,c)}(z;\rho), \\
\frac{\partial}{\partial z}[(\sqrt{z})^{\nu+2}\frac{\partial}{\partial z}%
(C_{\nu}^{(b,\lambda-1)}(z;\rho))] & =\frac{-b(\lambda-1)}{2}(\sqrt {z}%
)^{\nu}C_{\nu+1}^{(b,\lambda)}(z;\rho).
\end{align*}
\end{theorem}

\begin{proof}
Recurrence formulas can be found as in Theorem 3.2 in Section 3.
\end{proof}

\begin{theorem}
Recurrence relations satisfied by the generalized four parameter spherical
Bessel and Bessel-Clifford functions are given by%
\begin{align*}
\frac{\partial}{\partial z}[z^{-\nu+\frac{1}{2}}\frac{\partial}{\partial z}%
[z^{\nu+1}g_{\nu}^{(b,c-1)}(z;\rho)]] & =z^{\frac{1}{2}}(c-\frac{3}{2}%
)g_{\nu-1}^{(b,c)}(z;\rho), \\
\frac{\partial}{\partial z}[(\sqrt{z})^{-\nu+2}\frac{\partial}{\partial z}[(%
\sqrt{z})^{2\nu}C_{\nu}^{(b,\lambda-1)}(z;\rho)]] & =(\sqrt{z})^{\nu -2}%
\frac{(\lambda-1)}{2}C_{\nu-1}^{(b,\lambda)}(z;\rho).
\end{align*}
\end{theorem}

\begin{proof}
Recurrence relations can be seen from Theorem 3.1 in Section 3.
\end{proof}

\begin{theorem}
Derivative formulas satisfied by the generalized four parameter spherical
Bessel and Bessel-Clifford functions are given by 
\begin{align*}
\frac{\partial}{\partial\rho}[g_{\nu}^{(b,c)}(z;\rho)] & =\frac{-1}{c-\frac{3%
}{2}}g_{\nu}^{(b,c-1)}(z;\rho), \\
\frac{\partial}{\partial\rho}[C_{\nu}^{(b,\lambda)}(z;\rho)] & =\frac {-1}{%
\lambda-1}C_{\nu}^{(b,\lambda-1)}(z;\rho).
\end{align*}
\end{theorem}

\begin{proof}
Derivatives can directly seen from Theorem 3.3 in Section 3.
\end{proof}

\begin{theorem}
Recurrence relations satisfied by the generalized four parameter spherical
Bessel and Bessel-Clifford functions are given by 
\begin{gather*}
z^{2}[\frac{\partial^{2}}{\partial z^{2}}g_{\nu-1}^{(b,c-1)}(z;\rho )+b\frac{%
\partial^{2}}{\partial z^{2}}g_{\nu+1}^{(b,c-1)}(z;\rho)]+z[(\frac {5}{2}%
-\nu)\frac{\partial}{\partial z}g_{\nu-1}^{(b,c-1)}(z;\rho)+b(\nu +\frac{7}{2%
})\frac{\partial}{\partial z}g_{\nu+1}^{(b,c-1)}(z;\rho)] \\
=-\frac{(1-\nu)}{2}g_{\nu-1}^{(b,c-1)}(z;\rho)-b\frac{(2+\nu)}{2}g_{\nu
+1}^{(b,c-1)}(z;\rho), \\
-b(3\nu+5)z\frac{\partial}{\partial z}C_{\nu+1}^{(b,\lambda-1)}(z;\rho
)-b(\nu+1)^{2}C_{\nu+1}^{(b,\lambda-1)}(z;\rho)-2bz^{2}\frac{\partial^{2}}{%
\partial z^{2}}C_{\nu+1}^{(b,\lambda-1)}(z;\rho) \\
=(\nu+1)\frac{\partial}{\partial z}C_{\nu-1}^{(b,\lambda-1)}(z;\rho )+2z%
\frac{\partial^{2}}{\partial z^{2}}C_{\nu-1}^{(b,\lambda-1)}(z;\rho).
\end{gather*}
\end{theorem}

\begin{proof}
Recurrence formulas satisfied by generalized four parameter spherical Bessel
and Bessel-Clifford functions can be proved by Theorem 3.4 in Section 3.
\end{proof}

In Theorem 6.11, taking $b=-1,~\lambda =2$ and $\rho =0$ ~$C_{\nu
-1}^{(b,\lambda -1)}(z;\rho )$ is reduced to $C_{\nu -1}(z)$ and under the
same substitutions $C_{\nu +1}^{(b,\lambda -1)}(z;\rho )$ is reduced to $%
C_{\nu +1}(z)$.~For $b=1,~c=\frac{5}{2}$ and $\rho =0,~g_{\nu
-1}^{(b,c-1)}(z;\rho )~$is reduced to $j_{\nu -1}(z)$ and under the same
substitutions $g_{\nu +1}^{(b,c-1)}(z;\rho )~$\ is reduced to $j_{\nu
+1}(z). $ Hence, the following Corollary is obtained:

\begin{corollary}
Recurrence relations satisfied by the usual spherical Bessel and
Bessel-Clifford functions are given by 
\begin{gather*}
z^{2}[\frac{d^{2}}{dz^{2}}j_{\nu-1}(z)+\frac{d^{2}}{dz^{2}}j_{\nu +1}(z)]+z[(%
\frac{5}{2}-\nu)\frac{d}{dz}j_{\nu-1}(z)+(\nu+\frac{7}{2})\frac {d}{dz}%
j_{\nu+1}(z)] \\
=-\frac{(1-\nu)}{2}j_{\nu-1}(z)-\frac{(2+\nu)}{2}j_{\nu+1}(z), \\
(3\nu+5)z\frac{d}{dz}C_{\nu+1}(z)+(\nu+1)^{2}C_{\nu+1}(z)+2z^{2}\frac{d^{2}}{%
dz^{2}}C_{\nu+1}(z) \\
=(\nu+1)\frac{d}{dz}C_{\nu-1}(z)+2z\frac{d^{2}}{dz^{2}}C_{\nu-1}(z).
\end{gather*}
\end{corollary}

\begin{proof}
To obtain the recurrence relation satisfied by spherical Bessel function,
the following recurrence relation should be written%
\begin{equation*}
\frac{z}{2\nu+1}[~j_{\nu-1}(z)+j_{\nu+1}(z)]=j_{\nu}(z).
\end{equation*}
Differentiating the above recurrence formula with respect to $z\,~$yields%
\begin{equation*}
\frac{d}{dz}j_{\nu}(z)=\frac{1}{2\nu+1}[~j_{\nu-1}(z)+j_{\nu+1}(z)+z(\frac {d%
}{dz}j_{\nu-1}(z)+\frac{d}{dz}j_{\nu+1}(z))].
\end{equation*}
The second recurrence formula satisfied by spherical Bessel function is
given by%
\begin{equation*}
\frac{d}{dz}j_{\nu}(z)=\frac{1}{2\nu+1}[\nu~j_{\nu-1}(z)-(\nu+1)~j_{\nu
+1}(z)].
\end{equation*}
Comparing the last two equations, one can get%
\begin{equation*}
(\nu-1)~j_{\nu-1}(z)-(\nu+2)~j_{\nu+1}(z)=z(\frac{d}{dz}j_{\nu-1}(z)+\frac {d%
}{dz}j_{\nu+1}(z)).~
\end{equation*}
Inserting $~j_{\nu-1}(z)$ and $~j_{\nu+1}(z)$ 
\begin{align*}
~j_{\nu-1}(z) & =z^{-\nu-1}\frac{d}{dz}(z^{\nu+1}j_{\nu}(z)), \\
~j_{\nu+1}(z) & =z^{\nu}\frac{d}{dz}(z^{-\nu}j_{\nu}(z)),
\end{align*}
and then taking corresponding derivatives gives%
\begin{equation*}
(-2\nu-1)~j_{\nu}(z)+(2\nu+1)z\frac{d}{dz}j_{\nu}(z)=z^{2}(\frac{d}{dz}%
j_{\nu-1}(z)+\frac{d}{dz}j_{\nu+1}(z)).
\end{equation*}
Now, substituting $j_{\nu}(z)$ and $\frac{d}{dz}j_{\nu}(z)$ 
\begin{align*}
j_{\nu}(z) & =\frac{z}{2\nu+1}[~j_{\nu-1}(z)+j_{\nu+1}(z)], \\
\frac{d}{dz}j_{\nu}(z) & =\frac{1}{2\nu+1}[\nu~j_{\nu-1}(z)-(\nu
+1)~j_{\nu+1}(z)]
\end{align*}
in above recurrence formula, one can get%
\begin{equation*}
(\nu-1)z~j_{\nu-1}(z)-z(\nu+2)~j_{\nu+1}(z)=z^{2}(\frac{d}{dz}j_{\nu -1}(z)+%
\frac{d}{dz}j_{\nu+1}(z)).
\end{equation*}
Differentiating with respect to $z$ in the last recurrence formula and then
multiplying with $2$ on both sides gives%
\begin{align*}
& 2z^{2}(\frac{d^{2}}{dz^{2}}j_{\nu-1}(z)+\frac{d^{2}}{dz^{2}}j_{\nu+1}(z))
\\
& =z[(2\nu-6)\frac{~d}{dz}j_{\nu-1}(z)-(2\nu+8)\frac{d}{dz}j_{\nu
+1}(z)]+(2\nu-2)~j_{\nu-1}(z)-(2\nu+4)~j_{\nu+1}(z).
\end{align*}
Finally, using the fact that%
\begin{equation*}
(\nu-1)~j_{\nu-1}(z)-(\nu+2)~j_{\nu+1}(z)=z(\frac{d}{dz}j_{\nu-1}(z)+\frac {d%
}{dz}j_{\nu+1}(z))
\end{equation*}
result is obtained.

To get the recurrence formula satisfied by the usual Bessel-Clifford
function, the following recurrence formula can be written%
\begin{equation*}
zC_{\nu+2}(z)+(\nu+1)C_{\nu+1}(z)=C_{\nu}(z).
\end{equation*}
Taking derivative with respect to $z$ in above recurrence yields%
\begin{equation*}
C_{\nu+2}(z)+z\frac{d}{dz}C_{\nu+2}(z)+(\nu+1)\frac{d}{dz}C_{\nu+1}(z)=\frac{%
d}{dz}C_{\nu}(z).
\end{equation*}
Using the derivative formula%
\begin{equation*}
\frac{d}{dz}C_{\nu}(z)=C_{\nu+1}(z)
\end{equation*}
last recurrence formula can be written as 
\begin{equation*}
(\nu+2)\frac{d^{2}}{dz^{2}}C_{\nu}(z)+z\frac{d^{2}}{dz^{2}}C_{\nu+1}(z)=%
\frac{d^{2}}{dz^{2}}C_{\nu-1}(z).
\end{equation*}
Multiplying with $z$ on both sides of the above equation gives%
\begin{equation*}
(\nu+2)z\frac{d^{2}}{dz^{2}}C_{\nu}(z)+z^{2}\frac{d^{2}}{dz^{2}}C_{\nu
+1}(z)=z\frac{d^{2}}{dz^{2}}C_{\nu-1}(z).
\end{equation*}
Substituting%
\begin{equation*}
\frac{d^{2}}{dz^{2}}C_{\nu}(z)=\frac{d}{dz}C_{\nu+1}(z)
\end{equation*}
gives%
\begin{equation*}
(\nu+2)z\frac{d}{dz}C_{\nu+1}(z)+z^{2}\frac{d^{2}}{dz^{2}}C_{\nu+1}(z)=z%
\frac{d^{2}}{dz^{2}}C_{\nu-1}(z).
\end{equation*}
Multiplying by $2$ on both sides and then adding the term $(\nu+1)\frac{d}{dz%
}C_{\nu-1}(z)$ gives%
\begin{align*}
& z(2\nu+4)\frac{d}{dz}C_{\nu+1}(z)+2z^{2}\frac{d^{2}}{dz^{2}}C_{\nu
+1}(z)+(\nu+1)\frac{d}{dz}C_{\nu-1}(z) \\
& =(\nu+1)\frac{d}{dz}C_{\nu-1}(z)+2z\frac{d^{2}}{dz^{2}}C_{\nu-1}(z).
\end{align*}
Using the fact that%
\begin{equation*}
(\nu+1)z\frac{d}{dz}C_{\nu+1}(z)+(\nu+1)^{2}C_{\nu+1}(z)=(\nu+1)\frac{d}{dz}%
C_{\nu-1}(z)
\end{equation*}
which is directly seen by first two recurrence formulas satisfied by
Bessel-Clifford function in Section 1, whence the result.
\end{proof}

\begin{theorem}
Partial differential equations satisfied by the generalized four parameter
spherical Bessel and Bessel-Clifford functions are given by%
\begin{gather*}
(-4\nu^{2}-4\nu+6)z\frac{\partial^{3}}{\partial z\partial\rho^{2}}g_{\nu
}^{(b,c)}(z;\rho)+(-4\nu^{2}-4\nu+39)z^{2}\frac{\partial^{4}}{\partial
z^{2}\partial\rho^{2}}g_{\nu}^{(b,c)}(z;\rho)+28z^{3}\frac{\partial^{5}}{%
\partial z^{3}\partial\rho^{2}}g_{\nu}^{(b,c)}(z;\rho) \\
+4z^{4}\frac{\partial^{6}}{\partial z^{4}\partial\rho^{2}}%
g_{\nu}^{(b,c)}(z;\rho)+(\nu^{2}+\nu)\frac{\partial^{2}}{\partial\rho^{2}}%
g_{\nu }^{(b,c)}(z;\rho) \\
=-4bz^{2}g_{\nu}^{(b,c)}(z;\rho), \\
(\nu^{3}+4\nu^{2}+5\nu+2)\frac{\partial^{3}}{\partial z\partial\rho^{2}}%
C_{\nu}^{(b,\lambda)}(z;\rho)+(5\nu^{2}+21\nu+22)z\frac{\partial^{4}}{%
\partial z^{2}\partial\rho^{2}}C_{\nu}^{(b,\lambda)}(z;\rho) \\
+(8\nu+22)z^{2}\frac{\partial^{5}}{\partial z^{3}\partial\rho^{2}}C_{\nu
}^{(b,\lambda)}(z;\rho)+4z^{3}\frac{\partial^{6}}{\partial z^{4}\partial
\rho^{2}}C_{\nu}^{(b,\lambda)}(z;\rho) \\
=-bC_{\nu}^{(b,\lambda)}(z;\rho).
\end{gather*}
\end{theorem}

\begin{proof}
Partial differential equations satisfied by the generalized four parameter
spherical Bessel and Bessel-Clifford functions can be proved by Theorem 3.5
in Section 3.
\end{proof}

\section{Concluding Remarks}

The generalizations of Bessel functions were introduced by some authors in 
\cite{A.K, C, D.M, F.S, L.G, L, H.A.D, M, N.V.O}. In the present article,
defining the unification of the Bessel functions via the generalized
Pochhammer function, some potentially useful formulas are obtained. Some of
these important formulas are new. \ In the below table, some special cases
of the unified four parameter Bessel function are presented:

\subsection{Special Cases of the Unified Four Parameter Bessel Function $G_{%
\protect\nu }^{(b,c)}(z;\protect\rho )$}

\begin{center}
\begin{tabular}{lll}
\hline
No & Values of the parameters & Relation Between $G_{\nu }^{(b,c)}(z;\rho )$
and its special case \\ \hline
I & $b=-1,~c=\frac{1}{2},~\nu \rightarrow \nu +\frac{1}{2},~\rho =0$ & $%
G_{\nu +\frac{1}{2}}^{(-1,\frac{1}{2})}(z)=\frac{\Gamma (\nu +1)(\frac{z}{2}%
)^{\nu +\frac{1}{2}}}{\sqrt{\pi }[\Gamma (\nu +\frac{3}{2})]^{2}}~_{2}F_{3}(%
\frac{\nu +1}{2},\frac{\nu +2}{2};\frac{2\nu +3}{2},\frac{2\nu +3}{4},\frac{%
2\nu +5}{4};\frac{z^{2}}{4})$ \\ 
II & $b=1,~c=\frac{1}{2},~\nu \rightarrow \nu +\frac{1}{2},~\rho =0$ & $%
G_{\nu +\frac{1}{2}}^{(1,\frac{1}{2})}(z)=\frac{\Gamma (\nu +1)(\frac{z}{2}%
)^{\nu +\frac{1}{2}}}{\sqrt{\pi }[\Gamma (\nu +\frac{3}{2})]^{2}}~_{2}F_{3}(%
\frac{\nu +1}{2},\frac{\nu +2}{2};\frac{2\nu +3}{2},\frac{2\nu +3}{4},\frac{%
2\nu +5}{4};\frac{-z^{2}}{4}),$ \\ 
III & $b=-1,~c=1,~~\nu \rightarrow \nu +\frac{1}{2},~\rho =0$ & $G_{\nu +%
\frac{1}{2}}^{(-1,1)}(z)=\frac{(\frac{z}{2})^{\nu +\frac{1}{2}}}{\Gamma (\nu
+\frac{3}{2})}~_{0}F_{1}(-;\nu +\frac{3}{2};\frac{z^{2}}{4}),$ \\ 
IV & $b=1,~c=1,~~\nu \rightarrow \nu +\frac{1}{2},~\rho =0$ & $G_{\nu +\frac{%
1}{2}}^{(1,1)}(z)=\frac{(\frac{z}{2})^{\nu +\frac{1}{2}}}{\Gamma (\nu +\frac{%
3}{2})}~_{0}F_{1}(-;\nu +\frac{3}{2};\frac{-z^{2}}{4}),$ \\ 
V & $b=-1,~c=\frac{3}{2},~\nu \rightarrow \nu -\frac{1}{2},~\rho =0$ & $%
G_{\nu -\frac{1}{2}}^{(-1,\frac{3}{2})}(z)=\frac{z^{\nu -\frac{1}{2}}\Gamma
(\nu +1)}{\sqrt{\pi }[\Gamma (\nu +\frac{1}{2})]^{2}2^{\nu -\frac{3}{2}}}%
~_{2}F_{3}(\frac{\nu +1}{2},\frac{\nu +2}{2};\frac{2\nu +1}{2},\frac{2\nu +1%
}{4},\frac{2\nu +3}{4};\frac{z^{2}}{4}),$ \\ 
VI & $b=1,~c=\frac{3}{2},~\nu \rightarrow \nu -\frac{1}{2},~\rho =0$ & $%
G_{\nu -\frac{1}{2}}^{(1,\frac{3}{2})}(z)=\frac{z^{\nu -\frac{1}{2}}\Gamma
(\nu +1)}{\sqrt{\pi }[\Gamma (\nu +\frac{1}{2})]^{2}2^{\nu -\frac{3}{2}}}%
~_{2}F_{3}(\frac{\nu +1}{2},\frac{\nu +2}{2};\frac{2\nu +1}{2},\frac{2\nu +1%
}{4},\frac{2\nu +3}{4};\frac{-z^{2}}{4}),$ \\ 
VII & $b=-1,~c=1,~\nu \rightarrow \nu -\frac{1}{2},~\rho =0$ & $G_{\nu -%
\frac{1}{2}}^{(-1,1)}(z)=\frac{(\frac{z}{2})^{\nu -\frac{1}{2}}}{\Gamma (\nu
+\frac{1}{2})}~_{0}F_{1}(-;\nu +\frac{1}{2};\frac{z^{2}}{4}),$ \\ 
VIII & $b=1,~c=1,~\nu \rightarrow \nu -\frac{1}{2},~~\rho =0$ & $G_{\nu -%
\frac{1}{2}}^{(1,1)}(z)=\frac{(\frac{z}{2})^{\nu -\frac{1}{2}}}{\Gamma (\nu +%
\frac{1}{2})}~_{0}F_{1}(-;\nu +\frac{1}{2};\frac{-z^{2}}{4}),$ \\ 
IX & $b=-1,~c=1,~\nu =0,~\rho =0$ & $G_{0}^{(-1,1)}(z)=~_{0}F_{1}(-;1;\frac{%
z^{2}}{4}),$ \\ 
X & $b=1,~c=1,~\nu =0,~\rho =0$ & $G_{0}^{(1,1)}(z)=~_{0}F_{1}(-;1;-\frac{%
z^{2}}{4}),$ \\ 
XI & $b=1,~c=1,~\nu =\frac{1}{2},~\rho =0$ & $G_{\frac{1}{2}}^{(1,1)}(z)=%
\sqrt{\frac{2}{\pi z}}\sin z=J_{\frac{1}{2}}(z),$ \\ 
XII & $b=1,~c=1,~\nu =\frac{-1}{2},~\rho =0$ & $G_{-\frac{1}{2}}^{(1,1)}(z)=%
\sqrt{\frac{2}{\pi z}}\cos z=J_{-\frac{1}{2}}(z),$ \\ 
XIII & $b=1,~c=1,~\rho =0$ & $G_{\nu }^{(1,1)}(z)=J_{\nu }(z),$ \\ 
XIV & $b=-1,~c=1,~~\rho =0$ & $G_{\nu }^{(-1,1)}(z)=I_{\nu }(z).$ \\ \hline
\end{tabular}
\end{center}

\bigskip

\newpage

\subsection{Graphics}

In below, the graphics of the generalized two parameter Bessel functions of
the first kind $J_{\nu }^{(c)}(x)~$are drawn for the special cases $v=0$ and 
$c=1,2,3,~\nu =\frac{1}{2}$ and $c=1,2,3,~~\nu =\frac{3}{2}$ and $c=1,2,3$
and the graphics of the generalized two parameter spherical Bessel function $%
g_{\nu }^{(c)}(x)$ are drawn for the special cases $v=0$ and $c=\frac{3}{2},%
\frac{5}{2},\frac{7}{2},~v=\frac{1}{2}~$and $c=\frac{3}{2},\frac{5}{2},\frac{%
7}{2}$ and $v=\frac{3}{2}~$and $c=\frac{3}{2},\frac{5}{2},\frac{7}{2}~$%
respectively.

\vspace{5mm}

\subsection{\protect\bigskip Comparison of Graphs of the Generalized Two
Parameter Bessel Function $J_{\protect\nu}^{(c)}(x)~$with $J_{0}(x),J_{\frac{%
1}{2}}(x) and J_{\frac{3}{2}}(x)$}

\begin{center}
\begin{tabular}{lll}

\fbox{\includegraphics[width=5.05cm,height=7.5981cm]{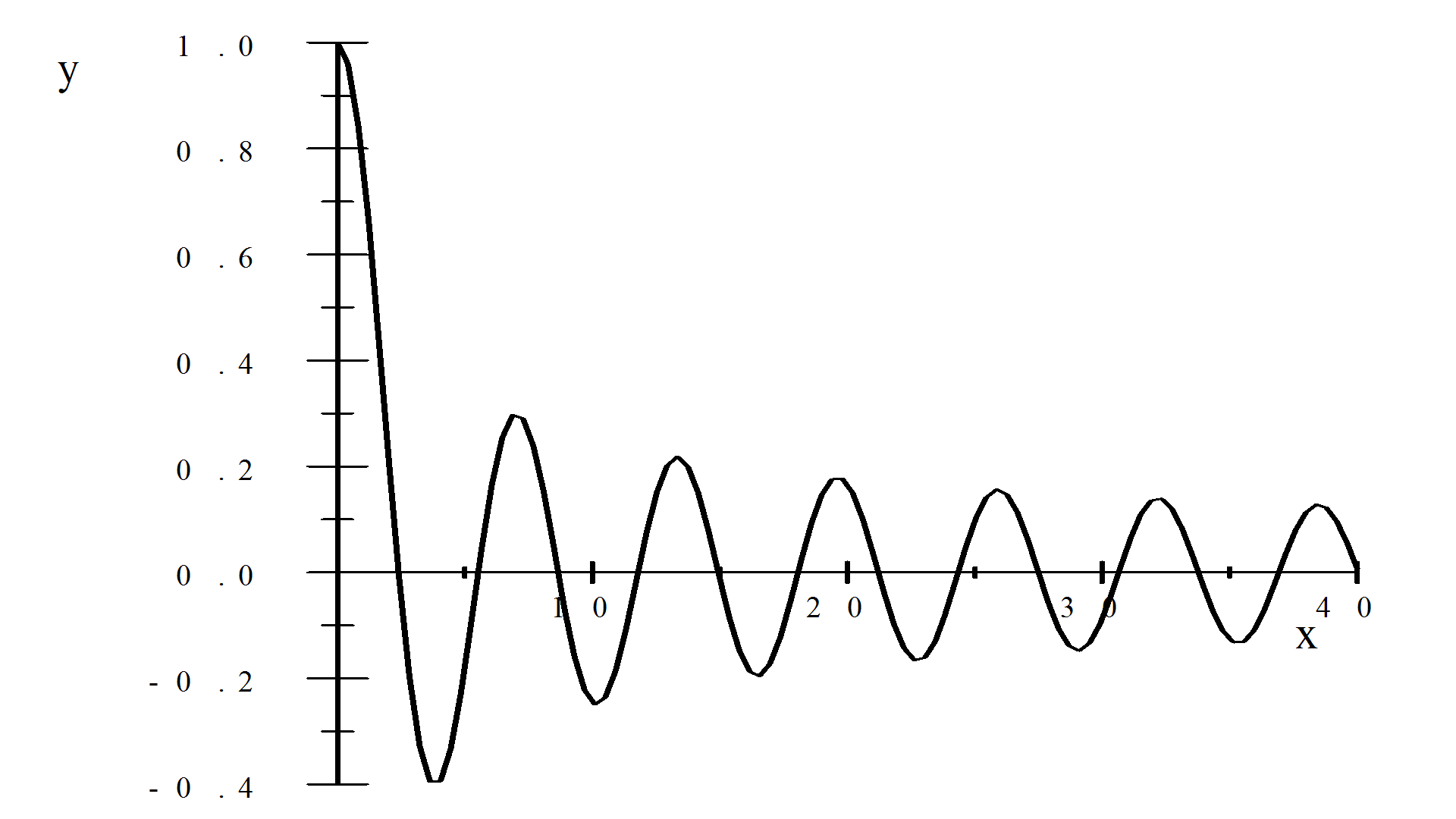}}
&
\fbox{\includegraphics[width= 5.1796cm,height=7.6289cm]{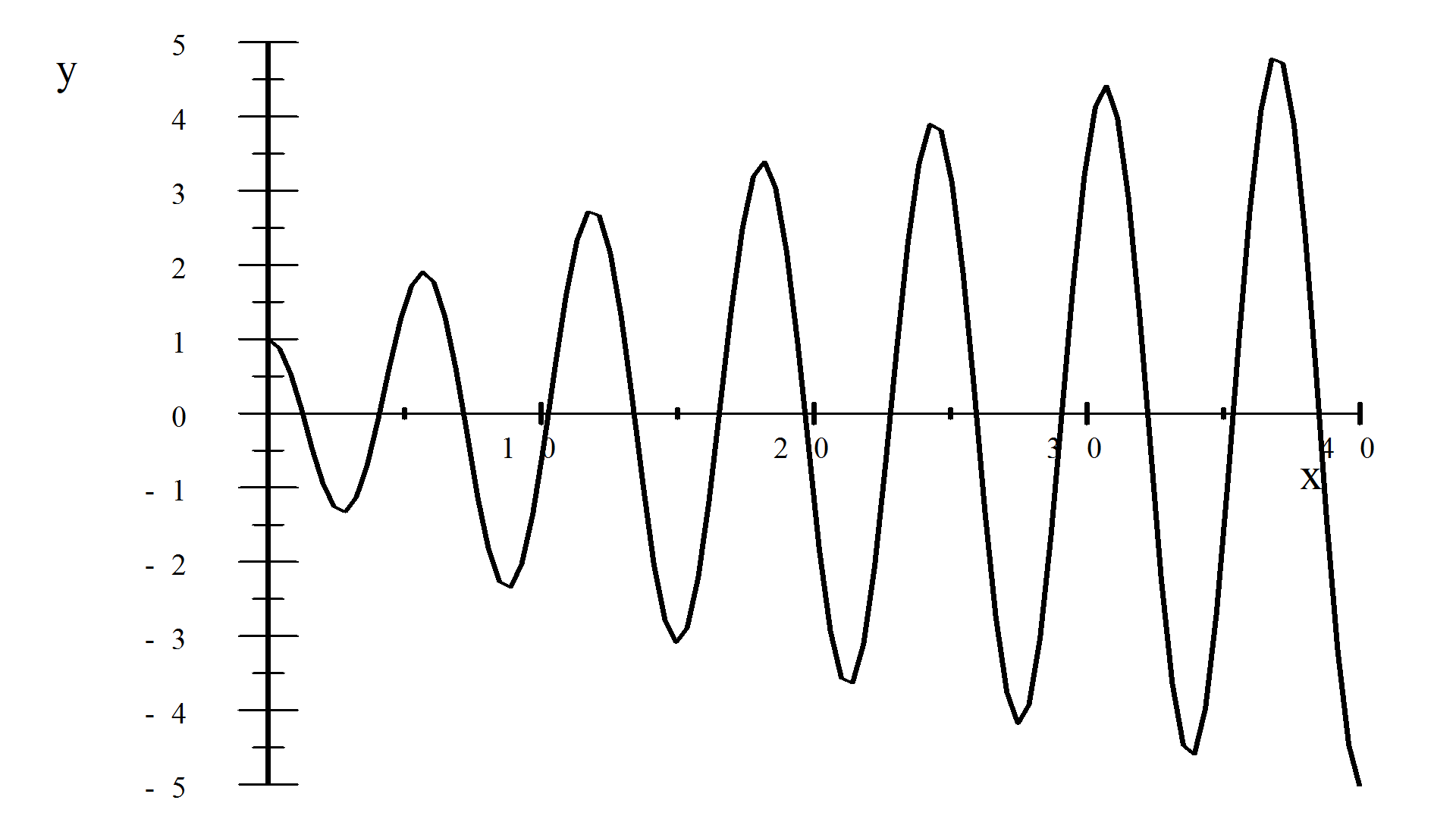}}
&
\fbox{\includegraphics[width=5.1796cm,height=7.6289cm]{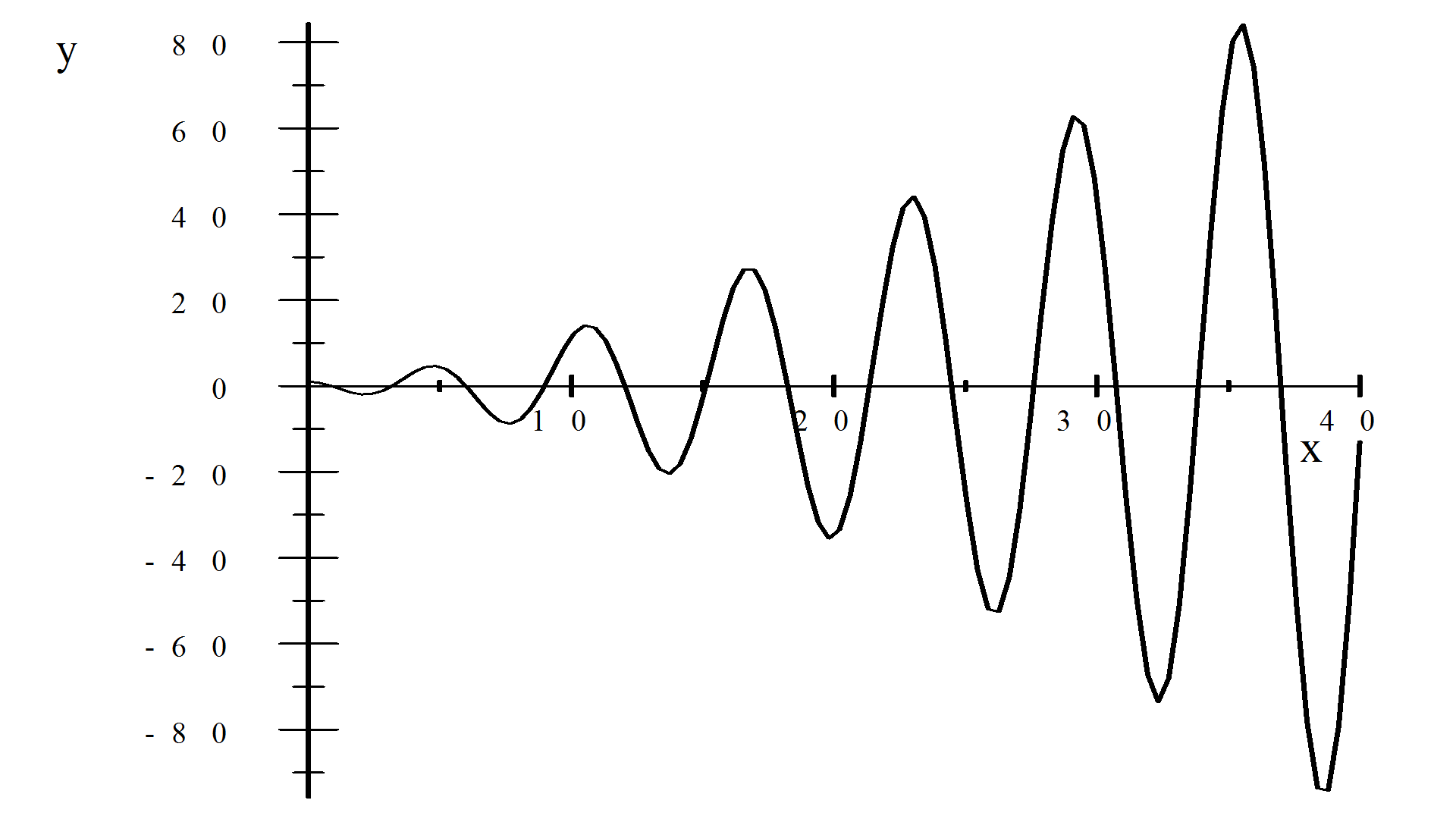}}
\\
FIGURE 1 Graph of $J_{0}(x)$ & FIGURE 2 Graph of $J_{0}^{(2)}(x)$ & FIGURE 3
Graph of $J_{0}^{(3)}(x)$%
\end{tabular}

\bigskip

\begin{tabular}{lll}

\fbox{\includegraphics[width=5.05cm,height=7.5981cm]{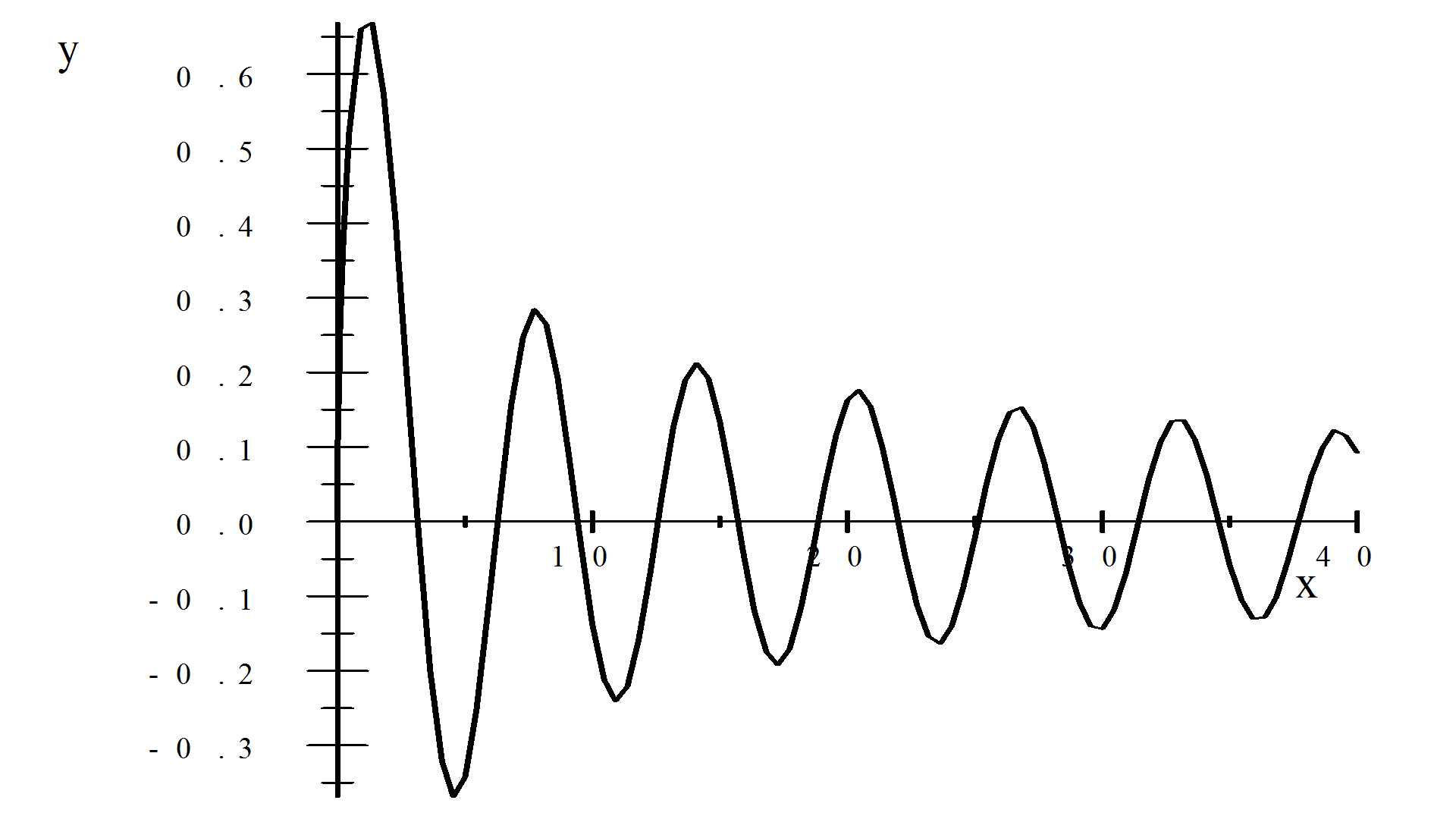}}
&
\fbox{\includegraphics[width= 5.1796cm,height=7.6289cm]{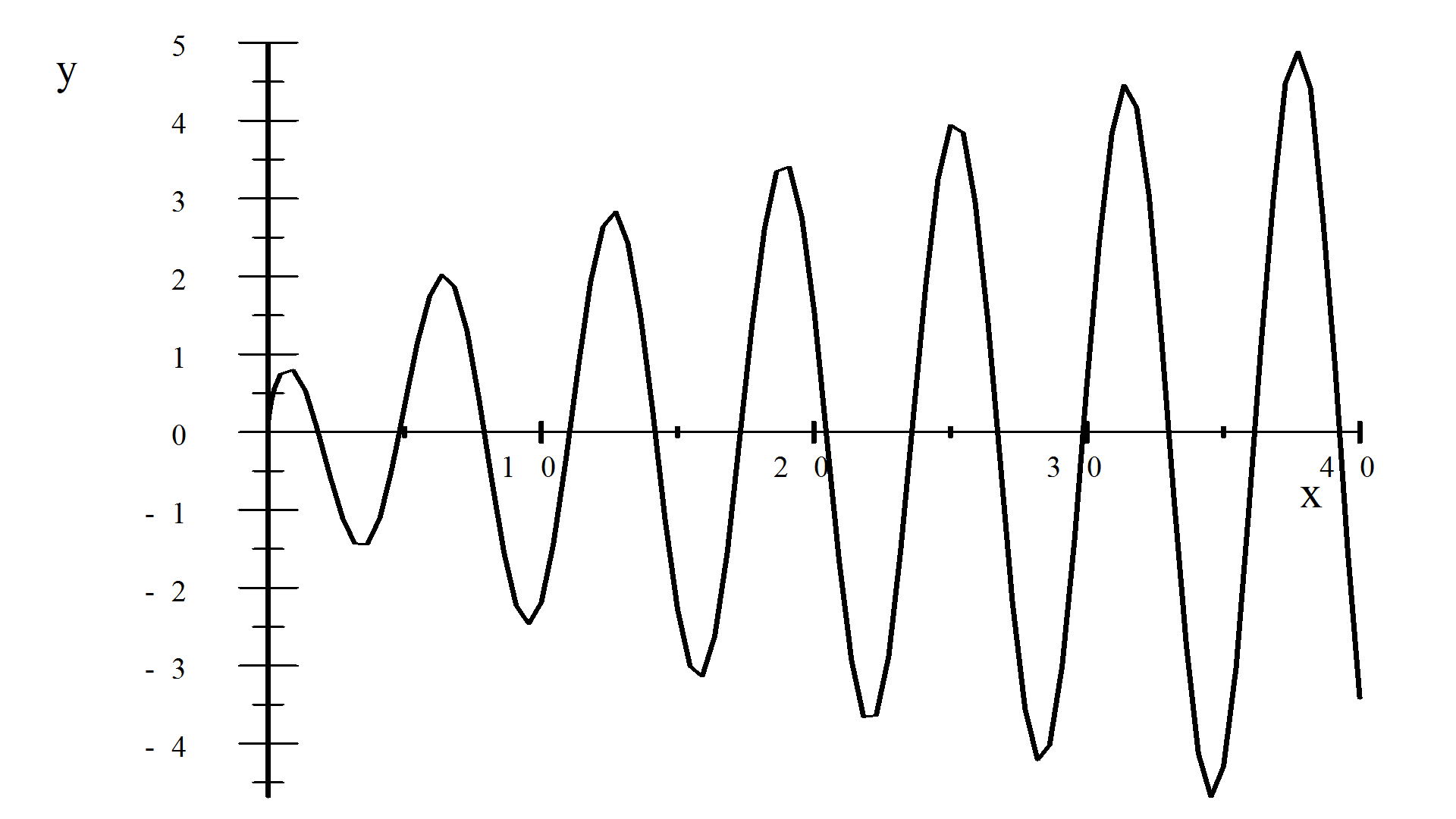}}
&
\fbox{\includegraphics[width=5.1796cm,height=7.6289cm]{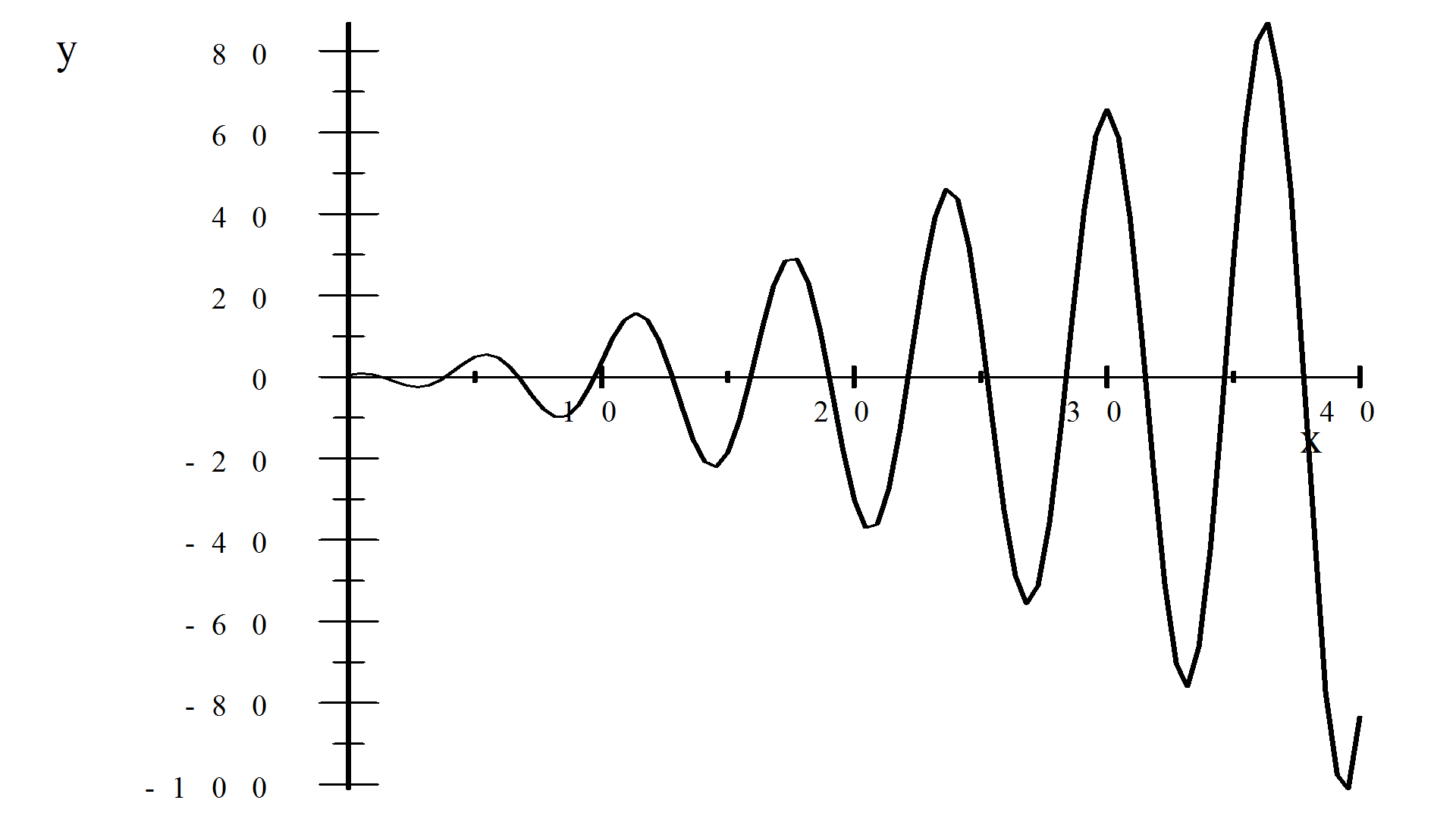}}
\\

FIGURE 4 Graph of $J_{\frac{1}{2}}(x)$ & FIGURE 5 Graph of $J_{\frac{1}{2}%
}^{(2)}(x)$ & FIGURE 6 Graph of $J_{\frac{1}{2}}^{(3)}(x)$%
\end{tabular}

\bigskip

\begin{tabular}{lll}

\fbox{\includegraphics[width=5.05cm,height=7.5981cm]{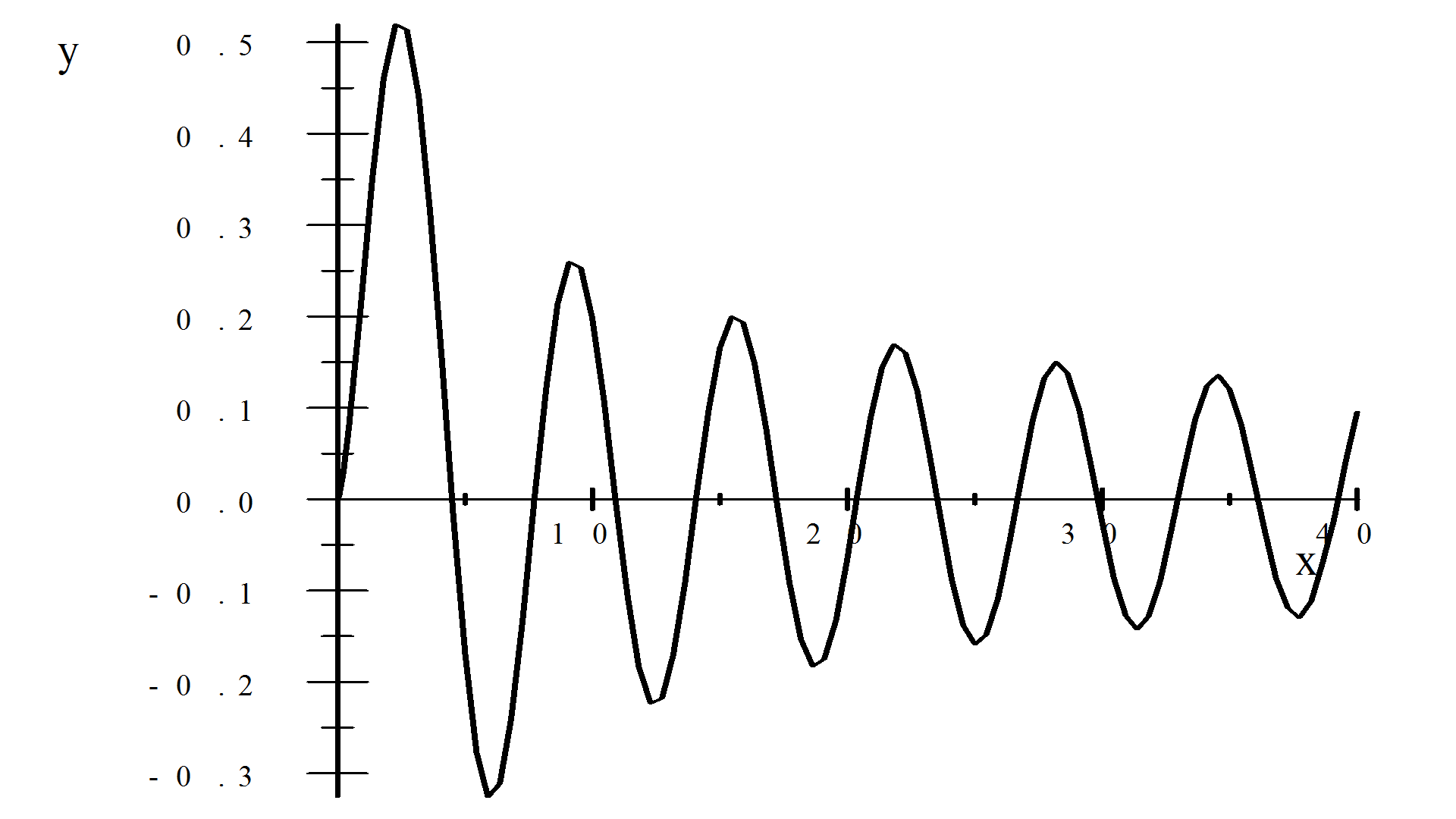}}
&
\fbox{\includegraphics[width= 5.1796cm,height=7.6289cm]{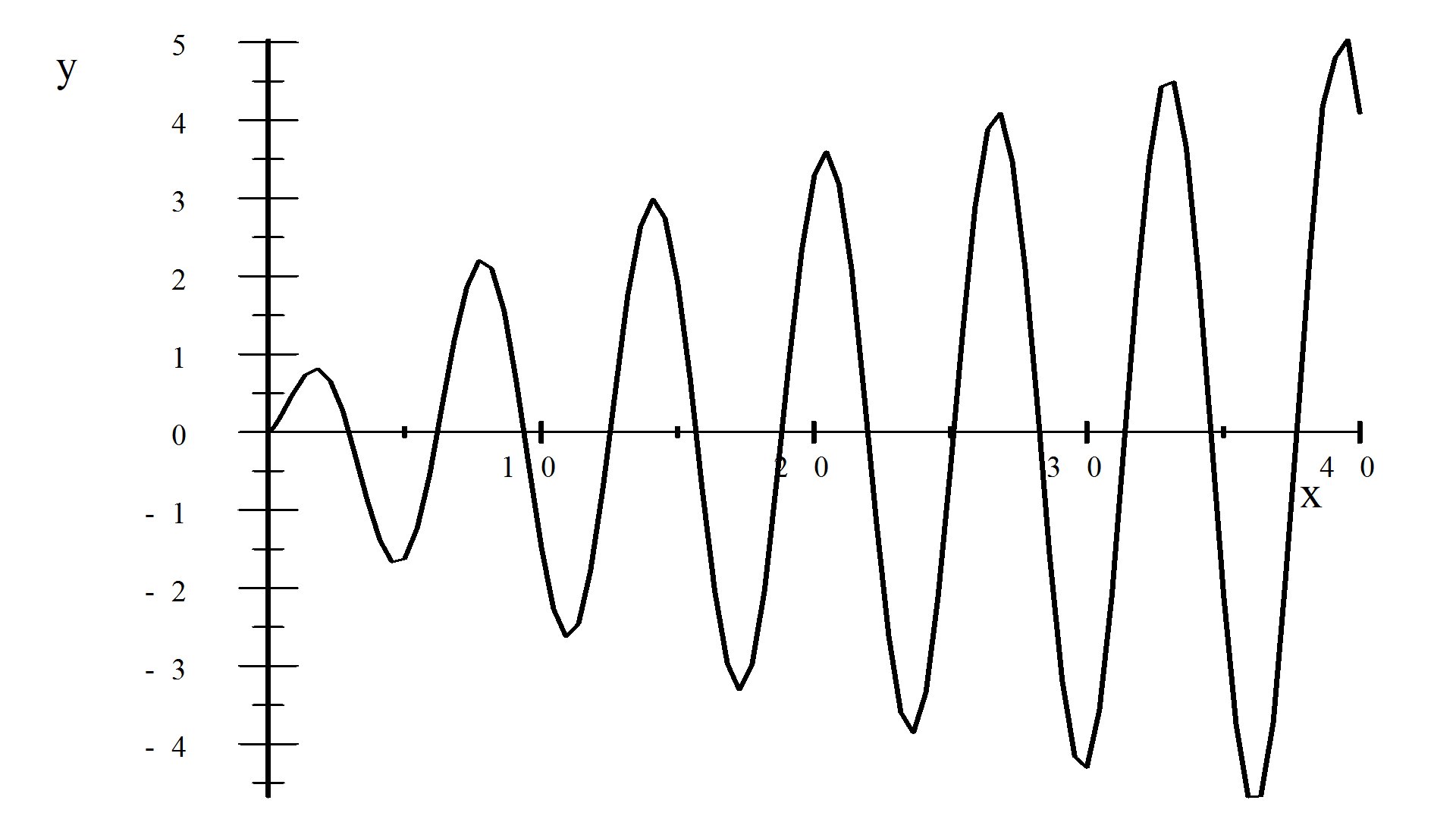}}
&
\fbox{\includegraphics[width=5.1796cm,height=7.6289cm]{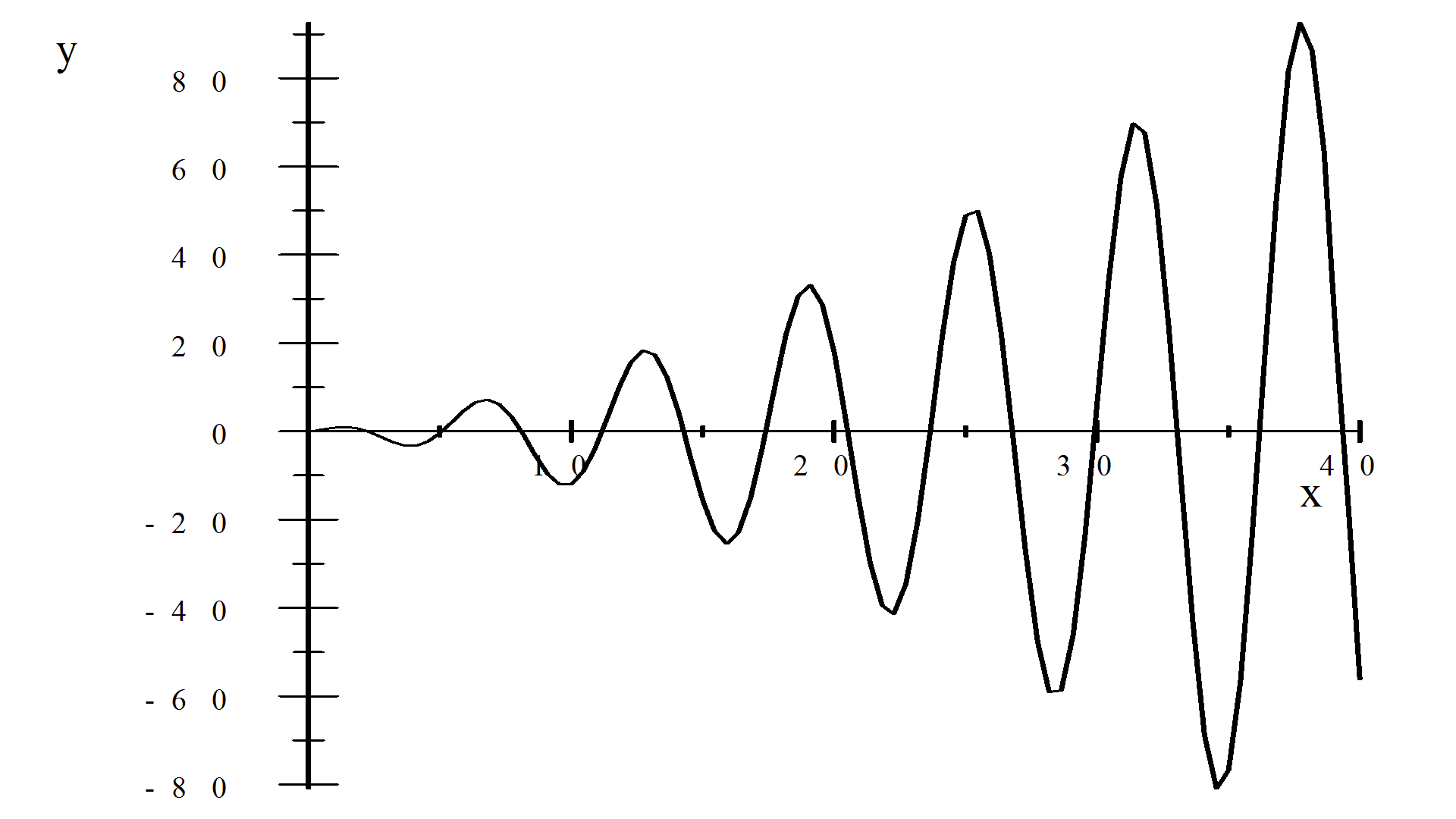}}
\\

FIGURE 7 Graph of $J_{\frac{3}{2}}(x)$ & FIGURE 8 Graph of $J_{\frac{3}{2}%
}^{(2)}(x)$ & FIGURE 9 Graph of $J_{\frac{3}{2}}^{(3)}(x)$%
\end{tabular}
\end{center}

\subsection{\protect\bigskip Comparison of Graphs of the Generalized Two
Parameter Spherical Bessel Function$~g_{\protect\nu}^{(c)}(x)$~with $%
j_{0}(x),~j_{\frac{1}{2}}(x)$ and $j_{\frac{3}{2}}(x)$}

\begin{center}
\begin{tabular}{lll}

\fbox{\includegraphics[width=5.05cm,height=7.5981cm]{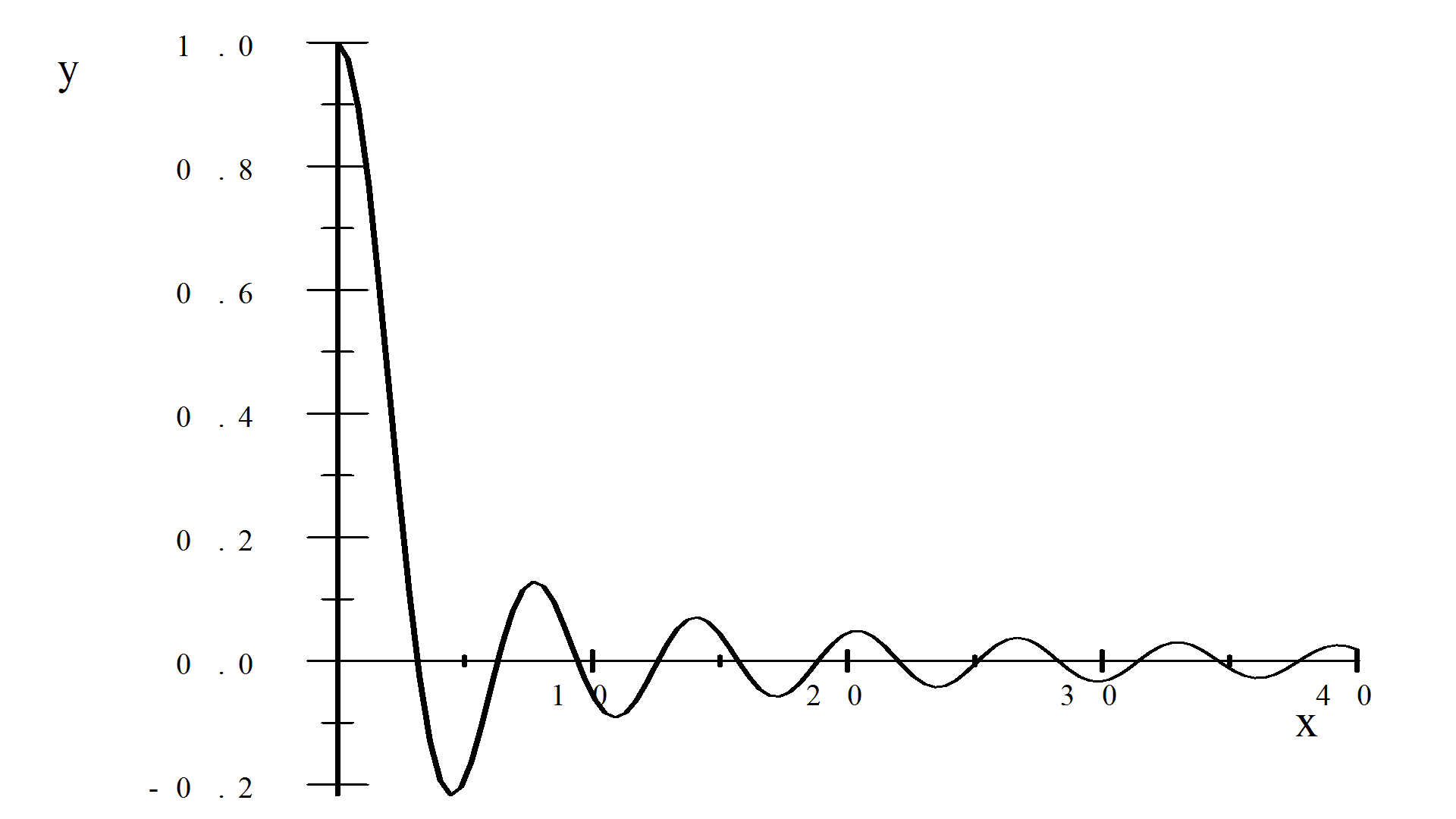}}
&
\fbox{\includegraphics[width= 5.1796cm,height=7.6289cm]{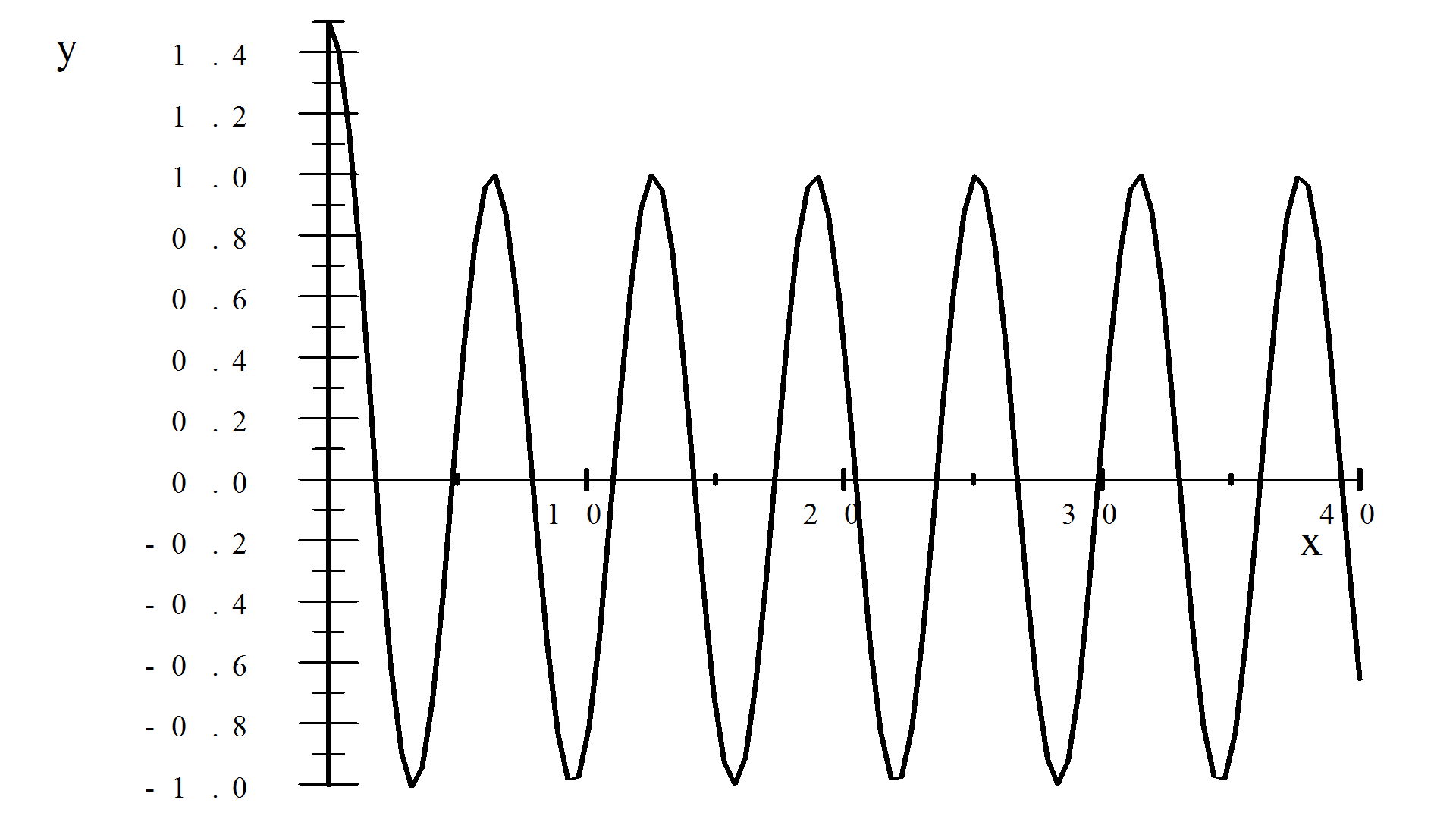}}
&
\fbox{\includegraphics[width=5.1796cm,height=7.6289cm]{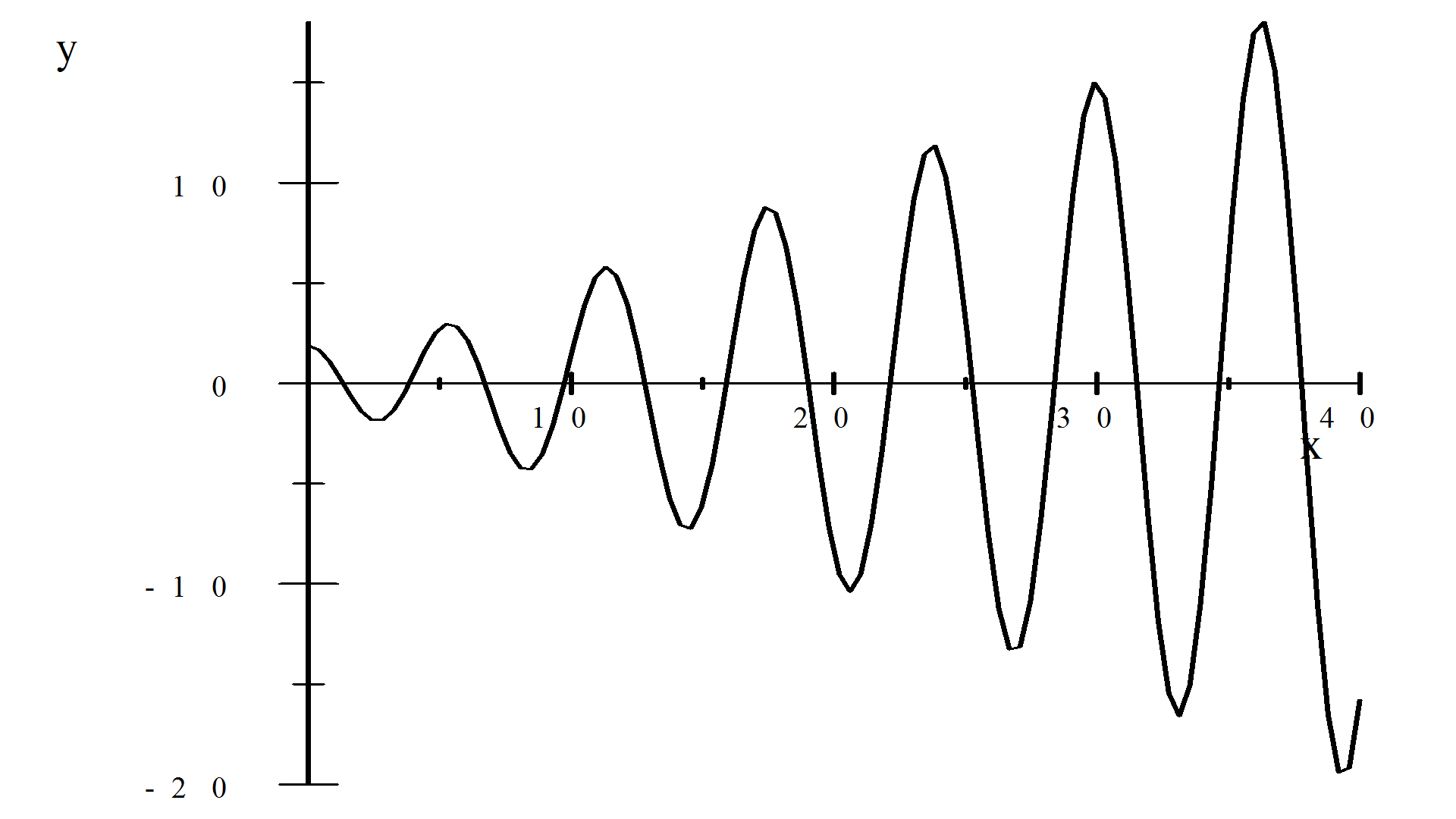}}
\\

FIGURE 10 Graph of~ $j_{0}(x)$ & FIGURE 11 Graph of $g_{0}^{(\frac{5}{2})}(x)
$ & FIGURE 12 Graph of $g_{0}^{(\frac{7}{2})}(x)$%
\end{tabular}

\bigskip

\begin{tabular}{lll}

\fbox{\includegraphics[width=5.05cm,height=7.5981cm]{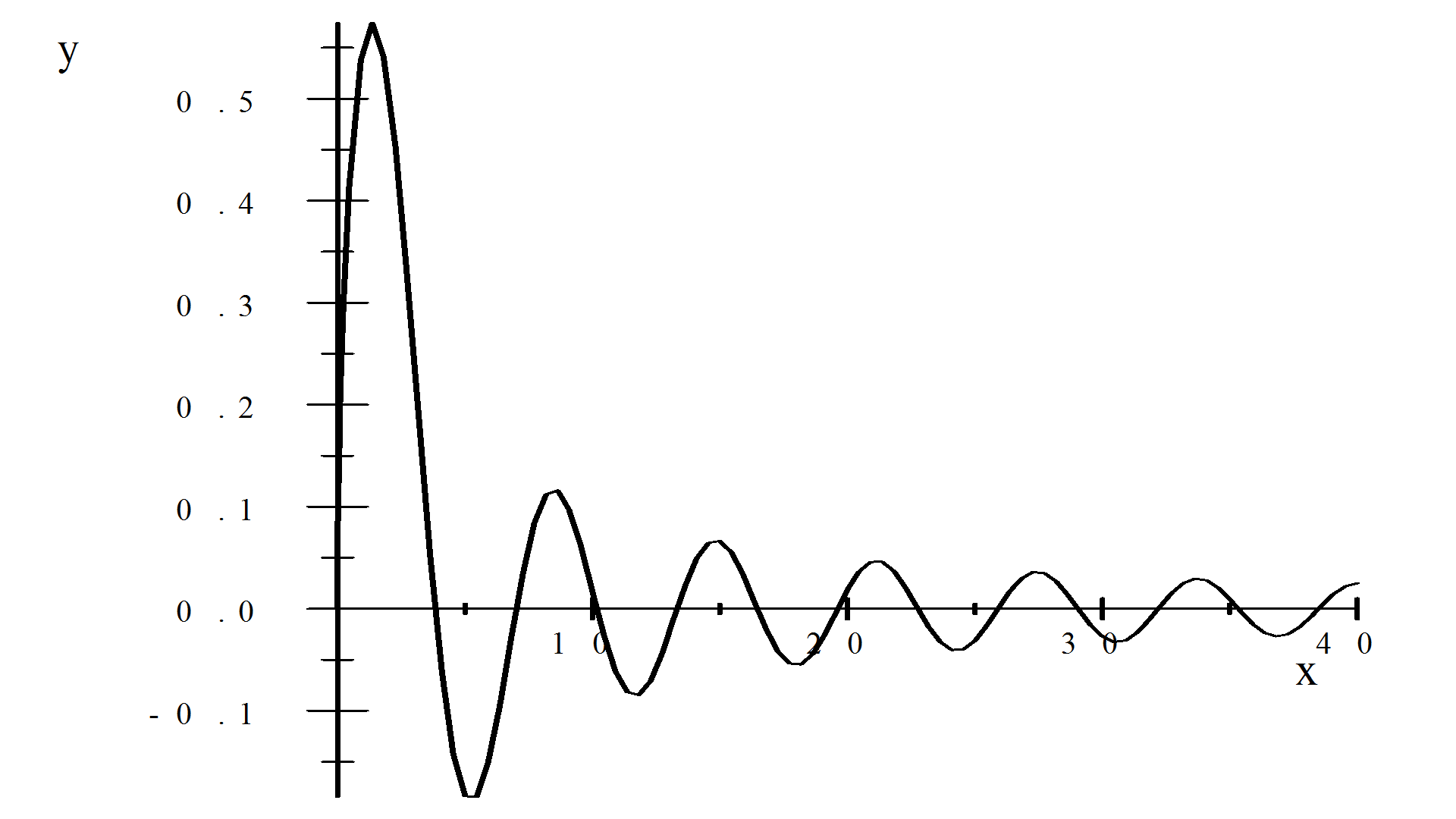}}
&
\fbox{\includegraphics[width= 5.1796cm,height=7.6289cm]{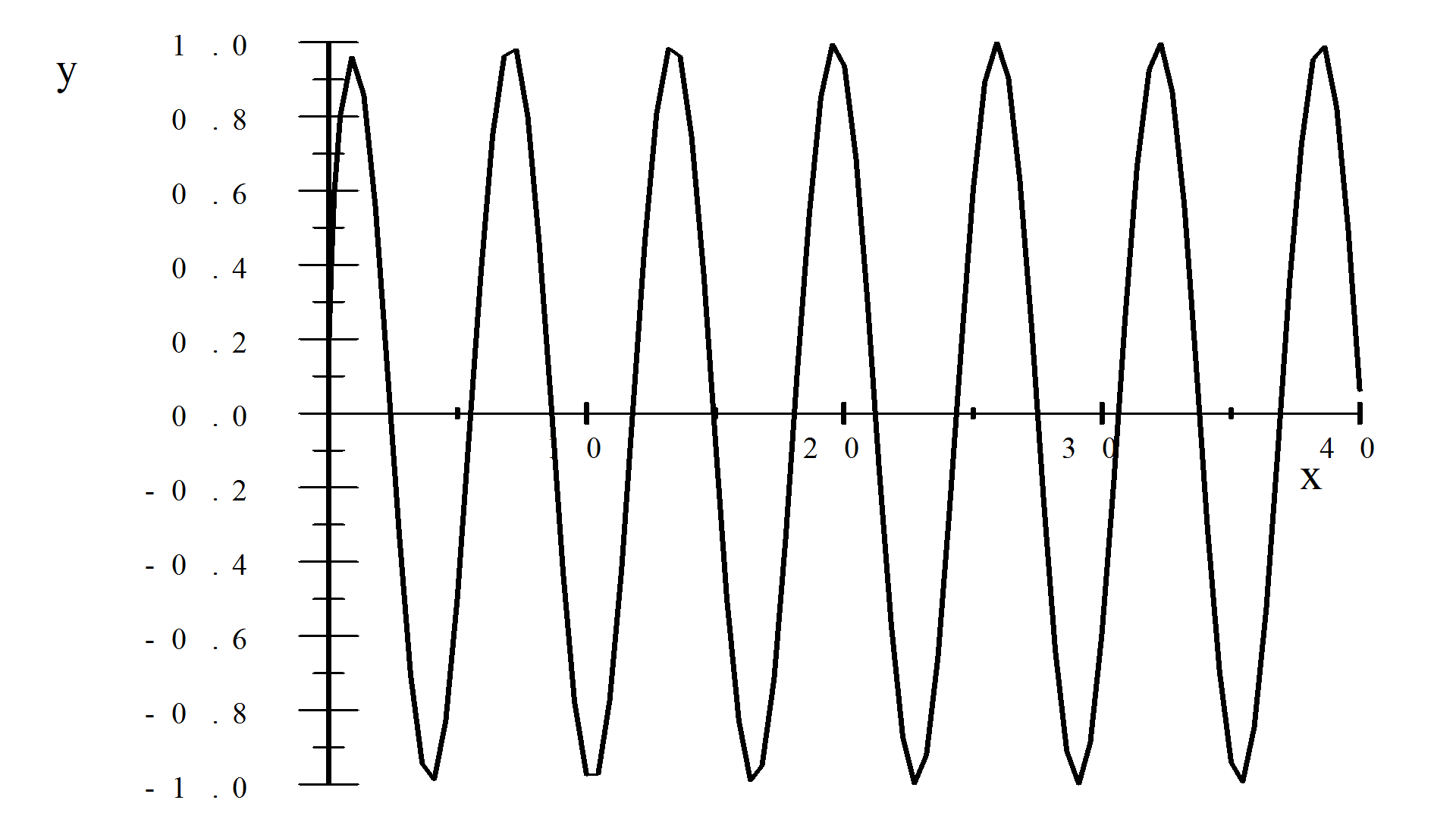}}
&
\fbox{\includegraphics[width=5.1796cm,height=7.6289cm]{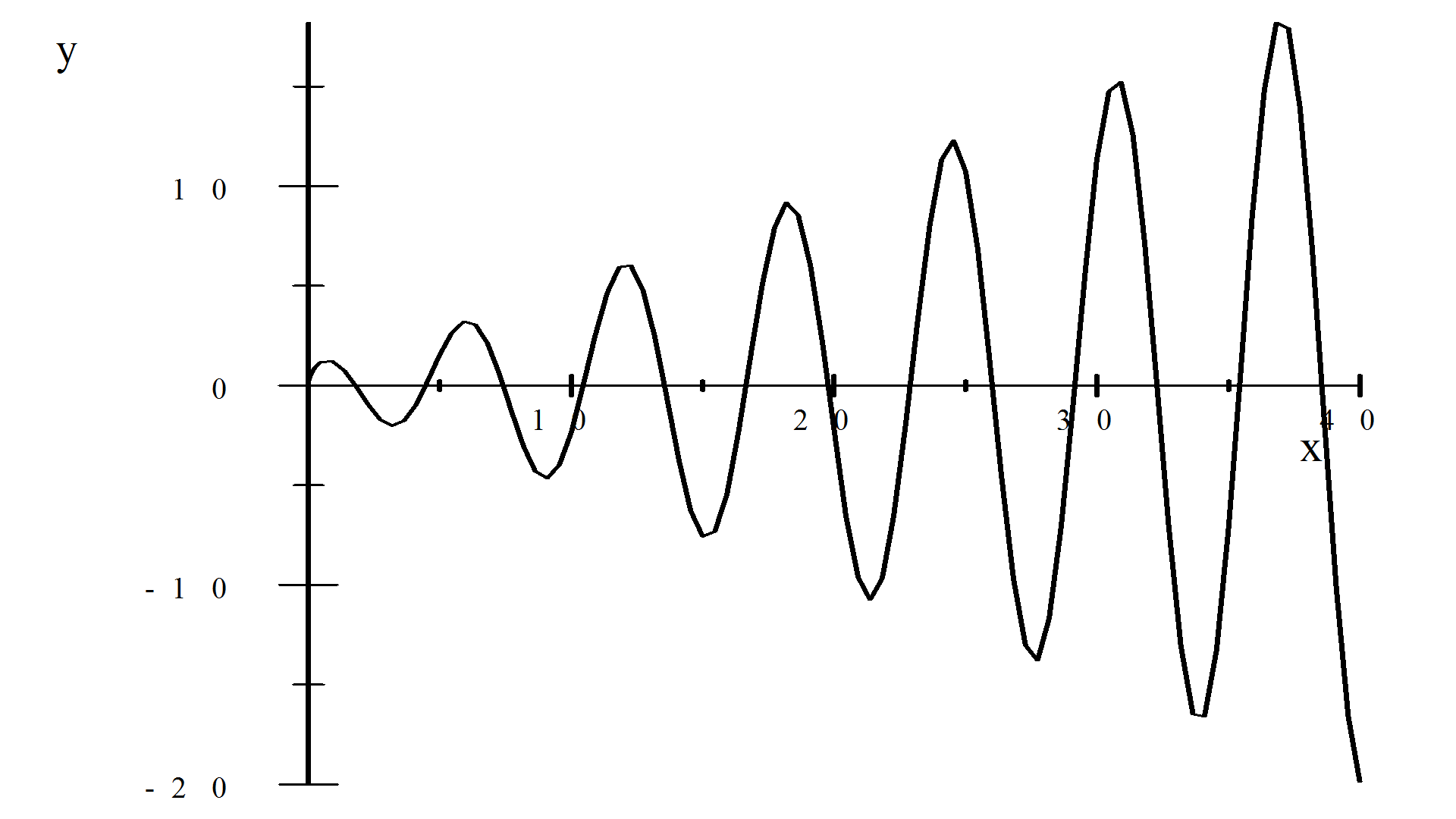}}
\\

FIGURE 13 Graph of $~j_{\frac{1}{2}}(x)$ & FIGURE 14 Graph of $g_{\frac{1}{2}%
}^{(\frac{5}{2})}(x)$ & FIGURE 15 Graph of $g_{\frac{1}{2}}^{(\frac{7}{2}%
)}(x)$%
\end{tabular}

\bigskip

\begin{tabular}{lll}

\fbox{\includegraphics[width=5.05cm,height=7.5981cm]{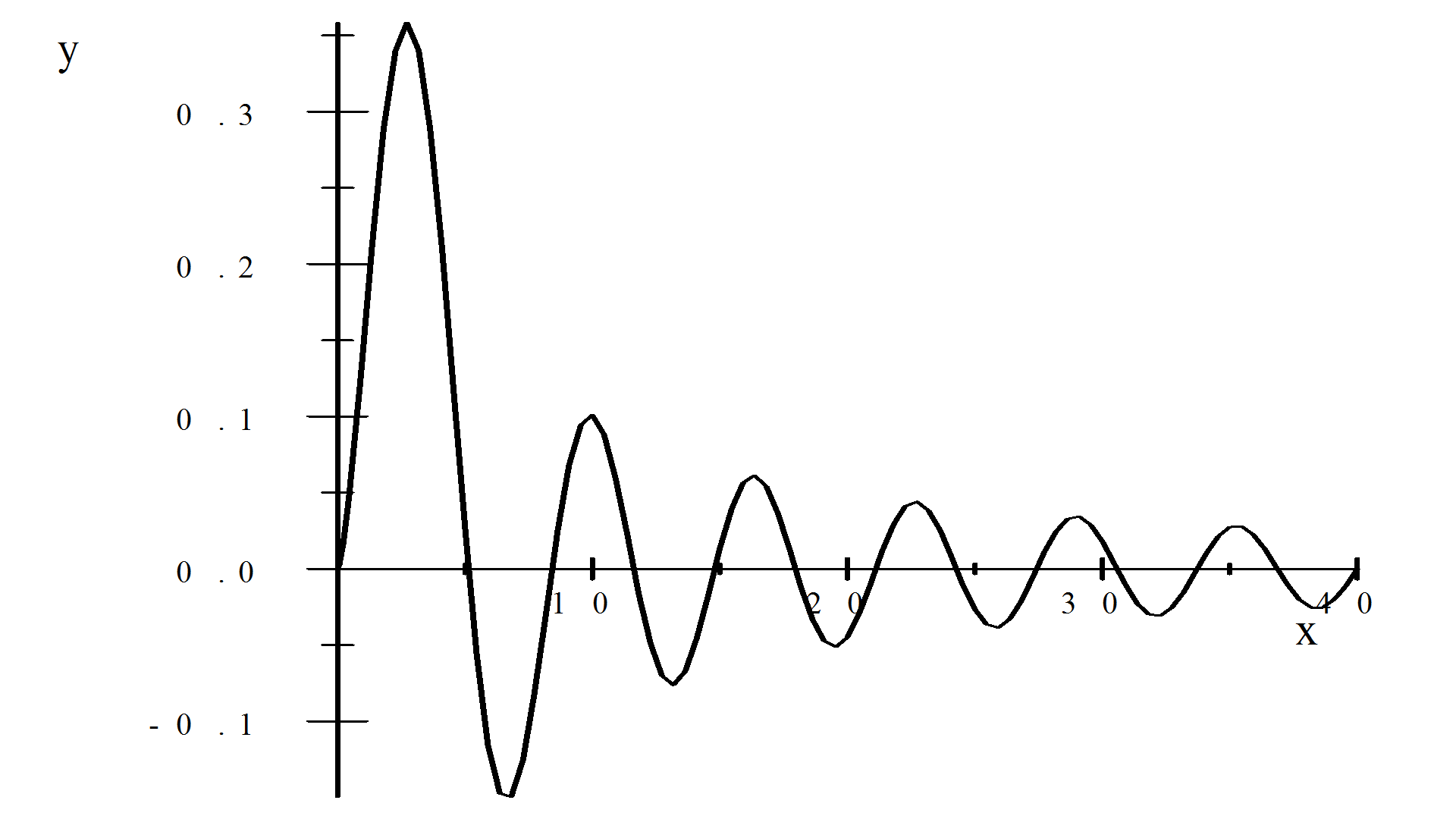}}
&
\fbox{\includegraphics[width= 5.1796cm,height=7.6289cm]{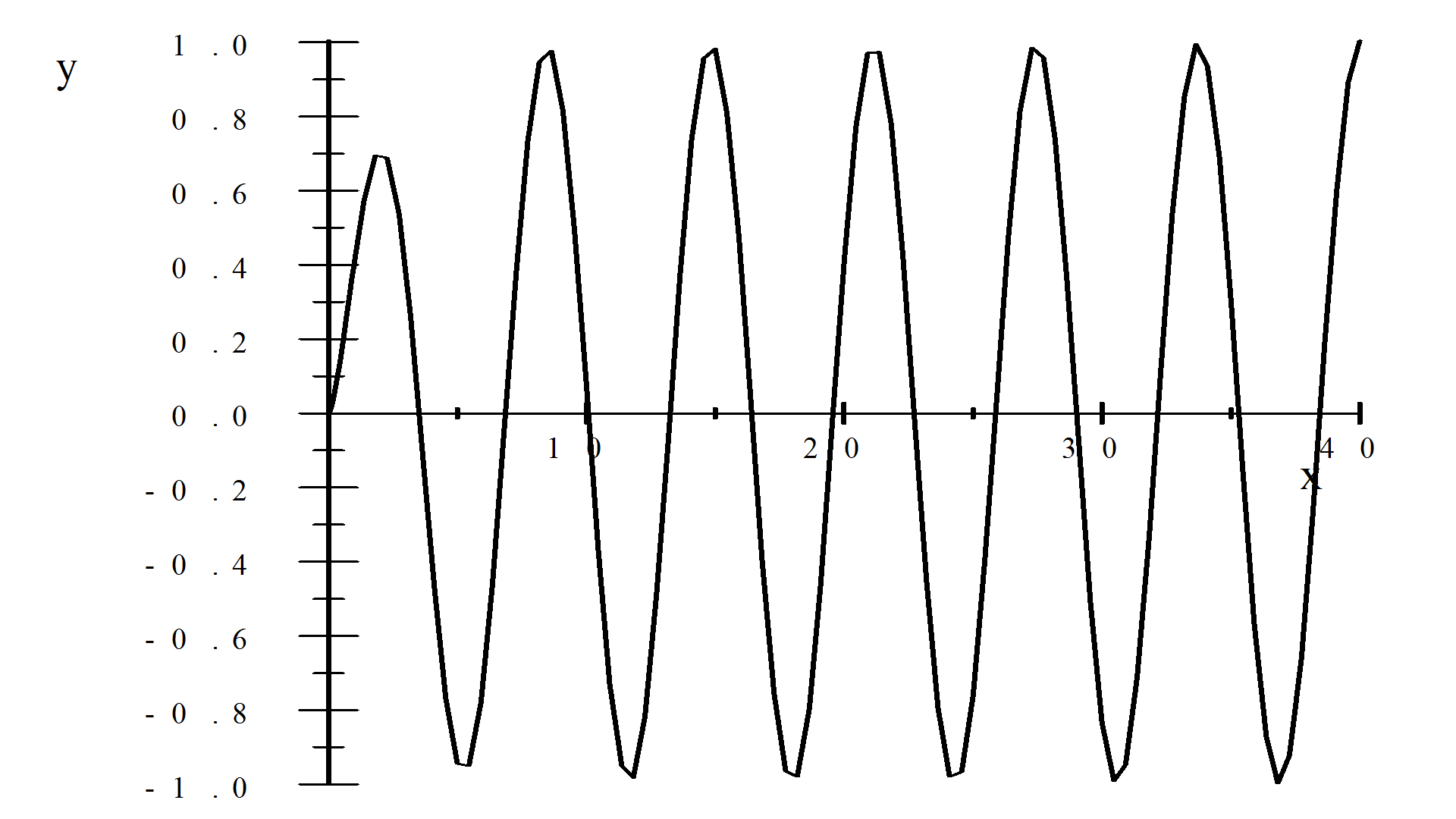}}
&
\fbox{\includegraphics[width=5.1796cm,height=7.6289cm]{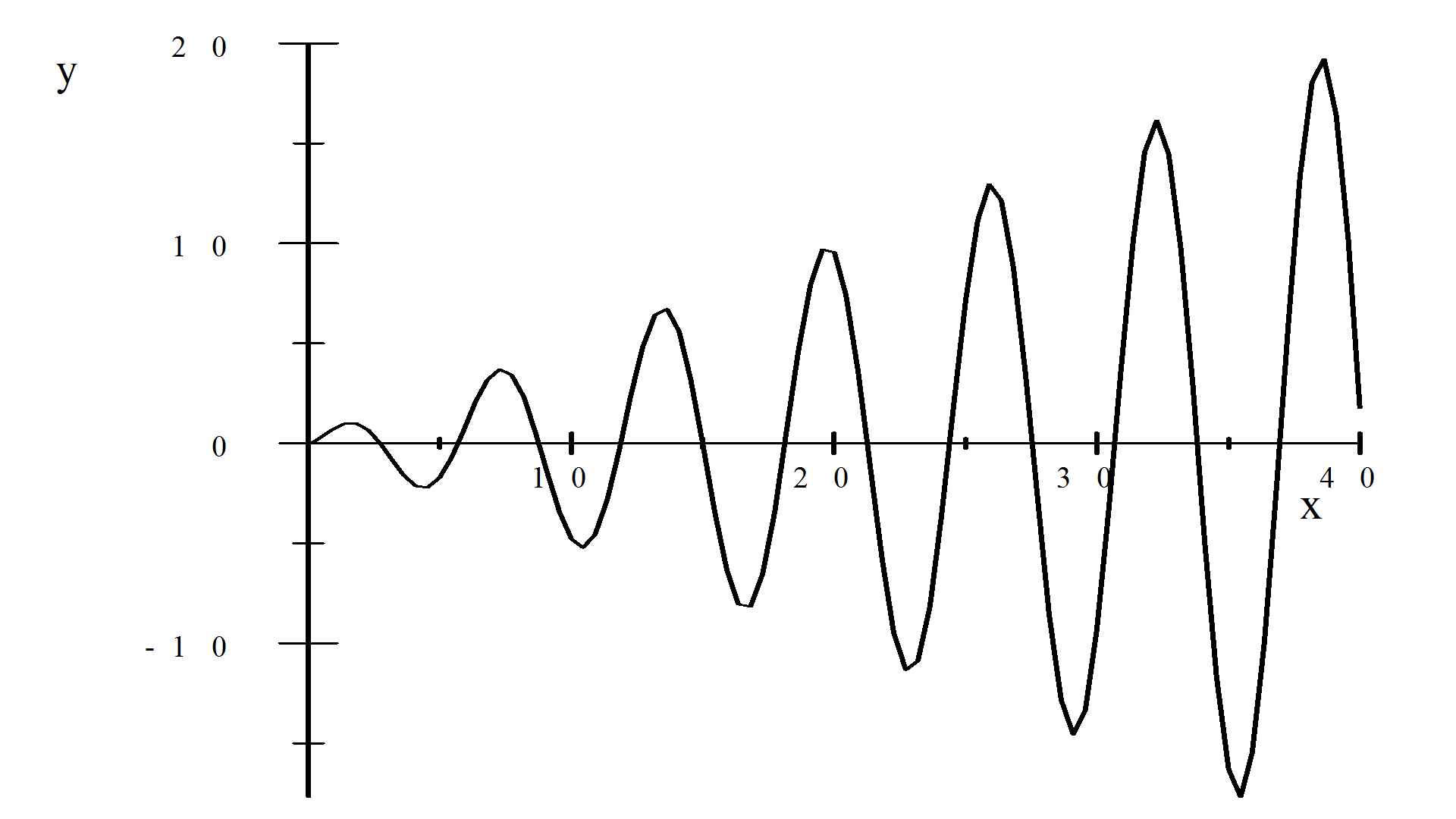}}
\\

FIGURE 16 Graph of $~j_{\frac{3}{2}}(x)$ & FIGURE 17 Graph of $g_{\frac{3}{2}%
}^{(\frac{5}{2})}(x)$ & FIGURE 18 Graph of $g_{\frac{3}{2}}^{(\frac{7}{2}%
)}(x)$%
\end{tabular}

\bigskip

\bigskip \bigskip
\end{center}

\subsection{Application Fields of the Generalized Two Parameter Modified
Bessel Function to Digital Signal Processing}

Digital signal processing affords greater flexibility, higher performance
(in terms of attenuation and selectivity), better time and environment
stability and lower equipment production costs than traditional analog
techniques. A digital filter is simply a discrete-time, discrete-amplitude
convolver. Digital filters are commonly used for audio frequencies for two
reasons. First, digital filters for audio are superior in price and
performance to the analog alternative. Second, audio analog to digital
converters and digital to analog converters can be manufactured with high
accuracy and are available at low cost. To design a digital filter, there
are some known methods such as Fourier series method, frequency sampling
method, window method. Besides, there are other various design methods to
design a filter. For instance, Neural Network \cite{D.A,H.B}, Genetic
Algorithm \cite{B.A}, Particle Swarm Optimization \cite{B.L.G,L.Z.H},
Discrete Cosine Transform \cite{C.T.L}, Chebychev Criterion \cite{G.K.R},
etc. Among of them, Kaiser Window \cite{J.S} has an advantage of trading-off
the transition width against the ripple that is defined by 
\begin{equation*}
w_{n}~=\{%
\begin{tabular}{l}
$\frac{I_{0}(\pi \alpha \sqrt{1-(\frac{2n}{N-1}-1)^{2}})}{I_{0}(\pi \alpha )}%
,~0\leq n\leq N-1$ \\ 
$0~\ \ ~~~~~~\ \ \ ~~~~\ \ ~~\ \ ~\,,otherwise$%
\end{tabular}%
\ \ \ 
\end{equation*}%
where \ $N$ is the length of the sequence,~$I_{0}$ is the zeroth order
modified Bessel function of the first kind and $\alpha $ is an arbitrary
non-negative real number that determines the shape of the window.

The generalized Kaiser-window function is defined by 
\begin{equation*}
w_{n}^{(c)}=\{%
\begin{tabular}{l}
$\frac{I_{0}^{(c)}(\pi \alpha \sqrt{1-(\frac{2n}{N-1}-1)^{2}})}{%
I_{0}^{(c)}(\pi \alpha )},~0\leq n\leq N-1$ \\ 
$0~~\ \ \ \ \ \ \ \ \ \ \ \ \ \ \ \ \ \ \ \ \ \ \ ,~otherwise$%
\end{tabular}%
\ \ \ 
\end{equation*}%
where $\func{Re}(c)>0.$ The case $c=1,$ it is reduced to usual Kaiser-window
function $w_{n}.$ Some properties of Kaiser-window function can be applied
to generalized Kaiser-window function, as well.$~~$

$~~~\ \ \ \ ~~\ \ ~\ $

\bigskip

\end{document}